\documentclass[11 pt]{amsart}
\usepackage{latexsym,amscd,amssymb, graphicx, amsthm, mathtools, young}  
\usepackage{blkarray}
\usepackage{tikz-cd, xcolor}

\usepackage[margin=1in]{geometry}

\numberwithin{equation}{section}

\newtheorem{theorem}{Theorem}[section]
\newtheorem{proposition}[theorem]{Proposition}
\newtheorem{corollary}[theorem]{Corollary}
\newtheorem{lemma}[theorem]{Lemma}
\newtheorem{conjecture}[theorem]{Conjecture}

\newtheorem{question}[theorem]{Question}

\newtheorem{example}[theorem]{Example}
\newtheorem{remark}[theorem]{Remark}

\definecolor{2purple}{RGB}{204,102,255}
\definecolor{3green}{RGB}{0,204,0}

\newtheorem{defn}[theorem]{Definition}
\theoremstyle{definition}

\newcommand{\cocharge}{{\mathrm {cocharge}}}

\newcommand{\maj}{{\mathrm {maj}}}
\newcommand{\inv}{{\mathrm {inv}}}
\newcommand{\minimaj}{{\mathrm {minimaj}}}
\newcommand{\sign}{{\mathrm {sign}}}
\newcommand{\area}{{\mathrm {area}}}
\newcommand{\dinv}{{\mathrm {dinv}}}

\newcommand{\GL}{{\mathrm {GL}}}
\newcommand{\Mat}{{\mathrm {Mat}}}
\newcommand{\st}{{\mathrm {st}}}
\newcommand{\Stir}{{\mathrm {Stir}}}
\newcommand{\Hilb}{{\mathrm {Hilb}}}
\newcommand{\Rise}{{\mathrm {Rise}}}
\newcommand{\Val}{{\mathrm {Val}}}
\newcommand{\grFrob}{{\mathrm {grFrob}}}
\newcommand{\des}{{\mathrm {des}}}
\newcommand{\sort}{{\mathrm {sort}}}
\newcommand{\coinv}{{\mathrm {coinv}}}
\newcommand{\rev}{{\mathrm {rev}}}

\newcommand{\Gr}{{\mathrm {Gr}}}

\newcommand{\Res}{{\mathrm {Res}}}

\newcommand{\SYT}{{\mathrm {SYT}}}

\newcommand{\Frob}{{\mathrm {Frob}}}

\newcommand{\ch}{{\mathrm {ch}}}
\newcommand{\conv}{{\mathrm {conv}}}

\newcommand{\symm}{{\mathfrak{S}}}

\newcommand{\code}{{\mathrm{code}}}

\newcommand{\CC}{{\mathbb {C}}}

\newcommand{\ZZ}{{\mathbb {Z}}}
\newcommand{\FF}{{\mathbb {F}}}

\newcommand{\PP}{{\mathbb{P}}}

\newcommand{\VV}{{\mathbb{V}}}
\newcommand{\OP}{{\mathcal{OP}}}

\newcommand{\RRR}{{\mathcal{R}}}
\newcommand{\VVV}{{\mathcal{V}}}
\newcommand{\WWW}{{\mathcal{W}}}

\newcommand{\UUU}{{\mathcal{U}}}

\newcommand{\AAA}{{\mathcal{A}}}

\newcommand{\BBB}{{\mathcal{B}}}

\newcommand{\EEE}{{\mathcal{E}}}
\newcommand{\FFF}{{\mathcal{F}}}
\newcommand{\LLL}{{\mathcal{L}}}
\newcommand{\SSS}{{\mathcal{S}}}
\newcommand{\GGG}{{\mathcal{G}}}

\newcommand{\yy}{{\mathbf {y}}}

\newcommand{\zz}{{\mathbf {z}}}
\newcommand{\xx}{{\mathbf {x}}}
\newcommand{\Fl}{{\mathrm {Fl}}}

\newcommand{\Park}{{\mathrm {Park}}}
\newcommand{\gr}{{\mathrm {gr}}}

\newcommand{\II}{{\mathbf {I}}}
\newcommand{\ii}{{\mathbf {i}}}
\newcommand{\jj}{{\mathbf {j}}}

\newcommand{\PM}{{\textsc {PM}}}

\newcommand{\neglex}{{\texttt {neglex}}}

\newcommand{\initial}{{\mathrm {in}}}

\begin{document}

\title[Generalizations of the flag variety and delta operators]
{Generalizations of the flag variety tied to the Macdonald-theoretic delta operators}

\author{Brendon Rhoades}
\address
{Department of Mathematics \newline \indent
University of California, San Diego \newline \indent
La Jolla, CA, 92093-0112, USA}
\email{bprhoades@math.ucsd.edu}

\maketitle

The {\em coinvariant ring} is the quotient $R_n = \CC[x_1, \dots, x_n]/\langle e_1, \dots, e_n \rangle$ 
of the polynomial ring $\CC[x_1, \dots, x_n]$
by the ideal generated by the $n$ elementary symmetric polynomials $e_1, \dots, e_n$.
Since the polynomials $e_1, \dots, e_n$ are symmetric, the quotient $R_n$ carries the structure of a graded
$S_n$-module and is one of the most important representations in algebraic combinatorics.
In geometric terms, Borel \cite{Borel} proved that the coinvariant ring 
presents the cohomology of the flag variety: we have
$H^{\bullet}(\Fl(n)) = R_n$. 

In the early 1990s, Garsia and Haiman \cite{GH, HaimanQuotient} initiated the study of the 
{\em diagonal coinvariant ring} $DR_n$ obtained by modding out the polynomial ring
$\CC[x_1, \dots, x_n, y_1, \dots, y_n]$ by the ideal generated by positive-degree invariants of the diagonal
action of $S_n$.
Setting the $y$-variables in $DR_n$ to zero recovers $R_n$, so that $DR_n$ is a bigraded extension of 
$R_n$.
A decade later, Haiman \cite{HaimanVanish} used deep algebraic geometry to express the bigraded
$S_n$-isomorphism type of $DR_n$ as the symmetric function $\nabla  e_n$,
where $\nabla$ is a remarkable symmetric function operator which has the Macdonald polynomials 
as its eigenbasis.

Thanks to Haiman's result, finding the bigraded $S_n$-isomorphism type of $DR_n$ amounts 
to finding the expansion of $\nabla e_n$ in the Schur basis $\{ s_{\lambda} \,:\, \lambda \vdash n \}$
of symmetric functions.
Although some coefficients in this expansion are known \cite{HaglundSchroder},
 this  remains an open problem.
Haglund, Haiman, Remmel, Loehr, and Ulyanov \cite{HHLRU} conjectured, and 
Carlsson and Mellit \cite{CM} proved, the {\em Shuffle Theorem}: a
formula for the monomial expansion of $\nabla e_n$.

The Shuffle Theorem is a purely combinatorial statement about the symmetric function $\nabla e_n$.
In the decade between its conjecture and its proof, various 
extensions and refinements were proposed \cite{BGLX,HMZ} in an attempt make $\nabla e_n$
vulnerable to inductive attack.
The last of these to appear before the Shuffle Theorem was proved was the {\em Delta Conjecture}
of Haglund, Remmel, and Wilson \cite{HRW}. This conjecture depends on two parameters $k \leq n$ and
predicts the monomial expansion
 $\Delta'_{e_{k-1}} e_n$ where the {\em delta operator} $\Delta'_{e_{k-1}}$ is another Macdonald eigenoperator.
 The Delta Conjecture reduces to the Shuffle Theorem when $k = n$.
 The `Rise Version' of the Delta Conjecture was proven recently and independently by D'Adderio-Mellit \cite{DM}
 and Blasiak-Haiman-Morse-Pun-Seelinger \cite{BHMPS},
 but the full conjecture remains open. 
 
 After the formulation of the Delta Conjecture, Haglund, Rhoades, and Shimozono \cite{HRS}
introduced a family $R_{n,k}$ of singly-graded quotients 
 of $\CC[x_1, \dots, x_n]$ whose isomorphism type is (essentially) the $t = 0$ specialization of the 
 symmetric function $\Delta'_{e_{k-1}} e_n$.
 The ring $R_{n,k}$ reduces to the classical coinvariant ring $R_n$ when $k = n$.
 Whereas the algebra of $R_n$ is governed by the combinatorics of permutations in $S_n$,
 the algebra of $R_{n,k}$ is governed by the combinatorics of ordered set partitions of 
 $[n] := \{1, \dots, n \}$ into $k$ blocks.
 Many algebraic properties of $R_n$ generalize naturally to $R_{n,k}$; these quotient rings 
 constitute a coinvariant algebra for the Delta Conjecture.
 
 Returning to geometry, Pawlowski and Rhoades \cite{PR} introduced a variety $X_{n,k}$
 of spanning line configurations which is homotopy equivalent to the flag variety $\Fl(n)$ when $k = n$.
 A point in $X_{n,k}$ is a tuple $(\ell_1, \dots, \ell_n)$ of 1-dimensional subspaces  $\ell_i \subseteq \CC^k$ such that 
 we have the vector space sum $\ell_1 + \cdots + \ell_n = \CC^k$.
 The ring $R_{n,k}$ presents the cohomology $H^{\bullet}(X_{n,k})$, and geometric properties
 of $X_{n,k}$ are governed by combinatorial properties of ordered set partitions 
 (recast as {\em Fubini words}).
 The variety $X_{n,k}$ is a flag variety for the Delta Conjecture.
 Griffin, Levinson, and Woo \cite{GLW} gave a different geometric model for the Delta Conjecture which simultaneously generalizes
 the theory of {\em Springer fibers}.
 
 In {\bf Section~\ref{Symmetric}} we provide the 
 necessary background on symmetric function theory,
 as well as an introduction to the combinatorics of the Delta Conjecture.
 In {\bf Section~\ref{Quotient}}
we introduce the generalized coinvariant rings $R_{n,k}$ and use orbit harmonics to study them.
 {\bf Section~\ref{Spanning}} is devoted to
 the variety $X_{n,k}$ of spanning line configurations; we 
 show that $R_{n,k}$ presents its cohomology.
 In {\bf Section~\ref{Generalized}} we cover remarkable work of Griffin \cite{Griffin}
 and Griffin-Levinson-Woo \cite{GLW} which unifies the delta coinvariant theory with
 the theory of Tanisaki quotients and Springer fibers.
 We close in {\bf Section~\ref{Future}} with future directions and conjectures related largely 
 to the  structure of the full symmetric function $\Delta'_{e_{k-1}} e_n$
 (rather than its $t = 0$ specialization).
 These are largely algebraic, with ties to rings of differential forms.
 Given the work of Haiman on Hilbert schemes and diagonal coinvariants, it is our hope
 that geometry will appear in the more general setting of $\Delta'_{e_{k-1}} e_n$, as well.

 \subsection{A note to the geometrically inclined reader}
 If one wants to learn about the varieties $X_{n,k}$ right away,
 it is advised to skip to Section~\ref{Spanning}
 (accepting certain combinatorial results in Section~\ref{Symmetric}
and algebraic results in Section~\ref{Quotient} as black boxes).
 On the other hand, the symmetric function theory of Section~\ref{Symmetric} was the primary motivation
 for defining the quotient rings $R_{n,k}$ of Section~\ref{Quotient}, which was in turn the primary motivation
 for finding a variety $X_{n,k}$ with $H^{\bullet}(X_{n,k}) = R_{n,k}$.
 So the combinatorially or algebraically inclined reader should not miss out on these!

 \subsection{A word on coefficients}
 We work over the complex field $\CC$ for convenience, but the geometric 
 results presented here
 can be shown to
 hold over the ring of integers $\ZZ$ with slightly more work.
 We describe the  form of the arguments used in passing from $\CC$ to $\ZZ$
 after presenting the  
 cohomology of $X_{n,k}$ (see Remark~\ref{coefficient-remark}).

\section{Symmetric functions and delta operators}
\label{Symmetric}

Much of the algebraic and geometric
material in this chapter involves the machinery of symmetric functions.
We quickly introduce the necessary background and refer the reader to  \cite{HaglundBook} for a comprehensive treatment.

\subsection{Symmetric functions and $S_n$-modules}
Let $\xx = (x_1, x_2, \dots )$ be an infinite list of variables.
For a partition $\lambda = (\lambda_1 \geq  \dots \geq \lambda_k \geq 0)$,
the {\em monomial symmetric function} $m_{\lambda} = m_{\lambda}(\xx)$ is the formal power series
\begin{equation}
m_{\lambda} := \sum_{1 \leq i_1 < \cdots < i_k} \sum_{(a_1, \dots, a_k)} x_{i_1}^{a_1} \cdots x_{i_k}^{a_k}
\end{equation}
where the internal sum is over all rearrangements $(a_1, \dots, a_k)$ of the parts of $\lambda$.
The {\em ring of symmetric functions} over a commutative ring $R$ is the $R$-submodule $\Lambda$ 
of the formal power series ring $R[[\xx]] = R[[x_1, x_2, \dots ]]$ given by 
$\Lambda := \mathrm{span}_R \{ m_{\lambda} \,:\, \lambda \text{ a partition} \}$.
It is easily seen that $\Lambda$ is closed under multiplication, and so forms a subring 
of $R[[\xx]]$.
We take our ground ring to
be $R = \CC(q,t)$ in this chapter, where $q$ and $t$ are variables.

The ring $\Lambda = \bigoplus_{n \geq 0} \Lambda_n$ is graded, with the degree $n$ piece
$\Lambda_n$ having basis $\{ m_{\lambda} \,:\, \lambda \vdash n \}$.
There are many other interesting bases of $\Lambda_n$ indexed naturally by partitions $\lambda \vdash n$;
we introduce those which will be of interest to us.
For $n \geq 0$, the {\em elementary, homogeneous,} and {\em power sum symmetric functions} are
given by
\begin{equation}
e_n := \sum_{i_1 < \cdots < i_n} x_{i_1} \cdots x_{i_n} \quad \quad
h_n := \sum_{i_1 \leq \cdots \leq i_n} x_{i_1} \cdots x_{i_n} \quad \quad 
p_n := \sum_{i \geq 1} x_i^n
\end{equation}
For a partition $\lambda = (\lambda_1, \lambda_2, \dots )$, 
we extend this  definition by setting 
\begin{equation}
e_{\lambda} := e_{\lambda_1} e_{\lambda_2} \cdots \quad \quad
h_{\lambda} := h_{\lambda_1} h_{\lambda_2} \cdots \quad \quad
p_{\lambda} := p_{\lambda_1} p_{\lambda_2} \cdots 
\end{equation}
Each of the sets $\{e_{\lambda} \,:\, \lambda \vdash n \}, \{ h_{\lambda} \,:\, \lambda \vdash n \},$
and $\{ p_{\lambda} \,:\, \lambda \vdash n \}$ are bases of $\Lambda_n$.

The power sum basis gives rise to an operation on formal power series called {\em plethysm}.
Let $k \geq 1$ and let $E = E(t_1, t_2, \dots )$ be a rational expression which depends on variables
$t_1, t_2, \dots $.  The {\em plethystic substitution} $p_k[E]$ of $E$ into $p_k$ is 
the rational function
\begin{equation}
p_k[E] := E(t_1^k, t_2^k, \dots )
\end{equation}
obtained by replacing each $t_i$ in $E$ with $t_i^k$.  More generally, if $F$ is any symmetric 
function, we define $F[E]$ by the rules
\begin{equation}
(F_1 + F_2)[E] := F_1[E] + F_2[E] \quad  \quad 
(F_1 \cdot F_2)[E] := F_1[E] \cdot F_2[E] \quad \quad
c[E] := c
\end{equation}
for any symmetric functions $F_1, F_2 \in \Lambda$ and all constants $c$.
Since $\{p_1, p_2, \dots \}$ is an algebraically independent generating set of $\Lambda$, the expression
$F[E]$ is 
well-defined.

For a partition 
$\lambda \vdash n$, a {\em semistandard Young tableau} of shape $\lambda$ is a filling 
$T$ of the Young diagram of $\lambda$ with positive integers which increase weakly across rows
and strictly down columns. 
For example, the tableau
\begin{equation*}
\begin{Young}
1 & 1 & 3 \\
2 & 4 & 4 \\
4
\end{Young}
\end{equation*}
is semistandard of shape $(3,3,1)$.
The {\em Schur function} $s_{\lambda}$ is given by
\begin{equation}
s_{\lambda} := \sum_T \xx^T
\end{equation}
where the sum is over all semistandard Young tableaux $T$ of shape $\lambda$ and $\xx^T$ denotes
the monomial $x_1^{c_1} x_2^{c_2} \cdots $ where $c_i$ is the number of $i$'s in $T$; the sequence
$\mu = (c_1, c_2, \dots )$ is the {\em content} of $T$.
Our example  tableau has content $(2,1,1,3)$ and contributes $x_1^2 x_2 x_3 x_4^3$ to the 
Schur function $s_{331}$.

The set $\{ s_{\lambda} \,:\, \lambda \vdash n \}$ forms the 
{\em Schur basis} of $\Lambda_n$. The {\em Hall inner product} $\langle -, - \rangle$
on $\Lambda$  is obtained by declaring the Schur basis to be orthonormal, i.e.
\begin{equation}
\langle s_{\lambda}, s_{\mu} \rangle := \delta_{\lambda,\mu}.
\end{equation}
The {\em omega involution} is the isometry $\omega$ of this inner product given by 
$\omega(s_{\lambda}) = s_{\lambda'}$, where $\lambda'$ is the partition conjugate to $\lambda$.
By construction, we have $\omega \circ \omega = \mathrm{id}$; it can also be shown that $\omega$
is a ring map $\Lambda \rightarrow \Lambda$.

The combinatorics of symmetric functions is closely related to the representation theory of the symmetric 
group.  Irreducible representations of $S_n$ over $\CC$  
are in one-to-one correspondence
with partitions of $n$.  If $\lambda \vdash n$ is a partition, write $S^{\lambda}$ for the corresponding 
$S_n$-irreducible. For any finite-dimensional representation $V$ of $S_n$, there are unique multiplicities
$c_{\lambda}$ such that $V = \bigoplus_{\lambda \vdash n} c_{\lambda} S^{\lambda}$.
The {\em Frobenius image} of $V$ is the symmetric function
\begin{equation}
\Frob(V) := \sum_{\lambda \vdash n} c_{\lambda} \cdot s_{\lambda} \in \Lambda_n
\end{equation}
obtained by replacing each irreducible $S^{\lambda}$ with the corresponding Schur function $s_{\lambda}$.

We will consider  representations of $S_n$ equipped with both single and multiple gradings.
If $V = \bigoplus_{i \geq 0} V_i$ 
is a graded $S_n$-module, the {\em graded Frobenius image} is 
\begin{equation}
\grFrob(V; q) := \sum_{i \geq 0} \Frob(V_i) \cdot q^i.
\end{equation}
Generalizing further, if $V = \bigoplus_{i,j \geq 0} V_{i,j}$ is a bigraded $S_n$-module, the 
{\em bigraded Frobenius image} is 
\begin{equation}
\grFrob(V; q,t) := \sum_{i,j \geq 0} \Frob(V_{i,j}) \cdot q^i t^j.
\end{equation}

The Frobenius map provides a dictionary between the representation theory of $S_n$ and the combinatorics
of $\Lambda$.  For example, if $V$ is an $S_n$-module and $W$ is an $S_m$-module, we have
\begin{equation}
\Frob(V \circ W) = \Frob(V) \cdot \Frob(W)
\end{equation}
where $V \circ W := \mathrm{Ind}_{S_n \times S_m}^{S_{n+m}}(V \otimes W)$ is the induction product.
The omega involution on symmetric functions corresponds to a sign twist, viz.
\begin{equation}
\Frob ( \sign \otimes V) = \omega ( \Frob (V))
\end{equation}
where $\sign$ is the 1-dimensional sign representation of $S_n$.
A highly desirable way to show that a given symmetric function $F \in \Lambda_n$ is 
{\em Schur-positive} (i.e. has Schur expansion with nonnegative integer coefficients)
is to find an $S_n$-module $V$ with $\Frob(V) = F$; this will be a recurring theme in this chapter.

\subsection{Hall-Littlewood and Macdonald polynomials}
The symmetric functions considered in the previous subsection are rather specialized members 
of a rich hierarchy of symmetric functions involving the auxiliary parameters $q$ and $t$. 
Two families in this hierarchy will be important for us: the Hall-Littlewood and Macdonald polynomials.

The {\em Hall-Littlewood basis}
$\widetilde{H}_{\mu}(\xx;q)$ of $\Lambda$
depends on the parameter $q$.
We define these polynomials 
combinatorially  using the cocharge statistic of Lascoux and Sch\"utzenberger
 \cite{LS}.
We first describe cocharge on the level of words, and then extend to the case of tableaux.

Given any entry $w_j$ of a word $w$ and a positive integer $k$ which appears in $w$, write 
$\mathrm{cprev}(k, w_j)$ for the cyclically previous $k$ before $w_j$ in $w$.
More precisely, 
$\mathrm{cprev}(k, w_j)$ is the rightmost $k$ appearing to the left of $w_j$ (if such a $k$ exists),
and the rightmost $k$ in $w$ otherwise.
If $k$ does not appear in $w$, we leave $\mathrm{cprev}(k, w_j)$ undefined.

Let $\mu \vdash n$ be a partition with $k$ parts and let $w = w_1 \dots w_n$ be a word with content $\mu$.
That is, the word $w$ has $\mu_j$ copies of the letter $j$ for each $j$.
We decompose $w$ as a disjoint union  of permutations  $w^{(1)}, w^{(2)}, \dots $ as follows.
Let $w_{i_1} = 1$ be the rightmost $1$ in $w$ and recursively define $i_2, i_3, \dots, i_k$ by
\begin{equation*}
w_{i_{j+1}} = \mathrm{cprev}(j+1, w_{i_j})
\end{equation*}
for $j = 1, 2, \dots, k-1$. 
The {\em first standard subword} $w^{(1)}$ of $w$ is the subword of $w$ consisting of the entries 
$w_{i_1}, w_{i_2}, \dots, w_{i_k}$.
The word $w^{(1)}$ is well-defined because $w$ has partition content.
Extending this idea, the {\em standard subword decomposition} $w^{(1)}, w^{(2)}, w^{(3)} \dots $ of $w$
is defined by letting $w^{(i)}$ be the first standard subword of the word obtained from $w$ by 
erasing the letters in $w^{(1)}, w^{(2)}, \dots, w^{(i-1)}$.

An example should clarify these notions. Let $w = 4224511133$, a word of partition content. 
We bold the rightmost 1 in $w$, then bold the cyclically previous 2, then bold the cyclically previous 3, and so on
resulting in
$4 2 {\bf 2} {\bf 4} {\bf 5} 1 1 {\bf 1} 3 {\bf 3}$.
The permutation $w^{(1)} = [2, 4, 5, 1, 3]$ formed by the bolded letters is the first standard subword of $w$.
Erasing these bolded letters yields the smaller word $42113$. Repeating this procedure yields
${\bf 4} {\bf 2} 1 {\bf 1} {\bf 3}$ so that the second standard subword is $w^{(2)} = [4, 2, 1, 3]$.
Erasing the bolded letters once more yields the one-letter word $1$, so that $w^{(3)} = [1]$.

Let $w$ be a word with partition content. The {\em cocharge} $\cocharge(w)$ is defined as follows.
If $v \in S_n$ is a permutation of size $n$, the {\em cocharge sequence} $(cc_1, \dots, cc_n)$ of $v$
is defined by the initial condition $c_1 := 0$ and the recursion
\begin{equation}
cc_{i+1} := \begin{cases}
cc_i & \text{if $v^{-1}(i+1) > v^{-1}(i)$}  \\
cc_i + 1 & \text{if $v^{-1}(i+1) < v^{-1}(i)$}
\end{cases}
\end{equation}
and the cocharge of $v$ is the sum $\cocharge(v) := cc_1 + \cdots cc_n$ of its cocharge sequence.
We extend the definition of cocharge to the partition-content word $w$ by setting
\begin{equation}
\cocharge(w) := \cocharge(w^{(1)}) + \cocharge(w^{(2)}) + \cdots
\end{equation}
where $w^{(1)}, w^{(2)}, \dots $ is the standard subword decomposition of $w$.
For the word $w$ in the previous paragraph, we have
\begin{equation*}
\cocharge(4224511133) = 
\cocharge [2,4,5,1,3] + \cocharge [4,2,1,3] + \cocharge[1] = 6 + 4 + 0 = 10.
\end{equation*}

We are finally ready to define the polynomials $\widetilde{H}_{\mu}(\xx;q)$.
The {\em reading word} $\mathrm{read}(T)$ of a semistandard tableau $T$ is obtained by reading its entries 
from left to right within rows, proceeding from bottom to top. 
If $\mu \vdash n$ is a partition, the {\em (modified) Hall-Littlewood polynomial} $\widetilde{H}_{\mu}(\xx;q)$ has Schur
expansion
\begin{equation}
\widetilde{H}_{\mu}(\xx;q) := \sum_T q^{\cocharge(\mathrm{read}(T))} \cdot s_{\mathrm{shape}(T)}
\end{equation}
where the sum is over all semistandard tableaux $T$ of content $\mu$.  
For example, if $\mu = (3,2,2,2,1)$, the semistandard tableau
\begin{equation*}
\begin{Young}
1 & 1 & 1 & 3 & 3 \cr
2 & 2 & 4 & 5 \cr
4
\end{Young}
\end{equation*}
has content $\mu$ and reading word $4224511133$. Our earlier computations show that this tableau contributes
$q^{10} \cdot s_{541}$ to $\widetilde{H}_{\mu}(\xx;q)$.

Our final and most involved basis is that of {\em (modified) Macdonald polynomials}, 
depending on two parameters $q$ and $t$.
These polynomials may be characterized plethystically as follows.

Let $\geq$ be the {\em dominance order} on partitions given by 
\begin{center}
$\lambda \geq \mu$ if and only if 
$\lambda_1 + \lambda_2 + \cdots + \lambda_i \geq \mu_1 + \mu_2+ \cdots+ \mu_i$
for all $i$. 
\end{center}
Macdonald  \cite{MacdonaldSymmetric} showed that there is
a unique family $\{ \widetilde{H}_{\mu}(\xx;q,t) \,:\, \mu \vdash n \}$ of symmetric functions 
whose Schur expansions satisfy the {\em triangularity axioms}
\begin{equation}
\begin{cases}
\widetilde{H}_{\mu}[\xx(1-q);q,t] = \sum_{\lambda \geq \mu} a_{\lambda \mu}(q,t) \cdot s_{\lambda} \\
\widetilde{H}_{\mu}[\xx(1-t);q,t] = \sum_{\lambda \geq \mu'} b_{\lambda \mu}(q,t) s_{\lambda}
\end{cases}
\end{equation}
and which satisfy the {\em normalization axiom}
\begin{equation}
\langle \widetilde{H}_{\mu}[\xx;q,t], s_{(n)} \rangle = 1.
\end{equation}
Here we use the standard shorthand
\begin{equation}
\xx(1-q) := x_1 + x_2 + \cdots - q x_1 - q x_2 - \cdots \quad \text{and} \quad 
\xx(1-t) := x_1 + x_2 + \cdots - t x_1 - t x_2 - \cdots
\end{equation}
for expressions inside plethystic brackets.
Haglund, Haiman, and Loehr  \cite{HHL} gave a combinatorial formula for the polynomials 
$\widetilde{H}_{\mu}[\xx;q,t]$ involving tableaux.
Haiman  \cite{HaimanHilbert} proved that the $\widetilde{H}_{\mu}[\xx;q,t]$ are the  Frobenius images
$\grFrob(V_{\mu};q,t)$ of a certain bigraded $S_n$-module $V_{\mu}$, and hence Schur-positive.

\subsection{The diagonal coinvariant ring $DR_n$}
Let $\xx_n = (x_1, \dots, x_n)$ and $\yy_n = (y_1, \dots, y_n)$ be two lists of $n$ variables and 
let $\CC[\xx_n, \yy_n]$ be the polynomial ring over these variables.
We consider the {\em diagonal} action of $S_n$ on $\CC[\xx_n, \yy_n]$ given by
\begin{equation}
w \cdot x_i := x_{w(i)} \quad \quad w \cdot y_i := y_{w(i)} \quad \quad w \in S_n, \,\, 1 \leq i \leq n
\end{equation}
and let $\CC[\xx_n, \yy_n]^{S_n}_+ \subseteq \CC[\xx_n, \yy_n]$ be the space of $S_n$-invariants
with vanishing constant term.

In the early 1990s, Garsia and Haiman \cite{GH, HaimanQuotient}
 initiated the study of the {\em diagonal coinvariant ring}.
This is the quotient 
\begin{equation}
DR_n := \CC[\xx_n, \yy_n] / \langle \CC[\xx_n, \yy_n]^{S_n}_+ \rangle 
\end{equation}
obtained by modding out $\CC[\xx_n, \yy_n]$ by the ideal generated by the $S_n$-invariants with vanishing 
constant term. 
By considering $x$-degree and $y$-degree separately, the ring $DR_n = \bigoplus_{i, j \geq 0} (DR_n)_{i,j}$
is a doubly graded $S_n$-module.

Setting the $y$-variables to zero in $DR_n$, we recover the classical $S_n$-coinvariant ring 
\begin{equation}
R_n := \CC[x_1, \dots, x_n] / \langle \CC[x_1, \dots, x_n]^{S_n}_+ \rangle = \CC[x_1, \dots, x_n] / \langle e_1, e_2, \dots, e_n \rangle
\end{equation}
obtained from $\CC[x_1, \dots, x_n]$ by modding out by the elementary symmetric polynomials.
The singly graded $S_n$-module $R_n$ is among the most well-studied representations in algebraic 
combinatorics.
As an ungraded module, we have $R_n \cong_{S_n} \CC[S_n]$, so that $R_n$ is a graded refinement of the 
regular representation of $S_n$.
The ring $R_n$ has many interesting bases indexed by permutations $w \in S_n$ \cite{Artin, GS}
and presents
the
cohomology $H^{\bullet}(\Fl(n))$ of the flag variety \cite{Borel}.

The quotient ring $DR_n$ has remarkable  properties.
Using the algebraic geometry of the Hilbert scheme of $n$ points in $\CC^2$, 
Haiman proved \cite{HaimanVanish} that $DR_n$ has vector space dimension
$(n+1)^{n-1}$.
No elementary proof of this dimension formula is known.

The ungraded $S_n$-structure of $DR_n$ is governed by size $n$
{\em parking functions}; these are 
sequences $(a_1, \dots, a_n)$ of positive integers whose nondecreasing rearrangement
$(b_1 \leq \cdots \leq b_n)$ satisfies $b_i \leq i$ for all $i$.
The symmetric group $S_n$ acts on the set $\Park_n$ of length $n$ parking functions
by subscript permutation.  A beautiful
combinatorial argument  called the {\em Cycle Lemma} gives the enumeration $|\Park_n| = (n+1)^{n-1}$.
Haiman proved that 
\begin{equation}
\label{ungraded-dr-isomorphism}
DR_n \cong \CC[\Park_n] \otimes \sign
\end{equation}
as ungraded $S_n$-modules.  
Carlsson and Oblomkov \cite{CO} used the theory of affine Springer fibers to find a monomial basis
of $DR_n$ naturally indexed by parking functions. Their work gives an alternative (still geometric)
proof of the dimension
formula
$\dim DR_n = (n+1)^{n-1}$.

Haiman refined the isomorphism \eqref{ungraded-dr-isomorphism} by giving a description of
the bigraded Frobenius image $\grFrob(DR_n; q,t)$ of the diagonal coinvariants, where $q$ tracks
$x$-degree and $t$ tracks $y$-degree.
In order to state Haiman's result, we will need the {\em nabla operator} on symmetric functions 
introduced by F. Bergeron, Garsia, Haiman, and Tesler \cite{BGHT}.
This is the Macdonald eigenoperator $\nabla: \Lambda \rightarrow \Lambda$ characterized by 
\begin{equation}
\nabla: \widetilde{H}_{\mu}(\xx;q,t) \mapsto q^{\kappa(\mu)} t^{\kappa(\mu')} \cdot \widetilde{H}_{\mu}(\xx;q,t)
\end{equation}
where $\kappa$ is the partition statistic
$\kappa(\lambda) := \sum_{i \geq 1} (i-1) \cdot \lambda_i$. For 
example, suppose $\mu = (3,2)$ so that $\mu' = (2,2,1)$.
We compute $\kappa(\mu) = 2$ and $\kappa(\mu') = 4$
so that
\begin{equation*}
\nabla: \widetilde{H}_{32}(\xx;q,t) \mapsto q^2 t^4 \cdot \widetilde{H}_{32}(\xx;q,t).
\end{equation*}
Haiman used algebraic geometry  \cite{HaimanVanish} to prove
\begin{equation}
\label{dr-bigraded-nabla}
\grFrob(DR_n;q,t) = \nabla e_n,
\end{equation}
cementing a connection between the 
representation theory of $DR_n$ and the combinatorics of $\Lambda$.

Thanks to Haiman's theorem~\eqref{dr-bigraded-nabla}, finding the bigraded $S_n$-isomorphism type of $DR_n$ 
amounts to finding the coefficients of the Schur expansion 
$\nabla e_n = \sum_{\lambda \vdash n} c_{\lambda}(q,t) \cdot s_{\lambda}$. The polynomial 
$c_{\lambda}(q,t)$ is the bigraded multiplicity of $S^{\lambda}$ in $DR_n$, so has nonnegative
coefficients and satisfies the $q,t$-symmetry $c_{\lambda}(q,t) = c_{\lambda}(t,q)$.
Finding explicit manifestly positive formulas for the $c_{\lambda}(q,t)$ has proven daunting.
When $\lambda = (a,1^b)$ is a hook, Haglund's {\em $q,t$-Schr\"oder Theorem} 
\cite{HaglundSchroder}
describes $c_{\lambda}(q,t)$ as the bivariate generating function $\mathrm{Sch}_{\lambda}(q,t)$
 for a certain pair 
of statistics on Motzkin paths. Although $\mathrm{Sch}_{\lambda}(q,t)$ 
is manifestly positive, it is not manifestly
symmetric; no combinatorial proof of $q,t$-symmetry 
$\mathrm{Sch}_{\lambda}(q,t) = \mathrm{Sch}_{\lambda}(t,q)$ is known. 
For general partitions $\lambda \vdash n$, there is not even a conjectural combinatorial expression
for $c_{\lambda}(q,t)$.

When finding the Schur expansion of a symmetric function $F \in \Lambda$ is too difficult, it can be easier
to find its monomial expansion. Using a dazzling array of symmetric function machinery, in 2015 
Carlsson and Mellit \cite{CM}
proved the {\em Shuffle Theorem} 
\begin{equation}
\label{shuffle-theorem}
\nabla e_n = \sum_{P} q^{\area(P)} t^{\dinv(P)} \xx^P
\end{equation}
where the sum is over size $n$ word parking functions $P$, thus giving a combinatorial
monomial expansion of $\nabla e_n$. 
The ingredients in the right hand side of Equation~\eqref{shuffle-theorem} may be described as follows.

A {\em Dyck path} of size $n$ is a lattice path from $(0,0)$ to $(n,n)$ which never sinks below the diagonal
$y = x$.  A {\em word parking function} $P$ is obtained from a Dyck path by labeling its vertical steps with 
nonnegative integers which increase going up vertical runs.
To any word parking function $P$, we associate a monomial
weight  $\xx^P = x_1^{c_1} x_2^{c_2} \cdots $ where
$c_i$ is the number of copies of $i$ in $P$.
In the example below we have $n = 5$ and  $\xx^P  = x_1 x_2^2 x_3 x_6$.

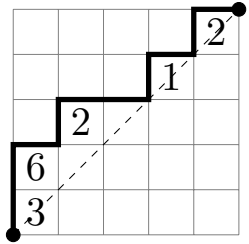
\begin{center}
\begin{tikzpicture}[scale=0.6]
(8,0) rectangle +(5,5);
\draw[help lines] (8,0) grid +(5,5);
\draw[dashed] (8,0) -- +(5,5);
\coordinate (prev) at (8,0);
\draw [color=black, line width=2] (8,0)--(8,2)--(9,2)--(9,3)--(11,3)--(11,4)--(12,4)--(12,5)--(13,5);

\draw (8,0) node [scale=0.5, circle, draw,fill=black]{};
\draw (13,5) node [scale=0.5, circle, draw,fill=black]{};

\node[align=left,scale=1.4] at (8.5,0.5) {3};
\node[align=left,scale=1.4] at (8.5,1.5) {6};
\node[align=left,scale=1.4] at (9.5,2.5) {2};
\node[align=left,scale=1.4] at (11.5,3.5) {1};
\node[align=left,scale=1.4] at (12.5,4.5) {2};
 \end{tikzpicture}
\end{center}

If $P$ is a size $n$ word parking function, the {\em area} of $P$ is the number $\area(P)$
of full boxes between $P$ and the diagonal.  In particular,
the area of $P$ depends only on the Dyck path underlying $P$,
not on its labeling with positive integers.
In the above example we have $\area(P) = 2$.

The other statistic $\dinv(P)$ appearing in \eqref{shuffle-theorem}
depends on the labeling of $P$ (not just the Dyck path) and is more involved.
A pair of labels $a < b$ in $P$ form a {\em diagonal inversion} if 
\begin{itemize}
\item  $a$ and $b$ appear in the same diagonal of $P$, and $a$ appears in a lower row than $b$
(this is a {\em primary diagonal inversion}), or
\item $a$ appears one diagonal below $b$ in $P$, and $a$ appears in a higher row than $b$
(this is a {\em secondary diagonal inversion}). 
\end{itemize}
In the above example, the single primary diagonal inversion is $(1,2)$ on the main diagonal.
We have the secondary diagonal inversions $(1,6), (2,6),$ and $(1,2)$ involving the main diagonal
and the `superdiagonal'.
We let $\dinv(P)$ be the total number of primary and secondary diagonal inversions, so that $\dinv(P) = 4$
in our case.

The  identity
\eqref{shuffle-theorem} was conjectured by 
Haglund, Haiman, Loehr, Remmel, and Ulyanov \cite{HHLRU} in 2005.
In the decade between its conjecture and its proof, many more general symmetric 
function identities were conjectured \cite{BGLX, HMZ, HRW}
in an attempt to place 
\eqref{shuffle-theorem} in a position for inductive attack.
Like \eqref{shuffle-theorem} itself, these
 identities have two sides: a `symmetric function side' involving the action of operators such as 
$\nabla$ on the ring $\Lambda$ and a `combinatorial side' involving statistics 
on combinatorial objects.
The Carlsson-Mellit proof  \cite{CM} of \eqref{shuffle-theorem} went through one such refinement:
the {\em Compositional Shuffle Conjecture} of Haglund, Morse, and Zabrocki \cite{HMZ}.
Shortly before the work of \cite{CM}, Haglund, Remmel and Wilson \cite{HRW} 
formulated another such refinement called  the {\em Delta Conjecture}. 
This refinement proved more difficult than the Shuffle Theorem itself; its `Rise Version'
resisted proof for 
five years and its `Valley Version' remains open.
The Delta Conjecture involves a family of symmetric function operators generalizing $\nabla$.

\subsection{Delta operators}
Given a partition $\mu$, define a polynomial $B_{\mu} = \sum_{(i,j)} q^{i-1} t^{j-1}$, where the sum is over all 
matrix coordinates $(i,j)$ of cells in $\mu$.  For example, if $\mu = (3,2)$ the cells of $\mu$ 
contribute the monomials
\begin{equation*}
\begin{Young}
1 & q & q^2 \cr t & qt 
\end{Young}
\end{equation*}
so that $B_{\mu} = 1 + q + q^2 + t + qt$. If $F \in \Lambda$ is any symmetric function,
the {\em delta operator} $\Delta_F: \Lambda \rightarrow \Lambda$ indexed by $F$
is the Macdonald eigenoperator
\begin{equation}
\Delta_F: \widetilde{H}_{\mu}(\xx;q,t) \mapsto F[B_{\mu}] \cdot \widetilde{H}_{\mu}(\xx;q,t)
\end{equation}
or, in non-plethystic language, 
\begin{equation}
\Delta_F: \widetilde{H}_{\mu}(\xx;q,t) \mapsto F( \dots, q^{i-1} t^{j-1}, \dots) \cdot \widetilde{H}_{\mu}(\xx;q,t)
\end{equation}
where $(i,j)$ range over all cells of $\mu$.
For example, we have
\begin{equation*}
\Delta_F: \widetilde{H}_{32}(\xx;q,t) \mapsto F(1, q, q^2, t, qt) \cdot \widetilde{H}_{32}(\xx;q,t)
\end{equation*}
where the notation $F(1, q, q^2, t, qt)$ means that
the variables  $x_1, x_2, x_3, x_4, x_5$ in $F(x_1, x_2, \dots )$  are set to $1, q, q^2, t, qt$ and all other variables
$x_6, x_7, \dots$ are set to zero. 
We will use a {\em primed} version of the delta operators 
$\Delta'_F: \Lambda \rightarrow \Lambda$ characterized by
\begin{equation}
\Delta'_F: \widetilde{H}_{\mu}(\xx;q,t) \mapsto F[B_{\mu} - 1] \cdot \widetilde{H}_{\mu}(\xx;q,t)
\end{equation}
so that the arguments $F( \dots, q^{i-1} t^{j-1}, \dots )$ of the eigenvalue of $\widetilde{H}_{\mu}(\xx;q,t)$
do not contain the northwest cell of $\mu$ with coordinates $(1,1)$. In our example, this corresponds
to the filling
\begin{equation*}
\begin{Young}
\cdot & q & q^2 \cr t & qt 
\end{Young}
\end{equation*}
and we have
\begin{equation*}
\Delta'_F: \widetilde{H}_{32}(\xx;q,t) \mapsto F(q, q^2, t, qt) \cdot \widetilde{H}_{\mu}(\xx;q,t).
\end{equation*}

The delta operators were defined by F. Bergeron and Garsia.
 A quick check on the Macdonald basis shows that 
 \begin{equation}
 \Delta_{e_n} = \Delta'_{e_{n-1}} = \nabla
 \end{equation}
 as operators on the space $\Lambda_n$ of degree $n$ symmetric functions,
 so the Shuffle Theorem gives the monomial expansion 
 of $\Delta'_{e_{n-1}} e_n$.
 For $k \leq n$, the Delta Conjecture predicts monomial expansions
of the more general symmetric functions $\Delta'_{e_{k-1}} e_n$.
 Roughly speaking, these expansions are obtained by refining the parking
 function statistics $\area$ and $\dinv$.

The statistic $\area(P)$ admits a natural refinement.
If $P$ is a size $n$ word parking function and $1 \leq i \leq n$,
we set
\begin{equation}
a_i(P) := \text{the number of full boxes in row $i$ between $P$ and the diagonal.}
\end{equation}
The numbers $a_1(P), \dots, a_n(P)$ depend only on the Dyck path underlying $P$, not on the choice of labeling.
We have 
$\area(P) = a_1(P) + \cdots + a_n(P)$.

Refining $\dinv(P)$ is more complicated.
For $1 \leq i \leq n$, define a number $d_i(P)$ by
\begin{multline}
d_i(P) := |\{i < j \leq n \,:\, a_i(P) = a_j(P), \ell_i(P) < \ell_j(P)\}|  \\
+ |\{i < j \leq n \,:\, a_i(P) = a_j(P) + 1, \ell_i(P) > \ell_j(P)\}|
\end{multline}
where $\ell_i(P)$ is the label in row $i$ of $P$.
Roughly speaking, $d_i(P)$ counts diagonal inversions in $P$ which `originate in' row $i$.
We have
$\dinv(P) = d_1(P) + \cdots + d_n(P)$
For our example parking function, the values of $a_i(P)$ and $d_i(P)$ are as 
follows.

\begin{center}
\begin{tikzpicture}[scale=0.6]
(8,0) rectangle +(5,5);
\draw[help lines] (8,0) grid +(5,5);
\draw[dashed] (8,0) -- +(5,5);
\coordinate (prev) at (8,0);
\draw [color=black, line width=2] (8,0)--(8,2)--(9,2)--(9,3)--(11,3)--(11,4)--(12,4)--(12,5)--(13,5);

\draw (8,0) node [scale=0.5, circle, draw,fill=black]{};
\draw (13,5) node [scale=0.5, circle, draw,fill=black]{};

\node[align=left,scale=1.4] at (8.5,0.5) {3};
\node[align=left,scale=1.4] at (8.5,1.5) {6};
\node[align=left,scale=1.4] at (9.5,2.5) {2};
\node[align=left,scale=1.4] at (11.5,3.5) {1};
\node[align=left,scale=1.4] at (12.5,4.5) {2};

\node[align=left,scale=1.4] at (14.5,5.5) {$i$};
\node[align=left,scale=1.4] at (14.5,4.5) {$5$};
\node[align=left,scale=1.4] at (14.5,3.5) {$4$};
\node[align=left,scale=1.4] at (14.5,2.5) {$3$};
\node[align=left,scale=1.4] at (14.5,1.5) {$2$};
\node[align=left,scale=1.4] at (14.5,0.5) {$1$};

\node[align=left,scale=1.4] at (16,5.4) {$a_i$};
\node[align=left,scale=1.4] at (16,4.5) {$0$};
\node[align=left,scale=1.4] at (16,3.5) {$0$};
\node[align=left,scale=1.4] at (16,2.5) {$1$};
\node[align=left,scale=1.4] at (16,1.5) {$1$};
\node[align=left,scale=1.4] at (16,0.5) {$0$};

\node[align=left,scale=1.4] at (17.5,5.5) {$d_i$};
\node[align=left,scale=1.4] at (17.5,4.5) {$0$};
\node[align=left,scale=1.4] at (17.5,3.5) {$1$};
\node[align=left,scale=1.4] at (17.5,2.5) {$1$};
\node[align=left,scale=1.4] at (17.5,1.5) {$2$};
\node[align=left,scale=1.4] at (17.5,0.5) {$0$};
 \end{tikzpicture}
\end{center}

The Shuffle Theorem gives an expression for $\nabla e_n$ in terms of the statistics 
$\area = a_1 + \cdots + a_n$ and $\dinv = d_1 + \cdots + d_n$.
To get an expression for $\Delta'_{e_{k-1}} e_n$, we cancel out the contributions of some of the
 $a_1, \dots, a_n, d_1, \dots, d_n$ to these statistics.
If $P$ is a labeled Dyck path of size $n$, the set $\Val(P)$ of {\em contractible valleys} of $P$
is
\begin{multline}
\Val(P) := \{2 \leq i \leq n \,:\, a_i(P) < a_{i-1}(P)\} \\  \cup \{2 \leq i \leq n \,:\, a_i(P) = a_{i-1}(P), \ell_i(P) > \ell_{i-1}(P)\}.
\end{multline}
In other words, a valley (i.e. an east step followed by a north step) in row $i$ is contractible if moving this valley
one box east gives a valid labeled Dyck path.
in the example above we have $\Val(P) = \{4,5\}$.

\begin{conjecture} (Haglund-Remmel-Wilson \cite{HRW})
{\em (The Delta Conjecture)}
For positive integers $k \leq n$,
\begin{align}
\label{riseform}
\Delta'_{e_{k-1}} e_n
&= \{z^{n-k}\} \left[ \sum_{P \in \mathcal{LD}_n} q^{\mathrm{dinv}(P)} t^{\mathrm{area}(P)}
\prod_{i \, : \, a_i(P) > a_{i-1}(P)} \left( 1 + z/t^{a_i(P)} \right) \xx^P \right] \\
&= \{z^{n-k}\} \left[ \sum_{P \in \mathcal{LD}_n} q^{\mathrm{dinv}(P)} t^{\mathrm{area}(P)}
\prod_{i \in \mathrm{Val}(P)} \left( 1 + z/q^{d_i(P) + 1} \right) \xx^P \right].
\end{align}
Here the operator $\{z^{n-k}\}$ extracts the coefficient of $z^{n-k}$.
\end{conjecture}

We denote by $\Rise_{n,k}(\xx;q,t)$ and $\Val_{n,k}(\xx;q,t)$ the two combinatorial expressions
on the right-hand side, so the Delta Conjecture may be written more succinctly as 
\begin{equation}
 \label{delta-theorem}
 \Delta'_{e_{k-1}} e_n = \Rise_{n,k}(\xx;q,t) = \Val_{n,k}(\xx;q,t).
 \end{equation}
The expression $\Rise_{n,k}$ is obtained from the decomposition 
$\area = a_1 + \cdots + a_n$ whereas $\Val_{n,k}$ is obtained from
$\dinv = d_1 + \cdots + d_n$.
 The Delta Conjecture specializes to the Shuffle Theorem when $k = n$.

 There has been substantial progress on the Delta Conjecture since its introduction in \cite{HRW}.
 The `Rise Version', i.e. the equality 
 $\Delta'_{e_{k-1}} e_n = \Rise_{n,k}(\xx;q,t)$ was proven independently 
 by D'Adderio-Mellit \cite{DM} and
 Blasiak-Haiman-Morse-Pun-Seelinger \cite{BHMPS} using different methods.
D'Adderio and Mellit proved a refined  `compositional' version of the 
 equality  $\Delta'_{e_{k-1}} e_n = \Rise_{n,k}(\xx;q,t)$
 conjectured by D'Adderio-Iraci-Wyngaerd \cite{DIW} involving `$\Theta$-operators'
 on symmetric functions. In \cite{DM} these $\Theta$-operators are used to promote the Carlsson-Mellit proof
 of the Shuffle Theorem to a proof of the Rise Version of the Delta Conjecture.
 Blasiak et. al. proved an `extended' version of $\Delta'_{e_{k-1}} e_n = \Rise_{n,k}(\xx;q,t)$ which gives 
 a monomial expansion of 
  $\Delta_{h_r} \Delta'_{e_{k-1}} e_n$. The proof in \cite{BHMPS} uses an action of
 Schiffmann algebra $\mathcal{E}$ on the ring of symmetric functions.
 Neither result in \cite{DM} or \cite{BHMPS} contains the other.
 
Far less is known about the `Valley Version' 
 $\Delta'_{e_{k-1}} e_n = \Val_{n,k}(\xx;q,t)$ of the Delta Conjecture.
It is not even known whether $\Val_{n,k}(\xx;q,t)$ is symmetric in the $\xx$-variables.
However, the Valley Version is proven when one of the variables $q, t$ is set to zero. Combining results in
 \cite{GHRY, HRW, RhoadesOSP},
we have
\begin{equation}
\label{valley-delta-at-zero}
\Delta'_{e_{k-1}} e_n \mid_{t = 0} = \Val_{n,k}(\xx;q,0) = \Val_{n,k}(\xx;0,q).
\end{equation}
We focus on $t = 0$ specialization $\Delta'_{e_{k-1}} e_n \mid_{t = 0}$ for most of this
chapter; in the final section we look at $\Delta'_{e_{k-1}} e_n$ itself.

\section{Quotient rings}
\label{Quotient}

\subsection{The ring $R_{n,k}$}
The motivation for the Shuffle Theorem was to develop a combinatorial understanding of $\nabla e_n$,
the bigraded Frobenius image of the diagonal coinvariants $DR_n$.
In turn, the ring $DR_n$ was initially motivated by the classical coinvariant ring $R_n$ with its 
ties to the flag variety $\Fl_n$.  As a significant member of a family of (sometimes conjectural) identities
developed to prove the Shuffle Theorem, it is natural to ask whether there is an analogous 
algebraic and geometric theory tied to $\Delta_{e_{k-1}} e_n$.
To this end, Haglund, Rhoades, and Shimozono \cite{HRS} introduced the following quotient rings.

\begin{defn}
\label{rnk-defn}
Let $k \leq n$ be positive integers. We define an ideal $I_{n,k} \subseteq \CC[x_1, \dots, x_n]$ by 
\begin{equation*}
I_{n,k} := \langle e_n, e_{n-1}, \dots, e_{n-k+1}, x_1^k, x_2^k, \dots, x_n^k \rangle
\end{equation*}
and denote the corresponding quotient ring by
\begin{equation*}
R_{n,k} := \CC[x_1, \dots, x_n]/I_{n,k}
\end{equation*}
\end{defn}

The ideal $I_{n,k}$ is homogeneous, so $R_{n,k}$ is a graded ring. Since the generating set
of $I_{n,k}$ is stable under the action of $S_n$, the quotient $R_{n,k}$ is a graded $S_n$-module.
When $k = 1$, we recover the ground field
 $R_{n,1} = \CC$ since  $x_1, \dots, x_n \in I_{n,1}$.  When $k = n$, it is not hard to see that the variable
 powers 
 $x_i^n$ are in the classical coinvariant ideal $I_n = \langle e_1, \dots, e_n \rangle$,
 so that $I_{n,n} = I_n$.

%Borel presented \cite{Borel} the cohomology of $\Fl(n)$ as 
%\begin{equation}
%H^{\bullet}(\Fl(n); \ZZ) = \ZZ[\xx_n]/\langle e_1, e_2, \dots, e_n \rangle.
%\end{equation} 
%With an eye towards the geometry, we introduce integral versions of the ideal
%\begin{equation}
%I_{n,k}^{\ZZ} :=  \langle e_n, e_{n-1}, \dots, e_{n-k+1}, x_1^k, x_2^k, \dots, x_n^k \rangle \subseteq 
%\ZZ[x_1, \dots, x_n]
%\end{equation}
%and the quotient ring
%\begin{equation}
%R_{n,k}^{\ZZ} := \ZZ[x_1, \dots, x_n]/I_{n,k}^{\ZZ}
%\end{equation}
%of Definition~\ref{rnk-defn}.
%We will define a variety $X_{n,k}$ whose cohomology is presented by $R_{n,k}^{\ZZ}$.

The algebra $R_n$ and geometry $\Fl(n)$ are governed by the combinatorics
of $S_n$.
The following combinatorial objects play an analogous role for $R_{n,k}$ and $X_{n,k}$.

\begin{defn}
\label{fubini-words}
A word $w = [w(1), \dots, w(n)]$ over the positive integers is a {\em Fubini word} if 
 for all $i > 1$ such that the letter $i$ appears in $w$, the letter $i-1$ also appears in $w$.
We let $\WWW_{n,k}$ denote the family of Fubini words of length $n$ with maximum letter $k$.
\end{defn}

As an example, we have $\WWW_{3,2} = \{[1,1,2], [1,2,1], [2,1,1], [1,2,2], [2,1,2], [2,2,1] \}$. Words in $\WWW_{n,k}$
may be thought of as surjective maps $w: [n] \twoheadrightarrow [k]$, giving a natural identification
$\WWW_{n,n} = S_n$.

 There is another way to think of Fubini words that will be useful for us.
 An {\em ordered set partition} of $[n]$ is a sequence $\sigma = (B_1 \mid \cdots \mid B_k)$
 of nonempty subsets of $[n]$ such that we have a disjoint union
 $[n] = B_1 \sqcup \cdots \sqcup B_k$. 
 We write $\OP_{n,k}$ for the family of all ordered set partitions of $[n]$ into $k$ blocks.
There is a natural bijection between $\WWW_{n,k} \xrightarrow{\, \sim \,} \OP_{n,k}$, viz.
 \begin{equation*}
[3,2,2,4,1,3,2] \mapsto ( 5  \mid 2 \, 3 \, 7  \mid 1 \, 6  \mid 4 ).
\end{equation*}
This bijection
$\WWW_{n,k} \xrightarrow{\, \sim \,} \OP_{n,k}$ is the inverse map
$w \mapsto w^{-1}$ on the symmetric group $S_n$ when $k = n$,
so that Fubini words and ordered set partitions are `inverse objects'.
We have
\begin{equation}
|\WWW_{n,k}| = |\OP_{n,k}| =  k! \cdot \Stir(n,k)
\end{equation}
where $\Stir(n,k)$ is the (signless)
Stirling number of the second kind counting $k$-block set partitions of $[n]$.

We will see that 
Fubini words govern the structure of $R_{n,k}$.
The group $S_n$ acts on the set $\WWW_{n,k}$ by
\begin{equation}
v \cdot [w(1),  \dots, w(n)] := [w(v^{-1}(1)), \dots, w(v^{-1}(n))] \quad \quad v \in S_n, \, 
[w(1), \dots, w(n)] \in \WWW_{n,k}.
\end{equation}
The permutation module $\CC[\WWW_{n,k}]$ is the regular representation $\CC[S_n]$  when $k = n$.
For general $k \leq n$, we will prove that $R_{n,k} \cong \CC[\WWW_{n,k}]$ as ungraded $S_n$-modules,
so  $R_{n,k}$ is a graded refinement of $\CC[\WWW_{n,k}]$.

\subsection{Coinversion codes} 
Our first step towards proving $R_{n,k} \cong \CC[\WWW_{n,k}]$ is to attach monomials
in $\CC[x_1, \dots, x_n]$ to ordered set partitions in $\OP_{n,k}$ (or, equivalently, Fubini words in $\WWW_{n,k}$).
The gadget we use for this is the {\em coinversion code} of an ordered set partition \cite{RWLine}. We review the classical case of permutations in $S_n$,
and then explain how to generalize to $\OP_{n,k}$.

The {\em inversion number} $\inv(w)$ and the
 {\em coinversion number} $\coinv(w)$ of a permutation $w \in S_n$ are
\begin{equation}  \inv(w) := | \{ 1 \leq i < j \leq n \,:\, w(i) > w(j)\}| \quad \quad 
\coinv(w) := | \{ 1 \leq i < j \leq n \,:\, w(i) < w(j) \} |.
\end{equation}
These statistics are complementary: we have
$\coinv(w) + \inv(w) = {n \choose 2}$ for any $w \in S_n$.
Their common generating function is the 
{\em Mahonian distribution}  
\begin{equation}
\sum_{w \in S_n} q^{\inv(w)} = \sum_{w \in S_n} q^{\coinv(w)} = [n]!_q
\end{equation}
where
\begin{equation}
[n]!_q := (1+q)(1+q+q^2) \cdots (1+q+q^2 + \cdots + q^{n-1})
\end{equation}
is the standard $q$-analogue of $n!$.
We  work primarily with coinversions in this subsection.

The statistic $\coinv$ can be broken into pieces.
The {\em coinversion code} associated to a permutation $w \in S_n$ is the sequence
$\code(w) = (c_1,  \dots, c_n)$ where 
\begin{equation}
c_i := | \{ j > i \,:\, w(j) > w(i) \}|
\end{equation}
For example, we have $\code ( [3,1,5,2,4] )= (2, 3, 0, 1, 0)$.
The coinversion code refines $\coinv$ in the sense that $\coinv(w) = c_1 + \cdots + c_n$.

Let $E_n := \{ (a_1, \dots, a_n) \,:\, 0 \leq a_i \leq n-i \}$ be the set of words over the positive integers
which are componentwise $\leq$ the staircase $(n-1, n-2, \dots, 1, 0)$.  It is not hard to see that 
$\code(w) \in E_n$ for any $w \in S_n$. In fact,
the map $\code: w \mapsto \code(w)$ 
is a bijection $S_n \xrightarrow{ \, \sim \, } E_n$; see \cite[Sec. 1.3]{EC1} for  details.

%Given a sequence of positive integers and copies of the symbol $\varnothing$, 
%define the {\em coinversion labeling} by labeling the $\varnothing$'s 
%with $0, 1, 2, \dots $ from right to left.  For example, the coinversion labeling of
%$\varnothing \, 1 \, \varnothing \, 2 \, \varnothing$ is shown in superscripts as 
%$\varnothing^2 \, 1 \, \varnothing^1 \, 2 \, \varnothing^0$.
%For $(c_1, \dots, c_n) \in E_n$,  define the one-line notation of a permutation 
%$\iota(c_1, \dots, c_n) = w(1) \dots w(n) \in S_n$
%by starting with the length $n$ sequence $\varnothing \, \dots \, \varnothing$ and, processing $(c_1, \dots, c_n)$
%from left to right, inserting an $i$ at coinversion label $c_i$ (updating coinversion labels along the way).
%For example, when $n = 5$ and $(c_1, \dots, c_5) = (3,1,2,0,0)$ the insertion algorithm proceeds as follows
%\begin{equation*}
%\begin{tabular}{c | c | c}
%$i$ & $c_i$ & $w(1) \dots w(n)$ \\  \hline
%1 & 3 & $\varnothing^4 \, \varnothing^3 \, \varnothing^2 \, \varnothing^1 \, \varnothing^0$ \\
%2 & 1 & $\varnothing^3 \, 1 \, \varnothing^2 \, \varnothing^1 \, \varnothing^0$ \\
%3 & 2 & $\varnothing^2\, 1 \, \varnothing^1 \, 2 \, \varnothing^0$ \\
%4 & 0 & $3 \, 1 \, \varnothing^1 \, 2 \, \varnothing^0$ \\
%5 & 0 & $3 \, 1 \, \varnothing^0 \, 2 \, 4$ 
%\end{tabular}
%\end{equation*}
%and we conclude that $\iota(3,1,2,0,0) = [3,1,5,2,4]$.
%It is not hard to see that the maps $\code: S_n \rightarrow E_n$ and $\iota: E_n \rightarrow S_n$ 
%are mutually inverse.

Our aim is to extend the story above from $S_n$ to $\OP_{n,k}$, beginning with the notion of inversion and
coinversion pairs. Let $\sigma = (B_1 \mid \cdots \mid B_k ) \in \OP_{n,k}$ be an ordered set partition and let
$1 \leq a \neq b \leq n$ be a pair of letters belonging to distinct blocks of $\sigma$ such that $b$ is 
minimal in its block.
The pair $\{a,b\}$ is an {\em inversion pair} of $\sigma$ if 
the block of $a$ is to the right of the block of $b$ and $a < b$.
Otherwise, the pair $\{a,b\}$ is a {\em coinversion pair}.
We have statistics
\begin{equation}
\inv(\sigma) := \# \text{ of inversion pairs in $\sigma$}  \quad \quad
\coinv(\sigma) := \# \text{ of coinversion pairs in $\sigma$} 
\end{equation}

As an example of these concepts, let $\sigma = ( 6 \mid 1 \, 4 \mid 2 \, 3 \, 7 \mid 5 ) \in \OP_{7,4}$.
The inversion pairs of $\sigma$ are $\{ 16, 26, 56, 24, 75 \}$ so that $\inv(\sigma) = 5$ and the coinversion
pairs
are $\{12, 15, 25, 36, 13, 35, 46, 45, 67, 17\}$ so that $\coinv(\sigma) = 10$.
Observe that (for example) $34$ is not an inversion pair
4 is not minimal in its block.
The statistics $\inv$ and $\coinv$ are complementary on $\OP_{n,k}$. More precisely,
we have 
\begin{equation}
\inv(\sigma) + \coinv(\sigma) = {k \choose 2} + (n-k)(k-1)
\end{equation}
for all $\sigma \in \OP_{n,k}$.

The  statistic $\inv$
 on $\OP_{n,k}$ was introduced by 
Steingr\'imsson \cite{Steingrimsson} in a general study of ordered set partition statistics and rediscovered
by Remmel-Wilson \cite{RemmelWilson}
in the context of the symmetric function $\Rise_{n,k}(\xx;q,t)$.
Steingr\'imsson proved that the generating 
function of $\inv$ is
\begin{equation}
\label{first-mahonian}
\sum_{\sigma \in \OP_{n,k}} q^{\inv(\sigma)} = [k]!_q \cdot \Stir_q(n,k)
\end{equation}
where the {\em $q$-Stirling number} is defined recursively by
\begin{equation}
\label{q-stirling-recursion}
\Stir_q(n,k) := \begin{cases}
\delta_{n,k} & n = 1 \\
\Stir_q(n-1,k-1) + [k]_q \cdot \Stir_q(n-1,k) & n > 1
\end{cases}
\end{equation}
Unlike in the permutation case, this distribution is {\bf not} palindromic. For example,
we have $\sum_{\sigma \in \OP_{3,2}} q^{\inv(\sigma)} = q^2 + 3q + 2$.
The distribution of $\coinv$ is the reversal
\begin{equation}
\label{second-mahonian}
\sum_{\sigma \in \OP_{n,k}} q^{\coinv(\sigma)} = \rev_q ( [k]!_q \cdot \Stir_q(n,k) ) =
q^{{k \choose 2} + (n-k)(k-1)} \cdot [k]!_{q^{-1}} \cdot \Stir_{q^{-1}}(n,k)
\end{equation}
so that (for example) $\sum_{\sigma \in \OP_{3,2}} q^{\coinv(\sigma)} = 2q^2 + 3q + 1$.
Here $\rev_q$ acts on polynomials in $q$ by reversing their coefficient sequences.

The statistic $\inv$ is tied to the Delta Conjecture.
Haglund, Remmel, and Wilson \cite{HRW} showed that 
\begin{equation}
\langle \Rise_{n,k}(\xx; q,0), p_{1^n} \rangle = 
\sum_{\sigma \in \OP_{n,k}} q^{\inv(\sigma)},
\end{equation}
giving a connection between the statistic $\inv$ and delta operators.
The same authors 
defined three other statistics
($\dinv, \maj,$ and $\minimaj$) on $\OP_{n,k}$ which satisfy analogous identities
\begin{equation}
\begin{cases}
\langle \Rise_{n,k}(\xx; 0,q), p_{1^n} \rangle = 
\sum_{\sigma \in \OP_{n,k}} q^{\maj(\sigma)} \\
\langle \Val_{n,k}(\xx; q,0), p_{1^n} \rangle = 
\sum_{\sigma \in \OP_{n,k}} q^{\dinv(\sigma)} \\
\langle \Val_{n,k}(\xx; 0,q), p_{1^n} \rangle = 
\sum_{\sigma \in \OP_{n,k}} q^{\minimaj(\sigma)},
\end{cases}
\end{equation}
so that ordered set partitions model the Delta Conjecture at $t = 0$.
Wilson \cite{WMultiset} and Rhoades \cite{RhoadesOSP} used a generalization 
of these statistics to `ordered multiset partitions' to prove
\begin{equation}
\label{combinatorial-delta-at-zero}
\Rise_{n,k}(\xx;q,0) = \Rise_{n,k}(\xx;0,q) = \Val_{n,k}(\xx;q,0) = \Val_{n,k}(\xx;0,q),
\end{equation}
giving combinatorial evidence for the Delta Conjecture.

Just as in the permutation case, the coinversion statistic on $\OP_{n,k}$ can be broken into pieces.
Let $\sigma = (B_1 \mid \cdots \mid B_k) \in \OP_{n,k}$ be an ordered set partition.
The {\em coinversion code} of $\sigma$ is the sequence $\code(\sigma) = (c_1, \dots, c_n)$ where
\begin{equation}
c_i := \begin{cases}
| \{ j > \ell \,:\, \min(B_j) > i \} | & \text{if $i = \min(B_{\ell})$ is minimal in its block,} \\
(\ell - 1) + | \{ j > \ell \,:\, \min(B_{j}) > i \} | &  \text{if $i \in B_{\ell}$ is not minimal in its block.}
\end{cases}
\end{equation}
In our example $\sigma = ( 6 \mid 1 \, 4 \mid 2 \, 3 \, 7 \mid 5 ) \in \OP_{7,4}$ we have 
$\code(\sigma) = (c_1, \dots, c_7) = (2,1,3,2,0,0,2)$.
As in the case of $S_n$, we have
$\coinv(\sigma) := c_1 + \cdots + c_n$ for any $\sigma \in \OP_{n,k}$.

The map $\sigma \mapsto \code(\sigma)$ is injective on $\OP_{n,k}$; to describe its image we need
some terminology.
Recall that a {\em shuffle} of two sequences $(a_1, \dots, a_r)$ and $(b_1, \dots, b_s)$
is an interleaving $(c_1, \dots, c_{r+s})$ of these sequences which preserves the relative order of the 
$a$'s and the $b$'s.
An {\em $(n,k)$-staircase} a shuffle of the length $k$ staircase $(k-1, k-2, \dots, 1, 0)$ with the sequence
$(k-1, \dots, k-1)$ consisting of $n-k$ copies of $k-1$.
For example, when $n = 5$ and $k = 3$, the $(n,k)$-staircases are the shuffles of $(2,1,0)$ and $(2,2)$
shown here
\begin{equation*}
(2,1,0,2,2) \quad (2,1,2,0,2) \quad (2,2,1,0,2) \quad (2,1,2,2,0) \quad (2,2,1,2,0) \quad (2,2,2,1,0)
\end{equation*}
The number of $(n,k)$-staircases is the binomial coefficient ${n-1 \choose k-1}$.

\begin{defn}
\label{substaircase}
For $k \leq n$, a word $(c_1, \dots , c_n)$ over the positive integers is {\em $(n,k)$-substaircase} 
if it is componentwise $\leq$ at least one $(n,k)$-staircase.  Let $E_{n,k}$ be the family 
of $(n,k)$-substaircase sequences.
\end{defn}

As an example, we have $(1,1,0,2,0) \in E_{5,3}$ because $(1,1,0,2,0) \leq (2,1,0,2,2)$ componentwise.
We also have $(1,1,0,2,0) \leq (2,1,2,2,0)$ componentwise: sequences in $E_{n,k}$ can fit under
more than one $(n,k)$-staircase.

An alternative characterization of the set $E_{n,k}$ can be formulated in terms of 
``skip sequence avoidance". For any subset $S = \{s_1 < s_2 < \cdots < s_r \} \subseteq [n]$, define
the {\em skip sequence} $\gamma(S) = (\gamma(S)_1, \dots, \gamma(S)_n)$ by
\begin{equation}
\gamma(S)_i = \begin{cases}
i - j + 1 & \text{if $i = s_j \in S$} \\
0 & \text{if $i \notin S$}
\end{cases}
\end{equation}
For example, when $n = 8$ we have $\gamma(2458) = (0,2,0,3,3,0,0,5)$. Thus, the first nonzero
entry in $\gamma(S)$ is $\min(S)$ and every time the set $S$ `skips' an element of $[n]$, the corresponding
entry of $\gamma(S)$ is incremented.

\begin{defn}
\label{skip-defn}
A length $n$ sequence $(a_1, \dots, a_n)$ of nonnegative integers is {\em $(n,k)$-nonskip}
if 
\begin{enumerate}
\item we have $a_i < k$ for all $i$ and
\item whenever $|S| = n-k+1$, the componentwise inequality $\gamma(S) \leq (a_n, \dots, a_1)$ does {\bf not}
hold.
\end{enumerate}
Note the reversal $(a_n, \dots, a_1)$ in the second condition.
\end{defn}

For example, if $(n,k) = (5,3)$ the sequence $(1,1,0,2,0)$ is $(n,k)$-nonskip because all of its entries are $< 3$
and it is not componentwise $\leq$ any of the reverse skip sequences
\begin{equation*}
(0,0,1,1,1) \quad (0,2,0,1,1) \quad (0,2,2,0,1), \quad (0,2,2,2,0).
\end{equation*}
For arbitrary $k \leq n$ it turns out that
\begin{equation}
E_{n,k} = \{ \text{all $(n,k)$-nonskip sequences $(a_1, \dots, a_n)$} \}
\end{equation}
so that substaircase and nonskip are equivalent.
 Coinversion codes give yet another characterization of $E_{n,k}$: we have
 $\coinv(\sigma) \in E_{n,k}$ for all $\sigma \in \OP_{n,k}$, and
the function 
\begin{equation}
\code: \OP_{n,k} \xrightarrow{\,  \, \sim \, \, } E_{n,k}
\end{equation}
is a bijection.  The inverse
  $\iota: E_{n,k} \longrightarrow \OP_{n,k}$ of this map is the following insertion algorithm.
  
  Given a sequence $(B_1 \mid \cdots \mid B_k)$ of possibly empty sets of positive integers, the 
  {\em coinversion labeling} is obtained by first labeling the empty sets with $0, 1, 2, \dots, j-1$
  from left to right (where there are $j$ empty sets), and then labeling the nonempty sets with
  $j, j+1, \dots, k-1$ from right to left.  If $(c_1, \dots, c_n) \in E_{n,k}$, we define 
  $\iota(c_1, \dots, c_n) = (B_1 \mid \cdots \mid B_k)$ by starting with a sequence of $k$ empty sets
  and iteratively inserting $i$ into the set with coinversion label $c_i$, updating coinversion labels as we go.
  For example, if $(c_1, \dots, c_n) = (2,1,3,2,0,0,2)$ this insertion is as follows
 \begin{equation*}
\begin{tabular}{c | c | c}
$i$ & $c_i$ & $(B_1 \mid \cdots \mid B_k)$ \\  \hline
1 & 2 & $(\varnothing^3 \mid \varnothing^2 \mid \varnothing^1 \mid \varnothing^0)$ \\
2 & 1 & $(\varnothing^2 \mid 1^3 \mid \varnothing^1 \mid \varnothing^0)$ \\
3 & 3 & $(\varnothing^1 \mid 1^2 \mid 2^3 \mid \varnothing^0)$ \\
4 & 2 & $(\varnothing^1 \mid 1^2 \mid 2 \, 3^3 \mid \varnothing^0)$ \\
5 & 0 & $(\varnothing^1 \mid 1 \, 4^2 \mid 2 \, 3^3 \mid \varnothing^0)$ \\
6 & 0 & $(\varnothing^0 \mid 1 \, 4^1 \mid 2 \, 3^2 \mid 5^3)$ \\
7 & 2 & $(6^0 \mid 1 \, 4^1 \mid 2 \, 3^2 \mid 5^3)$ \\
\end{tabular}
\end{equation*}
and we conclude $\iota(2,1,3,2,0,0,2) =  ( 6 \mid 1 \, 4 \mid 2 \, 3 \, 7 \mid 5 )$.
Our results on codes and ordered set partitions may be summarized as follows (see \cite{HRW, RWLine}).

\begin{theorem}
\label{coinversion-bijection}
Let $k \leq n$
and let $(a_1, \dots, a_n)$ be a length $n$ sequence of nonnegative integers. The following are equivalent.
\begin{enumerate}
\item  The sequence $(a_1, \dots, a_n)$ is componentwise $\leq$ some $(n,k)$-staircase.
\item  The sequence $(a_1, \dots, a_n)$ is $(n,k)$-nonskip.
\item  We have $\code(\sigma) = (a_1, \dots, a_n)$ for some ordered set partition $\sigma \in \OP_{n,k}$.
\end{enumerate}
\end{theorem}

 \subsection{Demazure characters and Gr\"obner theory}  Theorem~\ref{coinversion-bijection} shows that
 the collection of monomials
 \begin{equation}
 \AAA_{n,k} := \{ x_1^{a_1} \cdots x_n^{a_n} \,:\, (a_1, \dots, a_n) \in E_{n,k} \}
 \end{equation}
 is in bijection with $\OP_{n,k}$ (or $\WWW_{n,k}$).  When $k = n$, E. Artin proved \cite{Artin} that
 the family 
 $\AAA_n := \AAA_{n,n}$ of `sub-staircase monomials' descends to a basis of 
 the classical coinvariant ring $R_n$.  In this subsection and the next, we will see that 
 $\AAA_{n,k}$ descends to a basis of $R_{n,k}$. 
 We show that $\AAA_{n,k}$ descends to a spanning set of $R_{n,k}$ using Gr\"obner theory; we start 
 by reviewing the ``Gr\"obner basics".
 
 A total order $<$ on the monomials in $\CC[x_1, \dots, x_n]$ is a {\em term order} if
 \begin{itemize}
 \item  for any monomial $m$ we have $1 \leq m$, and
 \item  $m < m'$ implies $m'' m < m'' m'$ for any monomials $m, m', m''$.
 \end{itemize}
 Given a term order $<$, for any nonzero polynomial $f \in \CC[x_1, \dots, x_n]$ denote by 
 $\initial_<(f)$ the largest monomial under $<$ appearing in $f$.

 In this chapter, we shall exclusively use the $\neglex$ term order, which is the lexicographical
term order {\bf with respect to the variable ordering $x_n > \cdots > x_2 > x_1$.}  Explicitly, we have
$x_1^{a_1} \cdots x_n^{a_n} < x_1^{b_1} \cdots x_n^{b_n}$ if there exists $1 \leq i \leq n$ 
so that $a_i < b_i$ and $a_{i+1} = b_{i+1}, \dots, a_n = b_n$.
 
 If $I \subseteq \CC[x_1, \dots, x_n]$
 is an ideal, the corresponding {\em initial ideal} is the monomial ideal
 $\initial_<(I) := \langle \initial(f) \,:\, 0 \neq f \in I \rangle$.  A finite subset $G = \{g_1, \dots, g_r \} \subseteq I$
 of nonzero polynomials in $I$ is a {\em Gr\"obner basis} if 
 $\initial_<(I) = \langle \initial_<(g_1), \dots, \initial_<(g_r) \rangle$.  This condition implies that 
 $I = \langle G \rangle$.  Furthermore, the set 
 \begin{multline}
 \{ m  \,:\, \text{$m$ a monomial in $\CC[x_1, \dots, x_n]$ not contained in $\initial_<(I)$} \} = \\
 \{ m  \,:\, \text{$m$ a monomial in $\CC[x_1, \dots, x_n]$ such that $\initial_<(g) \nmid m$ for all $g \in G$} \}
 \end{multline}
 of monomials not contained in $\initial_<(I)$
  descends a basis for $\CC[x_1, \dots, x_n]/I$ as a $\CC$-vector space.
 This is the {\em standard monomial basis} of $\CC[x_1, \dots, x_n]/I$; it is determined by the term order $<$
 and the ideal $I$.
 
 Although Gr\"obner bases are not unique, there is a canonical choice of Gr\"obner basis for an ideal $I$
 and a term order $<$.  A Gr\"obner basis $G$ is called {\em minimal} if for any distinct elements 
 $g, g' \in G$ we have $\initial_<(g) \nmid \initial_<(g')$.
 The Gr\"obner basis $G$ is {\em reduced} if it  is minimal and, for any distinct $g, g' \in G$,
 no monomial appearing in $g$ is divisible by $\initial_<(g')$.
 Every ideal $I \subseteq \CC[x_1, \dots, x_n]$ has a unique reduced Gr\"obner basis for a given term order.

Gr\"obner bases are of central importance in many computational algorithms in commutative algebra.
However, even for ideals $I \subseteq \CC[x_1 ,\dots, x_n]$ with a nice generating set and of algebraic or geometric 
significance, the polynomials appearing in a Gr\"obner basis can have unpredictable monomials with 
ugly coefficients. Miraculously, the Gr\"obner bases of our ideals $I_{n,k}$ are very structured.
We present the classical case of $R_n$, and then outline the generalization to $R_{n,k}$.

The coinvariant ideal $I_n = \langle e_1, e_2, \dots, e_n \rangle$ has the alternate 
presentation $I_n = \langle h_1, h_2, \dots, h_n \rangle$ in terms of the complete homogeneous 
symmetric polynomials in $x_1, x_2, \dots, x_n$.  From this one can show that 
we have
\begin{equation}
h_{n-i}(x_1, x_2, \dots, x_i) \in I_n
\end{equation}
for $1 \leq i \leq n$.
With respect to the $\neglex$ order $<$, we have $\initial_<(h_{n-i}(x_1, x_2, \dots, x_i))= x_i^{n-i}$ so that the 
standard monomial basis of $R_n$ is contained in 
\begin{equation}
\{ \text{monomials $m$ in $x_1, \dots, x_n \,:\, x_i^{n-i} \nmid m$ for $1 \leq i \leq n$} \} = \AAA_n,
\end{equation}
the family of substaircase monomials.  Since $\dim R_n = n! = |\AAA_n|$, this implies that
$\AAA_n$ is the standard monomial basis of $R_n$ and 
$\{ h_{n-i}(x_1, x_2, \dots, x_i) \,:\, 1 \leq i \leq n \}$ is a Gr\"obner basis of $I_n$. It is easily seen that this 
Gr\"obner basis is reduced.

In order to generalize this story to $R_{n,k}$, we  need an  analogue 
of the  $h_{n-i}(x_1, x_2, \dots, x_i)$.
These will be the {\em Demazure characters} (or {\em key polynomials}) $\kappa_{\gamma}$
indexed by weak compositions (i.e. sequences of nonnegative integers)
$\gamma = (\gamma_1, \dots, \gamma_n)$ of length $n$.
Given $\gamma$, the polynomial $\kappa_{\gamma}$ is defined recursively as follows:
\begin{equation}
\kappa_{\gamma} := \begin{cases}
x_1^{\gamma_1} \cdots x_n^{\gamma_n} & \text{if $\gamma_1 \geq \cdots \geq \gamma_n$} \\
\pi_i( \kappa_{s_i \cdot \gamma} ) & \text{if $\gamma_i < \gamma_{i+1}$}
\end{cases}
\end{equation}
where $s_i \cdot \gamma = (\gamma_1, \dots, \gamma_{i+1}, \gamma_i, \dots, \gamma_n)$ swaps
the $i^{th}$ and $(i+1)^{st}$ parts of $\gamma$, the {\em isobaric divided difference operator}
$\pi_i$ acts on polynomials $f \in \CC[x_1, \dots, x_n]$  by
\begin{equation}
\pi_i(f) := \frac{x_i f - x_{i+1}(s_i \cdot f)}{x_i - x_{i+1}}
\end{equation}
and $s_i \cdot f(x_1, \dots, x_i, x_{i+1}, \dots, x_n) = f(x_1, \dots, x_{i+1}, x_i, \dots, x_n)$.

When $\gamma = (\gamma_1 \leq \cdots \leq \gamma_n)$ is weakly decreasing, the 
Demazure character $\kappa_{\gamma}$ coincides with the Schur polynomial
$s_{\rev(\gamma)}$ corresponding to the reversal of $\gamma$.
Thus, Demazure characters interpolate between monomials and Schur polynomials.
The polynomial $\kappa_{\gamma}$ has a representation-theoretic interpretation as the 
character of a certain indecomposable $B$-module where $B \subseteq GL_n(\CC)$
is the Borel subgroup.
It is not too hard to check inductively that 
\begin{equation}
\label{demazure-leading-term}
\initial_<(\kappa_{\gamma}(x_1, \dots, x_n)) = x_1^{\gamma_1} \cdots x_n^{\gamma_n}
\end{equation}
in the $\neglex$ term order.

\begin{lemma}
\label{demazure-containment}
Let $S \subseteq [n]$ be a subset of size $|S| = n-k+1$, let $\gamma(S) = (\gamma_1, \dots, \gamma_n)$
be the corresponding skip sequence, and let 
$\rev(\gamma(S)) = (\gamma_n, \dots, \gamma_1)$ be its reversal.  We have
\begin{equation*}
\kappa_{\rev(\gamma(S))}(x_1, \dots, x_n) \in \langle e_n, e_{n-1}, \dots, e_{n-k+1} \rangle \subseteq
\ZZ[x_1, \dots, x_n].
\end{equation*}
\end{lemma}

Lemma~\ref{demazure-containment} was proven by Haglund-Rhoades-Shimozono
\cite[Lem. 3.4]{HRS} using an extension of the (dual) Pieri rule to Demazure characters
due to Haglund-Luoto-Mason-van Willigenburg \cite{HLMW} and a sign-reversing involution to explicitly
write $\kappa_{\rev(\gamma(S))}(x_1, \dots, x_n)$ as a polynomial combination of 
$e_n, e_{n-1}, \dots, e_{n-k+1}$.
As a consequence of Lemma~\ref{demazure-containment}, we obtain a spanning set of $R_{n,k}$.

\begin{proposition}
\label{a-spans}
The set $\AAA_{n,k}$ of $(n,k)$-substaircase monomials spans $R_{n,k}$.
\end{proposition}

\begin{proof}
By Lemma~\ref{demazure-containment}, Equation~\eqref{demazure-leading-term}, and the fact that 
the variable powers $x_1^k, x_2^k, \dots, x_n^k$ lie in $I_{n,k}$, we 
see that $\AAA_{n,k}$ contains the standard monomial basis for $R_{n,k}$ with respect to the $\neglex$
term order.
\end{proof}

\subsection{Orbit harmonics} Proposition~\ref{a-spans}
bounds the dimension of $R_{n,k}$ from above. We bound this dimension from below using {\em orbit harmonics},
 a general technique for turning ungraded group actions on finite sets into graded actions on quotient rings.
 Orbit harmonics dates back to at least the work of Kostant \cite{Kostant} for the classical coinvariant ring and was later
 applied by Garsia-Procesi \cite{GP} to Tanisaki quotients and by Garsia-Haiman \cite{GH2} in the context of Macdonald polynomials.
We outline the general approach, and then explain how it applies in our setting.

For any finite subset $Z \subseteq \CC^n$,  let $\II(Z) \subseteq \CC[x_1, \dots x_n]$ be the  ideal
\begin{equation}
\II(Z) := \{ f \in \CC[x_1, \dots, x_n] \,:\, f(\zz) = 0 \text{ for all $\zz \in Z$} \}
\end{equation}
of polynomials which vanish on $Z$.
The {\em coordinate ring} $\CC[x_1, \dots, x_n]/\II(Z)$ is the ring of polynomial maps
$Z \longrightarrow \CC$.
Since $Z$ is finite, Lagrange interpolation guarantees that any function $Z \longrightarrow \CC$
is the restriction of a polynomial map $f \in \CC[x_1, \dots, x_n]$ and 
have
 \begin{equation}
 \label{ideal-space-isomorphism}
\CC[Z] \cong \CC[x_1, \dots, x_n]/\II(Z)
 \end{equation}
 where $\CC[Z]$ is the vector space with basis $Z$.
The ideal $\II(Z) \subseteq \CC[x_1, \dots, x_n]$ is usually not homogeneous, so
the ring $\CC[x_1, \dots, x_n]/\II(Z)$ is usually not graded. However, there is a standard way to produce a graded ideal
from $\II(Z)$ as follows.

Given a nonzero polynomial $f \in \CC[x_1, \dots, x_n]$, let $\tau(f)$ denote the highest degree component of $f$.
That is, if $f = f_d + \cdots + f_1 + f_0$ with $f_i$ homogeneous of degree $i$ and $f_d \neq 0$, we have
$\tau(f) = f_d$.
For any ideal $I \subseteq \CC[x_1, \dots, x_n]$, the {\em associated graded ideal} is
\begin{equation}
\gr \, I := \langle \tau(f) \,:\, f \in I - \{0\} \rangle \subseteq \CC[x_1, \dots, x_n]
\end{equation}
By construction, the ideal $\gr \, I$ is homogeneous so that the quotient
$\CC[x_1, \dots, x_n]/\gr \, I$ is a graded ring.
The rings $\CC[x_1, \dots, x_n]/I$ and $\CC[x_1, \dots, x_n]/\gr \, I$ share a number of features; an important one for us 
is as follows.

\begin{lemma}
\label{associated-graded-lemma}
Let $\BBB \subseteq \CC[x_1, \dots, x_n]$ be a set of homogeneous polynomials and let $I \subseteq \CC[x_1, \dots, x_n]$
be an ideal. If  $\BBB$ descends to a basis of $\CC[x_1, \dots, x_n]/ \gr \, I$ then
$\BBB$ descends to a basis of $\CC[x_1, \dots, x_n]/I$.
\end{lemma}

The isomorphism~\eqref{ideal-space-isomorphism} and 
Lemma~\ref{associated-graded-lemma} imply that for any finite point locus
$Z \subseteq \CC^n$, we have isomorphisms
\begin{equation}
\CC[Z] \cong \CC[x_1, \dots, x_n]/\II(Z) \cong \CC[x_1, \dots, x_n]/ \gr \, \II(Z)
\end{equation}
of vector spaces where the quotient $\CC[x_1, \dots, x_n]/ \gr \, \II(Z)$ is graded.
We upgrade these vector space isomorphisms by incorporating group actions.

Let $G \subseteq GL_n(\CC)$ be a finite matrix group. 
In addition to the natural action of $G$ on $\CC^n$, we have an action of $G$ on $\CC[x_1 ,\dots, x_n]$
by linear substitutions.
If $Z \subseteq GL_n(\CC)$ is $G$-stable,  the ideals 
$\II(Z)$ and $\gr \, \II(Z)$ are  $G$-stable, as well.  If $Z \subseteq \CC^n$ is finite, 
 we have 
 isomorphisms of ungraded $G$-modules 
\begin{equation}
\label{g-module-iso}
\CC[Z] \cong_G \CC[x_1, \dots, x_n]/\II(Z) \cong_G \CC[x_1, \dots, x_n]/\gr \, \II(Z)
\end{equation}
where the quotient $ \CC[x_1, \dots, x_n]/\gr \, \II(Z)$ has the additional structure of a graded $G$-module.
The first isomorphism $\CC[Z] \cong_G \CC[x_1, \dots, x_n]/\II(Z)$ in \eqref{g-module-iso}
is standard; the second isomorphism 
may be justified as follows.

Let $\BBB \subseteq \CC[x_1, \dots, x_n]$ be a set of homogeneous polynomials which descends to a basis 
of $\CC[x_1, \dots, x_n]/\II(Z)$.  By Lemma~\ref{associated-graded-lemma}, the set $\BBB$ also descends to a basis of 
$\CC[x_1, \dots, x_n]/\gr \, \II(Z)$.  The set $\BBB$ decomposes as a disjoint union
\begin{equation}
\BBB = \BBB_0 \sqcup \BBB_1 \sqcup \cdots \sqcup \BBB_r
\end{equation}
where $\BBB_d$ consists of polynomials $f \in \BBB$ of degree $d$.

Fix a group element $g \in G$.  Then $g$ acts on both $\CC[x_1, \dots, x_n]/\II(Z)$ and $\CC[x_1, \dots, x_n]/\gr \, \II(Z)$.
If we express these actions with respect to the basis
$\BBB = \BBB_0 \sqcup \BBB_1 \sqcup \cdots \sqcup \BBB_r$ we obtain matrices of block form
 \begin{equation*}
 \begin{pmatrix}
 A_0 & * & \cdots & * \\
  & A_1 & \cdots & * \\
  & & \ddots & \\
  & &  & A_r
 \end{pmatrix}
  \quad \quad \text{and} \quad \quad
  \begin{pmatrix}
 A_0 &  &  &  \\
  & A_1 &  &  \\
  & & \ddots & \\
  & &  & A_r
 \end{pmatrix}
 \end{equation*}
where the diagonal block $A_d$ has size $|\BBB_d|$ and is the same 
 in either matrix.  In particular, these matrices have the same trace.
 Since two representations of a finite group with the same character are isomorphic, we 
 obtain the second isomorphism $\CC[x_1, \dots, x_n]/\II(Z) \cong_G \CC[x_1, \dots, x_n]/\gr \, \II(Z)$ 
 in \eqref{g-module-iso}.

By varying the point locus $Z$ (and the group $G$, but we mostly consider $G = S_n$),
 a wide family of graded quotient modules  $\CC[x_1, \dots, x_n]/ \gr \, \II(Z)$ can be obtained.
The Hilbert series of $\CC[x_1, \dots, x_n]/\gr \, \II(Z)$ is typically an interesting statistic on the set $Z$.
It will turn out that $R_{n,k}$ may be realized as such a quotient.
Before going further, we give two simple examples of orbit harmonics which demonstrate an algebraic subtlety of this technique.

For our first example, let $G = S_2$ act on $\CC^2$ by reflection across the line $x_1 = x_2$.  
The point locus $Z = \{ (1,0), (1,1), (0,1) \}$ is $S_2$-stable and 
$x_1 (x_1 - 1), x_2 (x_2 - 1), (x_1-1) (x_2-1) \in \CC[x_1, x_2]$ vanish on $Z$.
Consequently, we have $\langle x_1 (x_1 - 1), x_2 (x_2 - 1), x_1 x_2 \rangle \subseteq \II(Z)$ and, taking highest degrees,
we have $\langle x_1^2, x_2^2, x_1 x_2 \rangle \subseteq \gr \, \II(Z)$.
Since $\dim \CC[x_1, x_2]/\gr \, \II(Z) = |Z| = 3 = \dim \CC[x_1, x_2]/\langle x_1^2, x_2^2, x_1 x_2 \rangle$, we see that 
\begin{equation*}
\CC[x_1, x_2]/\gr \, \II(Z) = \CC[x_1, x_2]/\langle x_1^2, x_2^2, x_1 x_2 \rangle.
\end{equation*}
Geometrically, the transformation $\CC[x_1, x_2]/\II(Z) \leadsto \CC[x_1, x_2]/\gr \, \II(Z)$ deforms the reduced
scheme $Z$ to a `triple point' in the origin as shown.
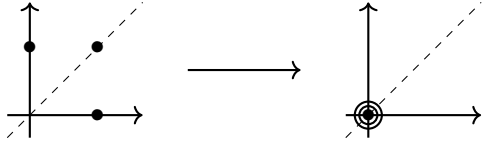
\begin{center}
\begin{tikzpicture}[scale = 0.3]
%\draw[help lines, color=gray!30, dashed] (-4.9,-4.9) grid (4.9,4.9);
\draw[->, thick] (-1,0)--(5,0) node[right]{$$};
\draw[->, thick] (0,-1)--(0,5) node[above]{$$};
\draw[-,dashed] (-1,-1) -- (5,5);
\node at (3,3) {$\bullet$};
\node at (0,3) {$\bullet$};
\node at (3,0) {$\bullet$};

\draw[->, thick] (7,2) -- (12,2);

\draw[->, thick] (14,0)--(20,0) node[right]{$$};
\draw[->, thick] (15,-1)--(15,5) node[above]{$$};
\draw[-,dashed] (14,-1) -- (20,5);
\node at (15,0) {$\bullet$};
\draw [thick] (15,0) circle (0.4);
\draw [thick] (15,0) circle (0.6);
\end{tikzpicture}
\end{center}
The Hilbert series and graded Frobenius image of $\CC[x_1, x_2]/\gr \, \II(Z)$ are easily computed to be
\begin{equation*}
\Hilb( \CC[x_1, x_2]/\gr \, \II(Z); q) = 1 + 2q \quad \quad \grFrob( \CC[x_1, x_2]/\gr \, \II(Z); q) = s_{2} + (s_2 + s_{11}) \cdot q 
\end{equation*}
yielding a graded refinement of the action of $S_2$ on $Z$.
We remark that $\grFrob( \CC[x_1, x_2]/\gr \, \II(Z); q)$ may be also expressed in the Hall-Littlewood basis as
$\widetilde{H}_{11}(\xx;q)  + \widetilde{H}_2(\xx;q) \cdot q$.

For our second example, we again consider the action of $G = S_2$ on $\CC^2$, but take the locus 
$Z' = \{ (2,0), (1,1), (0,2) \}$ as shown below. 
This locus is combinatorially equivalent to the locus $Z$ of the previous paragraph, but the graded ring 
obtained from applying orbit harmonics is different. 
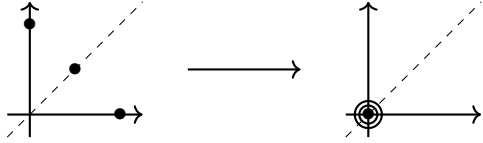
\begin{center}
\begin{tikzpicture}[scale = 0.3]
%\draw[help lines, color=gray!30, dashed] (-4.9,-4.9) grid (4.9,4.9);
\draw[->, thick] (-1,0)--(5,0) node[right]{$$};
\draw[->, thick] (0,-1)--(0,5) node[above]{$$};
\draw[-,dashed] (-1,-1) -- (5,5);
\node at (2,2) {$\bullet$};
\node at (0,4) {$\bullet$};
\node at (4,0) {$\bullet$};

\draw[->, thick] (7,2) -- (12,2);

\draw[->, thick] (14,0)--(20,0) node[right]{$$};
\draw[->, thick] (15,-1)--(15,5) node[above]{$$};
\draw[-,dashed] (14,-1) -- (20,5);
\node at (15,0) {$\bullet$};
\draw [thick] (15,0) circle (0.4);
\draw [thick] (15,0) circle (0.6);
\end{tikzpicture}
\end{center}
Indeed, the polynomials
$x_1(x_1 - 1)(x_1 - 2), x_2 (x_2 - 1) (x_2 - 2), x_1 + x_2 - 2 \in \CC[x_1, x_2]$ vanish on $Z'$ and it is not hard 
to check that $\gr \, \II(Z') = \langle x_1^3, x_2^3, x_1 + x_2 \rangle$ is generated by the top degree components of these polynomials.
From this we have the Hilbert and Frobenius series
\begin{equation*}
\Hilb( \CC[x_1, x_2]/\gr \, \II(Z'); q) = 1 + q + q^2 \quad \quad 
\grFrob(\CC[x_1, x_2]/\gr \, \II(Z'); q) = s_2 + s_{11} \cdot q + s_2 \cdot q^2
\end{equation*}
which are different than those for the ring $\CC[x_1, x_2]/\gr \, \II(Z)$, so we obtain a different graded refinement of the same
combinatorial action of $S_2$.
As they must, the graded Frobenius
images  $\grFrob(\CC[x_1, x_2]/\gr \, \II(Z);q)$ and $\grFrob(\CC[x_1, x_2]/ \gr \, \II(Z');q)$ have the same 
$q \rightarrow 1$ specialization, $2 \cdot s_2 + s_{11}$.
As with the case of $Z$, the Frobenius series for $Z'$ is positive in the $\widetilde{H}$-basis: we have
$\grFrob(\CC[x_1, x_2]/\gr \, \II(Z'); q) = \widetilde{H}_{11}(\xx;q) + q^2 \cdot \widetilde{H}_2(\xx;q)$.
We will see various other instances of Hall-Littlewood positivity of orbit harmonics modules throughout the paper.

 \begin{question}
 \label{orbit-positivity-conjecture}
 For which finite $S_n$-stable subsets $Z \subset \CC^n$ does the expansion of the symmetric function 
 \begin{equation}
 \grFrob(\CC[x_1, \dots, x_n]/\gr \, \II(Z); q)
 \end{equation}
 in the Hall-Littlewood
 basis $\{\widetilde{H}_{\lambda}(\xx;q)\}$ have coefficients in $\ZZ_{\geq 0}[q]$?
 \end{question}

 When $Z$ is a single $S_n$-orbit, the answer to Question~\ref{orbit-positivity-conjecture} is `yes' by the work of Garsia-Procesi \cite{GP};
 more generally, the answer to Question~\ref{orbit-positivity-conjecture} is `yes' for all loci $Z$ considered in this paper.
Griffin~\cite{GriffinTwo} analyzed the case where $Z$ is union of two $S_n$-orbits; when the coordinate sums of these orbits are different,
he proved that $ \grFrob(\CC[x_1, \dots, x_n]/\gr \, \II(Z); q) $ expands positively in the Hall-Littlewood basis.
However, he showed \cite[Ex. 1]{GriffinTwo} that if $n = 3$ and
$Z = S_3 \cdot \{ (4,0,0), (2,1,1) \}$ then Hall-Littlewood positivity fails.
Reineke, Rhoades, and Tewari \cite{RRT} give a locus $Z \subseteq \CC^n$ of size $|Z| = n^{n-2}$ such that 
$\Res^{S_n}_{S_{n-1}} \CC[x_1, \dots, x_n]/\gr \, \II(Z)$ is the coarea-graded parking representation of $S_{n-1}$, but
$\grFrob(\CC[x_1, \dots, x_n]/\gr \, \II(Z); q)$ does not expand positively in the basis $\{\widetilde{H}_{\lambda}(\xx;q)\}$.
Since the locus $Z$ in \cite{RRT} consists of the lattice points in a trimmed permutohedron, it seems that 
the answer to Question~\ref{orbit-positivity-conjecture} is more likely to be `no' when the coordinate sums of the $S_n$-orbits in $Z$
are `too often equal'.

For our next example, we explain how the classical coinvariant ring $R_n$
may be regarded as an orbit harmonics quotient.

\begin{example}
\label{permutation-example}
The regular representation of $S_n$ on $\CC[S_n]$ may be regarded as an action on a point locus.
More specifically, if $\alpha_1, \dots, \alpha_n \in \CC$ are distinct complex numbers we have
$\CC[S_n] \cong \CC[Z]$ where
\begin{equation}
Z = \{ (\alpha_{w(1)}, \dots, \alpha_{w(n)}) \,:\, w \in S_n \}
\end{equation}
is the (regular) orbit of $(\alpha_1, \dots, \alpha_n)$ under coordinate permutation.
Said differently, $Z$ is the vertex set of the permutahedron.

For $1 \leq d \leq n$, the difference
$e_d(x_1, \dots, x_n) - e_d(\alpha_1, \dots, \alpha_n)$ vanishes on $Z$ so that 
$\langle e_1, e_2, \dots, e_n \rangle \subseteq \gr \, \II(Z)$.
To prove that $\langle e_1, e_2, \dots, e_n \rangle = \gr \, \II(Z)$ we argue as follows.

 It is well-known that the polynomials 
$e_1, e_2, \dots, e_n$ form a homogeneous regular sequence in the rank $n$ polynomial ring $\CC[x_1, \dots, x_n]$.
It follows that 
\begin{equation}
\dim \CC[x_1, \dots, x_n]/\langle e_1, e_2, \dots, e_n \rangle = (\deg e_1) \cdot (\deg e_2) \cdots (\deg e_n) = n! = |X|
\end{equation}
which forces $\gr \, \II(Z) = \langle e_1, e_2, \dots, e_n \rangle$.  We derive the ungraded $S_n$-structure
 of the coinvariant ring
\begin{equation}
R_n =  \dim \CC[\xx_n]/\langle e_1, e_2, \dots, e_n \rangle =
\CC[x_1, \dots, x_n]/\gr \, \II(Z) \cong \CC[Z] \cong \CC[S_n],
\end{equation}
recovering a result of Chevalley \cite{Chevalley}.

This argument extends to any complex reflection group 
$G \subseteq GL_n(\CC)$.  The set $Z$ may be taken to be any regular $G$-orbit and the 
elementary symmetric polynomials $e_1, e_2, \dots, e_n$ are replaced with a {\em fundamental system
of invariants} $f_1, f_2, \dots, f_n$ satisfying $\CC[x_1, \dots, x_n]^G = \CC[f_1, f_2, \dots, f_n]$.
\end{example}

Example~\ref{permutation-example} generalizes to express $R_{n,k}$ as an orbit harmonics quotient.
Let $\alpha_1, \dots, \alpha_k \in \CC$ be distinct complex numbers and consider the locus
\begin{equation}
Z_{n,k} := \{ (\alpha_{w(1)}, \dots, \alpha_{w(n)}) \,:\, \text{ $[w(1), \dots, w(n)] \in \WWW_{n,k}$ is a Fubini word} \}
\end{equation}
of points in $\CC^n$. For example, when $n = 3$ and $k =2$ we have
\begin{equation*}
Z_{3,2} = \{ (\alpha_1, \alpha_1, \alpha_2), \, (\alpha_1, \alpha_2, \alpha_1), \, (\alpha_2, \alpha_1, \alpha_1), \, 
(\alpha_1, \alpha_2, \alpha_2), \, (\alpha_2, \alpha_1, \alpha_2), \, (\alpha_2, \alpha_2, \alpha_1) \}.
\end{equation*}
The correspondence $(\alpha_{w(1)}, \dots, \alpha_{w(n)}) \leftrightarrow [w(1), \dots, w(n)]$
is an $S_n$-equivariant bijection between $Z_{n,k}$ and $\WWW_{n,k}$.

\begin{theorem}
\label{orbit-harmonics-identification}
We have $\gr \, \II(Z_{n,k}) = I_{n,k}$ so that $R_{n,k} = \CC[x_1, \dots, x_n]/ \gr \, \II(Z_{n,k}) \cong \CC[\WWW_{n,k}]$
as ungraded $S_n$-modules. Furthermore, the set 
\begin{equation}
\{ \kappa_{\rev(\gamma(S))}(x_1, \dots, x_n) \,:\, S \subset [n] \,, \, |S| = n-k+1 \} \cup
 \{ x_1^k, x_2^k, \dots, x_n^k \}
\end{equation}
forms a Gr\"obner basis of $I_{n,k}$.
The set $\AAA_{n,k}$ of substaircase monomials is the standard monomial basis for $R_{n,k}$ in the $\neglex$
term order.
\end{theorem}

In the range $1 < k < n$, the Gr\"obner basis of Theorem~\ref{orbit-harmonics-identification} becomes reduced
when one removes the $\kappa_{\rev(\gamma(S)}$'s for sets $S$ containing $n$.
The proof of Theorem~\ref{orbit-harmonics-identification} is a more elaborate version of 
Example~\ref{permutation-example}.

\begin{proof}  We begin by listing some polynomials which vanish on $Z_{n,k}$.
For $1 \leq i \leq n$, the $i^{th}$ coordinate of any point $(x_1, \dots, x_n) \in Z_{n,k}$ lies in 
$\{\alpha_1, \dots, \alpha_k\}$.  Said differently, we have
$(x_i - \alpha_1) \cdots (x_i - \alpha_k) \in \II(Z_{n,k})$ so that $x_i^k \in \gr \, \II(Z_{n,k})$ for $1 \leq i \leq n$.
To show that the remaining generators $e_r(x_1, \dots ,x_n)$ of $I_{n,k}$ lie in $\gr \, \II(Z_{n,k})$,
we have a trick at our disposal.

Introduce a new variable $t$ and consider the rational expression
\begin{equation}
\frac{(1 + x_1 t) (1 + x_2 t) \cdots (1 + x_n t)}{(1 + \alpha_1 t)(1+ \alpha_2 t) \cdots (1 + \alpha_k t)}.
\end{equation}
When $(x_1, \dots, x_n) \in Z_{n,k}$ the $k$ factors in the denominator cancel with $k$ factors in the 
numerator, yielding a polynomial in $t$ of degree $n-k$. For $r  > n-k$, we see that
\begin{equation}
0 = \sum_{a + b = r} (-1)^b e_a(x_1, x_2, \dots, x_n) h_b(\alpha_1, \alpha_2, \dots , \alpha_k)
\end{equation}
so that 
$\sum_{a + b = r} (-1)^b e_a(x_1, x_2, \dots, x_n) h_b(\alpha_1, \alpha_2, \dots , \alpha_k) \in \II(Z_{n,k})$ and
taking the highest degree component gives $e_r(x_1, x_2, \dots, x_n) \in \gr \, \II(Z_{n,k})$. We conclude that
\begin{equation}
\label{containment-of-i}
I_{n,k} \subseteq \gr \, \II(Z_{n,k}) 
\end{equation}
so we have a canonical projection
\begin{equation}
\label{projection-onto-r}
\CC[x_1, \dots, x_n]/\gr \, \II(Z_{n,k}) \twoheadrightarrow R_{n,k}.
\end{equation}
Lemma~\ref{demazure-containment} implies that $\dim R_{n,k} \leq k! \cdot \Stir(n,k)$ which forces
\eqref{projection-onto-r} to be an isomorphism and \eqref{containment-of-i} to be an equality.
Proposition~\ref{a-spans} (and its proof) imply that the given set of polynomials is indeed a Gr\"obner
basis of $R_{n,k}$ and that $\AAA_{n,k}$ is the standard monomial basis of $R_{n,k}$.
\end{proof}

\subsection{Short exact sequences and graded module structure}
Theorem~\ref{orbit-harmonics-identification} determines the ungraded $S_n$-module structure
 of $R_{n,k}$.  In this subsection we describe its graded $S_n$-structure.
 This was originally achieved in \cite{HRS} using skewing operators; we outline
 a simpler approach due to Gillespie and Rhoades \cite{GR} using short exact sequences.
 The modules appearing in these sequences are a three-parameter version of $R_{n,k}$.

 \begin{defn}
 \label{rnks-definition}
 For nonnegative integers $k \leq s \leq n$, let $I_{n,k,s} \subset \CC[x_1, \dots, x_n]$ be the ideal
 \begin{equation*}
 I_{n,k,s} := \langle x_1^s, \dots, x_n^s, e_n, e_{n-1}, \dots, e_{n-k+1} \rangle
 \end{equation*}
 and let
 \begin{equation*}
 R_{n,k,s} := \CC[x_1, \dots, x_n]/I_{n,k,s}
 \end{equation*}
 be the corresponding quotient ring.\footnote{The roles of $k$ and $s$ in the ring $R_{n,k,s}$ agree with those in \cite{Griffin} and elsewhere,
 but are switched relative to their original appearance in \cite{HRS}.}
 \end{defn}
 
 We have $R_{n,k,k} = R_{n,k}$ so that the rings in Definition~\ref{rnks-definition} reduce to $R_{n,k}$
 when $k = s$. For $k = 0$, there are no elementary symmetric polynomials
 among the generators of  $I_{n,0,s}$ and the quotient has the simple form
 $R_{n,0,s} = \CC[x_1, \dots, x_n]/\langle x_1^s, \dots, x_n^s \rangle$.
 
 Like the rings $R_{n,k}$, the rings $R_{n,k,s}$ may be analyzed using orbit harmonics.  In particular, 
 let $\alpha_1, \alpha_2, \dots, \alpha_s \in \CC$ be $s$ distinct complex numbers and consider the locus
 $Z_{n,k,s} \subset \CC^n$ given by
 \begin{equation}
 Z_{n,k,s} := \{ (x_1, \dots, x_n) \,:\,  \{ \alpha_1, \dots, \alpha_k \} \subseteq \{ x_1, \dots, x_n \} 
 \subseteq \{ \alpha_1, \dots, \alpha_s \} \}.
 \end{equation}
 In other words, the locus $Z_{n,k,s}$ consists of points $(x_1, x_2, \dots, x_n) \in \CC^n$ whose coordinates
 are drawn from the list $(\alpha_1, \alpha_2, \dots, \alpha_s)$ in which the first $k$ entries 
 $\alpha_1, \dots, \alpha_k$ must occur.
 Using methods similar to those for $X_{n,k}$, orbit harmonics applies to give
 \begin{equation}
 I_{n,k,s} = \gr \, \II(Z_{n,k,s}) \text{ so that }
 R_{n,k,s} = \CC[\xx_n]/I_{n,k,s} \cong \CC[Z_{n,k,s}]
 \end{equation}
 where the isomorphism is of ungraded $S_n$-modules.

 Definition~\ref{rnks-definition} gives the appropriate interpolation between the rings $R_{n,k}$ of study
 and the simpler rings $R_{n,0,s}$.
 This interpolation takes the form of a short exact sequence.
Aside from an easy boundary case, the following is \cite[Lem 6.9]{HRS}.
 
 \begin{lemma}
 \label{short-exact-sequence}
 Let $k < s \leq n$ be nonnegative integers. There exists a short exact sequence of $S_n$-modules
 \begin{equation}
 0 \rightarrow R_{n,k,s-1} \xrightarrow{ \, \, \varphi \, \, } R_{n,k,s}
  \xrightarrow{ \, \, \pi \, \, } R_{n,k+1,s} \rightarrow 0
 \end{equation}
 where the first map is homogeneous of degree $n-k$ and the second is homogeneous of degree 0.
 \end{lemma}
 
The first map $\varphi$ in Lemma~\ref{short-exact-sequence} is multiplication by $e_{n-k}$ while 
the second map $\pi$ is the canonical projection coming from the ideal containment
$I_{n,k,s} \subseteq I_{n,k+1,s}$.
 In the language of symmetric functions, 
 Lemma~\ref{short-exact-sequence} implies
 \begin{equation}
 \label{short-exact-recursion}
 \grFrob(R_{n,k,s};q) = q^{n-k} \cdot \grFrob(R_{n,k,s-1};q) + \grFrob(R_{n,k+1,s};q)
 \end{equation}
 for all nonnegative $k < s \leq n$.  
 The graded $S_n$-structure of
 $R_{n,0,s} = \CC[x_1, \dots, x_n]/\langle x_1^s, \dots, x_n^s \rangle$ 
 is fairly simple, so the recursion \eqref{short-exact-recursion}
 can be solved for $\grFrob(R_{n,k,s};q)$.
 In the case $k = s$, this gives the graded $S_n$-structure of $R_{n,k}$.

 In order to describe the symmetric function $\grFrob(R_{n,k};q)$, we need some terminology.  Recall that 
for a partition $\lambda \vdash n$, a {\em standard Young tableau} $T$ of shape $\lambda$ is a filling of the boxes of $\lambda$
with $1, 2, \dots, n$ which increases across rows and down columns.  
An example standard tableau of shape $(4,3,1) \vdash 8$ is shown below.
\begin{center}
\begin{Young}
 1 & 3 & 4 & 7 \cr
 2 & 5 & 8 \cr
 6
\end{Young}
\end{center}
 We let $\SYT(\lambda)$ denote the family of standard tableaux of shape $\lambda$.  If $T$ is a standard tableau with $n$ boxes, and index
 $1 \leq i \leq n-1$ is a {\em descent} of $T$ if $i$ appears in a strictly higher row than $i+1$; the above tableau has descents $1, 4, 5,$ and $7$.
 The {\em descent number} $\des(T)$ is the number of descents in $T$ and the {\em major index} $\maj(T)$ is the sum of the descents;
 in our example we have $\des(T) = 4$ and $\maj(T) = 1 + 4 + 5 + 7 = 17$.
 We also use the {\em $q$-binomial} and {\em $q$-multinomial coefficients}
 \begin{equation}
 {n \brack k}_q := \frac{[n]!_q}{[k]!_q \cdot [n-k]!_q} \quad \quad 
 {n \brack a_1, \dots, a_r}_q := \frac{[n]!_q}{[a_1]!_q \cdots [a_r]!_q}
 \end{equation}
 where $0 \leq k \leq n$ and $a_1 + \cdots + a_r = n$.
 The following result collects material from \cite[Sec. 6]{HRS}.

 \begin{theorem}
 \label{graded-rnk-structure}
 Let $k \leq n$ be positive integers.  The graded Frobenius image 
 $\grFrob(R_{n,k};q)$ is given by
 \begin{align*}
 \grFrob(R_{n,k};q) &= \sum_{\lambda \vdash n} 
 \left[ \sum_{T \in \SYT(\lambda)}  q^{\maj(T)} \cdot
 {n-\des(T)-1 \brack n-k}_q \right] \cdot s_{\lambda}(\xx) \\
 &= q^{(n-k) \cdot (k-1)} \cdot
  \sum_{\substack{\lambda \vdash n \\ \ell(\lambda) = k}} 
 q^{\sum_i {m_i(\lambda) \choose 2} - \sum_i (i-1)(2 \lambda_i - 1)} \cdot {k \brack m_1(\lambda), \dots, m_n(\lambda)}_q \cdot 
 \widetilde{H}_{\lambda}(\xx;q) \\ &=
 (\rev_q \circ \omega) \Delta'_{e_{k-1}} e_n \mid_{t = 0}
 \end{align*}
 where $m_i(\lambda)$ is the multiplicity of $i$ as a part of $\lambda$.
 \end{theorem}
 
 The three expressions for $\grFrob(R_{n,k};q)$ in Theorem~\ref{graded-rnk-structure}
 are interesting for different reasons.
 The first expression is a tableau formula for the 
 Schur expansion of $\grFrob(R_{n,k};q)$ and directly gives the decomposition of $R_{n,k}$
 into irreducibles.
 The second expression
 shows that $\grFrob(R_{n,k};q)$ 
 expands positively in the Hall-Littlewood basis
 (as in Question~\ref{orbit-positivity-conjecture}).
 The final expression ties $R_{n,k}$ to delta operators.

\section{Spanning line configurations}
\label{Spanning}

Theorem~\ref{graded-rnk-structure} states that the quotient ring $R_{n,k}$ plays the same role for the delta operator expression
$\Delta'_{e_{k-1}} e_n$ that the classical coinvariant algebra $R_n$ plays for $\nabla e_n$.
That is, the ring $R_{n,k}$ is a coinvariant algebra for the Delta Conjecture.
In this section, we describe a variety $X_{n,k}$ of spanning line configurations defined by Pawlowski and Rhoades \cite{PR}
whose cohomology is presented by $R_{n,k}$.  
The variety $X_{n,k}$ is a flag variety for the Delta Conjecture.

Presenting the cohomology of $X_{n,k}$ uses two main geometric tools: affine pavings and Chern classes of vector bundles.
The ring $H^{\bullet}(X_{n,k})$ is generated by the Chern classes $x_1, \dots, x_n$ of certain obvious
 line bundles $\LLL_1, \dots, \LLL_n$ subject only to relations coming from the spanning condition and the Whitney 
Sum Formula.
This `Chern calculus'  also serves to present the cohomology of several related varieties of subspace configurations satisfying 
rank conditions.
We give a quick overview of these concepts with an emphasis on their combinatorial aspects;
 for a more thorough resource, see e.g. \cite{Fulton}.

\subsection{The variety $X_{n,k}$}
A {\em line} in a vector space $V$ is a 1-dimensional linear subspace 
$\ell \subseteq V$.
Let $\PP^{k-1}$ be the $(k-1)$-dimensional projective space of lines in $\CC^k$ and let
$(\PP^{k-1})^n := \PP^{k-1} \times \cdots \times \PP^{k-1}$ be its $n$-fold self product.
Our variety of study is as follows.

\begin{defn}
\label{xnk-definition}
Given $k \leq n$, we let $X_{n,k}$ be the subset of $(\PP^{k-1})^n$ given by
\begin{equation*}
X_{n,k} := \{ (\ell_1, \dots, \ell_n) \,:\, \ell_i \subseteq \CC^k \text{ \em{a line and} } \ell_1 + \cdots + \ell_n = \CC^k \}.
\end{equation*}
\end{defn}

It is not hard to see that $X_{n,k}$ is a Zariski open subvariety of $(\PP^{k-1})^n$ and, as such,
is a smooth complex manifold. 
We will see that 
 $X_{n,k}$ has cohomology presentation $H^{\bullet}(X_{n,k}) \cong R_{n,k}$.
When $k = 1$, we have the one-point space $X_{n,1} = \{* \}$ and $R_{n,1} = \CC$,
so the 
result holds in this case.
 At the other extreme, when $k = n$
we claim that $X_{n,n}$ is homotopy equivalent to the flag variety $\Fl(n)$.  
To see this, we argue as follows.

When $k = n$, for any point $(\ell_1, \dots, \ell_n) \in X_{n,n}$ we have a direct sum 
$\CC^n = \ell_1 \oplus \cdots \oplus \ell_n$.
There is a natural 
projection 
\begin{equation}
\pi: X_{n,n} \twoheadrightarrow \Fl(n)
\end{equation}
which sends the tuple $(\ell_1,  \dots, \ell_n)$ of lines to the flag 
$(V_1, \dots, V_n)$ whose $i^{th}$ 
piece is $V_i = \ell_1 \oplus \cdots \oplus \ell_i$.
 The projection
$X_{n,n} \xrightarrow{ \, \, \pi \, \, } \Fl(n)$ is a fiber bundle with fiber isomorphic to
\begin{equation}
U_n := \{ \text{all lower triangular $n \times n$ complex matrices with 1's on the diagonal} \},
\end{equation}
the lower triangular unipotent subgroup of $GL_n(\CC)$.
Since $U_n$ is affine,
we conclude that $\pi$ is a homotopy equivalence.

The variety $X_{n,k}$ is always smooth, but unlike the flag variety $\Fl(n)$ it is almost never compact.
Indeed, it could not be both.
The Hilbert
series of $R_{n,k}$ is rarely palindromic: for example, 
$\Hilb(R_{3,2};q) = 1 + 3q + 2q^2$.
Therefore, no space with cohomology $R_{n,k}$ can
 be both smooth and 
compact (otherwise it would satisfy Poincar\'e Duality).
In the next section we describe an alternative geometric model of $R_{n,k}$ (and, in fact, a more general
family of rings) which is compact but not smooth.  

The defect of $X_{n,k}$'s failure to be compact is ameliorated from the perspective of $S_n$-actions.
The isomorphism $H^{\bullet}(\Fl(n)) \cong R_n$ is known to hold not only as graded rings,
but also in the category of graded $S_n$-modules. While the $S_n$-structure of $R_n$ is simply variable
permutation, the flag variety $\Fl(n)$ admits no natural action of $S_n$. 
On the other hand, the group $S_n$ acts directly on the variety $X_{n,k}$ by
\begin{equation}
w \cdot (\ell_1, \dots, \ell_n) := (\ell_{w^{-1}(1)}, \dots, \ell_{w^{-1}(n)})
\end{equation}
giving an induced action on cohomology. 
When $k = n$,
the homotopy equivalence $\Fl(n) \simeq X_{n,n}$ induces an action of $S_n$ on $H^{\bullet}(\Fl(n))$.
For general $k$,
this action will be compatible with our presentation
$H^{\bullet}(X_{n,k}) \cong R_{n,k}$.

Our basic strategy for studying $X_{n,k}$ is to make use of {\em compactification}, a general
technique in geometry where one embeds a space of study $X$ into a better-understood compact space $Y$
in such a way that $\overline{X} = Y$. 
If this embedding is sufficiently nice, geometric properties of $Y$ may be transferred to $X$.

The $n$-fold self product $(\PP^{k-1})^n$ may be thought of as the family of all
$n$-tuples of lines $(\ell_1, \dots, \ell_n)$ in $\CC^k$, spanning or otherwise.
The space $(\PP^{k-1})^n$ is compact and has well-understood 
cohomology.
Considering spanning tuples only gives a natural inclusion
\begin{equation}
\iota: X_{n,k} \hookrightarrow (\PP^{k-1})^n
\end{equation}
which (since a generic collection of $n \geq k$ lines will span $\CC^k$)
satisfies  $\overline{X_{n,k}} = (\PP^{k-1})^n$.
We use this compactification to transfer properties of the well-understood product
$(\PP^{k-1})^n$ to $X_{n,k}$.
To do this, we need to show that the inclusion $\iota$ is sufficiently nice. 
This will be made precise in the next subsection using {\em affine pavings}.

\subsection{Affine pavings and Fubini words}
An important method for understanding a space $X$ is to decompose it as a disjoint union
$X = \bigsqcup_i C_i$ of cells $C_i$ which fit together to form a finite CW complex.
Since any finite CW complex is compact and $X_{n,k}$ is not compact, we need a different kind 
of decomposition in our context.

\begin{defn}
\label{affine-paving}
Let $X$ be a complex variety. A filtration of $X$ by Zariski closed subvarieties
\begin{equation*}
\varnothing = Z_0 \subset Z_1 \subset \cdots \subset Z_m = X
\end{equation*}
is an {\em affine paving}
of $X$ if each difference $Z_i - Z_{i-1}$ is isomorphic to a disjoint union
$\bigsqcup_j A_{ij}$ of affine spaces $A_{ij}$.  
The spaces $A_{ij}$ are the {\em cells} of the affine paving.

Given a disjoint union decomposition
$X = \bigsqcup_{i,j} A_{ij}$ of $X$ into affine spaces, we say that 
the $A_{ij}$ {\em induce an affine paving} if they arise as the cells of an affine paving of $X$.
\end{defn}

 As an example, projective space $\PP^{k-1}$ admits an affine paving
\begin{equation}
\label{standard-projective-affine-paving}
\varnothing \subset [ * : 0 :  \cdots : 0] \subset [* : * : \cdots : 0] \subset \cdots \subset 
[* : * : \cdots : *] = \PP^{k-1}
\end{equation}
with $k$ cells, one of each complex dimension $0, 1, \dots, k-1$.
We refer to this as the {\em standard affine paving} of $\PP^{k-1}$.
It is also a CW decomposition of $\PP^{k-1}$.
More generally, the CW decomposition of any flag variety $G/P$ into Schubert cells 
may be regarded as an affine paving.

Varieties with affine pavings enjoy the following two useful properties.
Suppose $X$ is a smooth variety which admits
an affine paving 
$\varnothing = Z_0 \subset Z_1 \subset \cdots \subset Z_m = X$ with 
cells $A_{ij}$.  
Recall that any closed subvariety $Z \subseteq X$ has a corresponding class $[Z]$ in the cohomology 
ring $H^{\bullet}(X)$.
\begin{enumerate}
\item 
The classes $[ \overline{A_{ij}} ]$ of the Zariski closures of these cells form a basis 
of  $H^{\bullet}(X)$.
\item 
If $0 \leq r \leq m$ and $U := X - X_r$, the inclusion map $\iota: U \hookrightarrow X$ induces 
a surjection $\iota^*: H^{\bullet}(X) \twoheadrightarrow H^{\bullet}(U)$ on cohomology.
\end{enumerate}
More explicitly, the map $\iota^*$ sends the classes $[ \overline{A_{ij}} ]$ for which $A_{ij} \cap U = \varnothing$ to zero while mapping
the cells $[ \overline{A_{ij}} ]$ for which $A_{ij} \subset U$ onto a linear basis of $H^{\bullet}(U)$.

Our goal is to prove that $((\PP^{k-1})^n, X_{n,k})$ forms a pair $(X, U)$ as in
(2) above.  
More precisely, we want to find an affine paving 
$\varnothing = Z_0 \subset Z_1 \subset \cdots \subset Z_m = (\PP^{k-1})^n$
for which $X_{n-k} = (\PP^{k-1})^n - X_r$ is a terminal part.

If $X$ and $Y$ are varieties which admit affine pavings with cells $\{ A_{ij} \}$ and $\{ B_{rs} \}$,
the products $A_{ij} \times B_{rs}$ of these cells induce a {\em product paving} of $X \times Y$.
In particular, the standard paving $\eqref{standard-projective-affine-paving}$ of $\PP^{k-1}$
induces a product paving of $(\PP^{k-1})^n$. 
Unfortunately, this paving is not suitable for our purposes:
cells $C$ of this product paving typically satisfy neither $C \subseteq X_{n,k}$ nor $C \cap X_{n,k} = \varnothing$.
For example, if $\mathbf{1} := \mathrm{span} (1, \dots, 1)$ is the line of constant vectors in $\CC^k$,
then $(\mathbf{1}, \dots,  \mathbf{1})$ will always be in the largest cell $C_0$ of the product paving 
and yet $(\mathbf{1}, \dots,  \mathbf{1}) \notin X_{n,k}$ for $1 < k  <n$.
To get around this hurdle, we introduce a nonstandard paving of $(\PP^{k-1})^n$.
In order to describe this paving, we coordinatize as follows.

Write $\Mat_{k \times n}$ for the affine space $\mathbb{A}^{k \times n}$ of $k \times n$ complex matrices 
equipped with an action of $GL_k(\CC) \times T_n$ where $\GL_k(\CC)$ acts on rows and the 
torus $T_n := (\CC^{\times})^n$ acts on columns.
 We introduce nested subsets $\UUU_{n,k} \subseteq \VVV_{n,k}$ of 
$\Mat_{k \times n}$ by
\begin{equation}
\VVV_{n,k} := \{ A \in \Mat_{k \times n} \,:\, A \text{ has no zero columns} \}
\end{equation}
and 
\begin{equation}
\UUU_{n,k} := \{ A \in \Mat_{k \times n} \,:\, A \text{ has no zero columns and is of full rank} \}
\end{equation}
so that $\UUU_{n,n} = GL_n(\CC)$. 
We have the natural identifications
\begin{equation}
(\PP^{k-1})^n = \VVV_{n,k}/T_n \quad \text{and} \quad 
X_{n,k} = \UUU_{n,k}/T_n
\end{equation}
of quotient spaces.
The sets $\VVV_{n,k}$ and $\UUU_{n,k}$
are both closed under the action of the product group $GL_k(\CC) \times T_n$.  
If $U_k \subseteq GL_k(\CC)$ denotes the unipotent subgroup of lower triangular matrices with 1's on the diagonal,
restriction gives an action of $U_k \times T_n$ on these sets.

A variant of row reduction gives convenient
 representatives of the $U_k \times T_n$-orbits
 inside $\VVV_{n,k}$.  These will be certain {\em pattern matrices}
 corresponding to words $w = [w(1), w(2), \dots, w(n)] \in [k]^n$ defined as follows.
A position $1 \leq j \leq n$ in such a word is {\em initial} if  $w(j)$ is the first 
 of its letter in $w$.  For example, in $[2,5,2,5,1]$ the initial positions are $1, 2,$ and $5$.
 
 \begin{defn}
 \label{pattern-matrix-definition}
 Let $w = [w(1), w(2), \dots, w(n)] \in [k]^n$ be a word.  The {\em pattern matrix} 
 $\PM(w)$ is the $k \times n$ matrix with entries in $\{0, 1, * \}$ whose entries are determined 
by the following procedure.
 \begin{enumerate}
 \item Set the entries $\PM(w)_{w(j),j} = 1$ for $j =1, 2, \dots, n$.
 \item If $1 \leq j \leq n$ and $j$ is initial, set $\PM(w)_{i,j} = *$ for all positions $(i,j)$ which are north
 and east of a 1.
 \item If $1 \leq j \leq n$ and $j$ is not initial, set $\PM(w)_{i,j} = *$ for all positions $(i,j)$ which are 
east of an {\bf initial} 1.
 \item Set all other entires of $\PM(w)$ equal to 0.
 \end{enumerate}
 A $k \times n$ matrix $A$ {\em fits the pattern} of $w$ if it can be obtained by replacing the $*$'s in 
 $\PM(w)$ with complex numbers.
 \end{defn}

For example, if $n = 8$, $k = 5$, and $w = [2,4,1,1,4,5,5,4]$ we have
 \begin{equation*}
 \PM[2,4,1,1,4,5,5,4] = 
 \begin{pmatrix}
 0 & 0 & 1 & 1 & 0 & 0 & * & 0\\
 1 & * & 0 & * & * & 0 & * & * \\
 0 & 0 & 0 & 0 & 0 & 0 & 0 & 0\\
 0 & 1 & 0 & * & 1 & 0 & * & 1\\
 0 & 0 & 0 & 0 & 0 & 1 & 1 & 0
 \end{pmatrix}
 \end{equation*}
 
 Given a word $w \in [k]^n$, let $U_k(w) \subseteq U_k$ be the subgroup of lower triangular unipotent
 $k \times k$ matrices $A$ for which $A_{ij} = 0$ whenever $j$ does not appear as a letter of 
 $w$.  In the above example, we have
 \begin{equation*}
 U_5[2,4,1,1,4,5,5,4] = \begin{pmatrix} 1 & 0 & 0 & 0 & 0 \\
 * & 1 & 0 & 0 & 0 \\
 * & * & 1 & 0 & 0 \\
 * & * & 0 & 1 & 0 \\
 * & * & 0 & *& 1 
 \end{pmatrix}
 \end{equation*}
 where the subdiagonal entries of the third column are zero since 3 does not appear in $w$.
 In particular, if every letter in $[k]$ appears in the word $w$, then $U_k(w) = U_k$ is the full lower
 triangular unipotent subgroup.

The following theorem of linear algebra is proven in \cite{PR}.

 \begin{theorem}
 \label{row-reduction-theorem} 
 {\em (Echelon form)}  Let $A \in \VVV_{n,k}$ be a $k \times n$ matrix with no zero columns. The 
matrix $A$ may be written uniquely as
\begin{equation*}
A = u B t
\end{equation*}
$w \in [k]^n$ is a word,
 $u \in U_k(w)$,  $B$ fits the pattern of $w$, and
 $t \in T_n$ is a member of the diagonal torus.
 The matrix $A$ represents a spanning line configuration $(\ell_1, \dots, \ell_n) \in X_{n,k}$ if and only if each of the
 letters $1, 2, \dots, k$ appear in the word $w$.
 \end{theorem}
 
Theorem~\ref{row-reduction-theorem} is best understood by example. Suppose $k = 3, n = 6$, and $A$ is the matrix
\begin{equation*}
A = \begin{pmatrix} 
0 & -1 & 2 & 0 & 0&  -1 \\
2 & 1 & 0 & 0 &3& -1 \\
2 & 3 & -4 & 3 &3& 2 \end{pmatrix}.
\end{equation*}
In order to obtain $B$ from $A$, we process columns from left to right.  The subgroup $U_3 \subseteq GL_3(\CC)$ allows 
for downward row operations yielding
\begin{equation*}
\begin{pmatrix} 
0 & -1 & 2 & 0 & 0 &  -1 \\
2 & 1 & 0 & 0 & 3 & -1 \\
2 & 3 & -4 & 3 & 3 & 2 \end{pmatrix}  \leadsto
\begin{pmatrix}
0 & -1 & 2 & 0 & 0&  -1 \\
2 & 1 & 0 & 0 & 3 &  -1 \\
0 & 2 & -4 & 3 & 0 & 3 \end{pmatrix}   \leadsto
\begin{pmatrix}
0 & -1 & 2 & 0 & 0 & -1 \\
2 & 0 & 2 & 0 & 3 &  -2 \\
0 & 0 & 0  & 3 & 0 & 1 \end{pmatrix}
\end{equation*}
so that $u \in U_3$ is given by
\begin{equation*}
u = \begin{pmatrix}
1 & 0 & 0 \\
1 & 1 & 0 \\
2  & -1 & 1 
\end{pmatrix}.
\end{equation*}
At this stage, we know that $B$ fits the pattern of $w = [2,1,1,3,2,3]$.
The fact that each letter $1, 2, 3$ appears in $w$ corresponds to the fact that the columns of $A$ have the full span $\CC^3$.
Finally, we use the action of $T_6$ on columns to scale the pivot positions to ones, yielding
\begin{equation*}
B = \begin{pmatrix}
0 & 1 & 1 & 0 & 0 & -1 \\
1 & 0 & 1 & 0 & 1 & -2 \\
0 & 0 & 0 & 1 & 0 & 1 
\end{pmatrix}
\end{equation*}
and
\begin{equation*}
t = \mathrm{diag} \left(  2, -1, 2, 3, 3, 1 \right) \in T_6.
\end{equation*}

Theorem~\ref{row-reduction-theorem} gives rise to a disjoint union decomposition of the projective
space product $(\PP_{k-1})^n$.  Given a word $w \in [k]^n$, define 
$\mathring{X}_w \subseteq (\PP^{k-1})^n$ by
\begin{equation*}
\mathring{X}_w := 
U_k(w) \cdot  \{ B  \,:\, \text{$B$ fits the pattern of $w$  } \} \cdot T_n/T_n
\end{equation*}
As sets, we have $(\PP^{k-1})^n = \bigsqcup_{w \in [k]^n} \mathring{X}_w$.  When $k = 2$ and $n = 3$, the eight strata of this decomposition
of $(\PP^1)^3$
look as follows
\begin{equation*}
X_{[2,2,1]} =   
\begin{pmatrix} 1 & 0 \\ * & 1 \end{pmatrix} \cdot \begin{pmatrix} 0 & 0 & 1 \\ 1 & 1 & 0 \end{pmatrix}
\quad
X_{[1,1,2]} = 
\begin{pmatrix} 1 & 0 \\ * & 1 \end{pmatrix} \cdot \begin{pmatrix} 1 & 1 &  *  \\ 0 & 0 & 1 \end{pmatrix}  \quad
X_{[1,2,1]} = 
\begin{pmatrix} 1 & 0 \\ * & 1 \end{pmatrix} \cdot \begin{pmatrix} 1 & * &  1  \\ 0 & 1 & 0 \end{pmatrix}  
\end{equation*}
\begin{equation*}
X_{[2,1,1]} =   
\begin{pmatrix} 1 & 0 \\ * & 1 \end{pmatrix} \cdot \begin{pmatrix} 0 & 1 & 1 \\ 1 & 0 & * \end{pmatrix} \quad
X_{[1,2,2]} = 
\begin{pmatrix} 1 & 0 \\ * & 1 \end{pmatrix} \cdot \begin{pmatrix} 1 & *  &  *   \\ 0 & 1 & 1 \end{pmatrix}  \quad
X_{[2,1,2]} = 
\begin{pmatrix} 1 & 0 \\ * & 1 \end{pmatrix} \cdot \begin{pmatrix} 0 & 1 &  0  \\ 1 & 0 & 1 \end{pmatrix}  
\end{equation*}
\begin{equation*}
X_{[1,1,1]} =   
\begin{pmatrix} 1 & 0 \\ * & 1 \end{pmatrix} \cdot \begin{pmatrix} 1 & 1 & 1 \\ 0 & 0 & 0 \end{pmatrix} \quad
X_{[2,2,2]} = 
\begin{pmatrix} 1 & 0 \\ 0 & 1 \end{pmatrix} \cdot \begin{pmatrix} 0 & 0  &  0   \\ 1 & 1 & 1 \end{pmatrix}.
\end{equation*}
Recall that the  Fubini words $\WWW_{n,k}$ are words $w \in [k]^n$ of length $n$ in which each letter $1, 2, \dots, k$ appears.
The cells $X_{[1,1,1]}$ and $X_{[2,2,2]}$ on the third line are indexed by words $[1,1,1]$ and $[2,2,2]$ which are not Fubini.
Correspondingly, these cells have empty intersection with the subvariety $X_{3,2}$.
Their union $X_{[1,1,1]} \cup X_{[2,2,2]}$ is the diagonal copy $\{ (\ell, \ell, \ell) \,:\, \ell \in \PP^1 \}$
of $\PP^1$ inside $(\PP^1)^3$.  
The remaining six cells indexed by the Fubini words in $\WWW_{3,2}$ form a partition of $X_{3,2} \subset (\PP^1)^3$.

 Theorem~\ref{row-reduction-theorem} is valid over any field.
 Before focusing on the complex geometry of $X_{n,k}$, we give an application over finite fields.
 The {\em dimension} $\dim(w)$ of a Fubini word $w$ is the number of $*$'s in its pattern matrix $\PM(w)$.
 Equivalently, $\dim(w)$ is the complex dimension of $\mathring{X}_w$, less the quantity ${k \choose 2}$ coming from the unitary group $U_k$.
 One can use the $q$-Stirling recursion \eqref{q-stirling-recursion} to verify that
 \begin{equation}
 \sum_{w \in \WWW_{n,k}} q^{\dim(w)} = [k]!_q \cdot \Stir_q(n,k)
 \end{equation}
 so that dimension is a Mahonian statistic on $\WWW_{n,k}$. 
 Theorem~\ref{row-reduction-theorem} has the following enumerative consequence; see \cite{PR}.
  
  \begin{corollary}
The number of $n$-tuples $(\ell_1, \dots, \ell_n)$ of lines in $\FF_q^n$ for which
 $\ell_1 + \cdots + \ell_n = \FF_q^n$ is $q^{{k \choose 2}} \cdot [k]!_q \cdot \Stir_q(n,k)$.
  \end{corollary}

We point out that the $q$-Stirling numbers have arisen in geometry before.
Billey and Cozkun \cite{BC} introduced {\em rank varieties} which are certain closed subvarieties $X(M)$ of the Grassmannian
$\Gr(k,n)$ of $k$-planes in $\CC^n$.  
The generating function $\sum_{M} q^{\dim X(M)}$ for the 
dimensions of these varieties for $n, k$ fixed equals $\Stir_q(n+1,n-k+1)$ \cite[Cor. 4.25]{BC}.
The author is unaware of a direct connection between rank varieties and $X_{n,k}$.

In the classical setting, the Schubert cells $\{ \mathring{X}_w \,:\, w \in S_n \}$ assemble to give a CW-decomposition
of the flag variety $\Fl(n)$.
This cannot be true for the decomposition $X_{n,k} = \bigsqcup_{w \in \WWW_{n,k}} \mathring{X}_w$ since any finite CW complex is compact
but $X_{n,k}$ is not.
The next result says that we have the next best thing: the collection $\{ \mathring{X}_w \,:\, w \in \WWW_{n,k} \}$ form the cells of an affine
paving of $X_{n,k}$.

\begin{theorem}
\label{affine-paving-theorem}
Let $k \leq n$ be positive integers. 
\begin{enumerate}
\item
 For any word $w \in [k]^n$, the subset 
$\mathring{X}_w \subseteq (\PP^{k-1})^n$ is affine with complex dimension equal to the number of 
$*$'s in the pattern matrix $\PM(w)$ plus the dimension of the affine space $U_k(w)$.
\item
 We have set-theoretic disjoint unions
 \begin{equation*}
 (\PP^{k-1})^n = \bigsqcup_{w \in [k]^n} \mathring{X}_w \quad \text{and} \quad
 X_{n,k} = \bigsqcup_{w \in \WWW_{n,k}} \mathring{X}_w
 \end{equation*}
 \item
 The cells $\mathring{X}_w$ induce an affine paving 
 $\varnothing = X_0 \subset X_1 \subset \cdots \subset X_m = (\PP^{k-1})^n$ of 
 $(\PP^{k-1})^n$ for which $X_{n,k} = (\PP^{k-1})^n - X_i$ for some $i$.
 \end{enumerate}
\end{theorem}

Items (1) and (2) of Theorem~\ref{affine-paving-theorem}  are direct consequences of
Theorem~\ref{row-reduction-theorem}.
Item (3) requires a bit more work; the basic idea is as follows.

For any position $(i,j)$ in a $k \times n$ matrix $A$, we 
define
\begin{equation}
\mathrm{rk}(i,j,A) := \text{rank of the northwest $i \times j$ submatrix of $A$}
\end{equation}
Since these northwest ranks are invariant under column scaling, we have a function
$\mathrm{rk}(i,j,-)$ on $(\PP^{k-1})^n$.
This gives a disjoint union decomposition
\begin{equation}
\label{omega-decomposition}
(\PP^{k-1})^n = \bigsqcup_{\rho} \Omega_{\rho}
\end{equation}
where the disjoint union ranges over all functions $\rho: [k] \times [n] \rightarrow \ZZ_{\geq 0}$
and 
\begin{equation}
\label{omega-rho-definition}
\Omega_{\rho} := \{ A \cdot T_n \,:\, \mathrm{rk}(i,j,A) = \rho(i,j) \text{ for all positions $(i,j)$ } \}
\end{equation}
One shows that the nonempty $\Omega_{\rho}$ decompose as a product 
 $\mathbb{A}^{d} \times \PP^{d_1} \times \cdots \times \PP^{d_r}$ of an affine space 
 with projective spaces of various dimension depending on $\rho$.
 The disjoint union decomposition in Theorem~\ref{affine-paving-theorem} (2) refines
 the decomposition \eqref{omega-decomposition} and the cells $\mathring{X}_w$ 
 contained in a fixed $\Omega_{\rho}$ trace out the standard product paving of its decomposition
  $\mathbb{A}^{d} \times \PP^{d_1} \times \cdots \times \PP^{d_r}$.

\subsection{Chern classes and cohomology presentation}  Our next task is to leverage the 
linear algebra of the last subsection to compute the cohomology ring $H^{\bullet}(X_{n,k})$.
To perform this computation, we need some basic facts about vector bundles.

Let $X$ be a complex algebraic variety. For $r \geq 0$, a {\em complex vector bundle}
 of rank $r$ over $X$ is an algebraic variety $\EEE$ equipped with a surjective morphism
 $\pi: \EEE \twoheadrightarrow X$ such that, for each $x \in X$, the fiber 
 $\EEE_x := \pi^{-1}(x)$ has the structure of a complex vector space.
 The morphism $\EEE \twoheadrightarrow X$ is required to be {\em locally trivial}
 in the sense that, for all $x \in X$, there is a neighborhood $U$ of $x$ 
and an isomorphism $\pi^{-1}(U) \xrightarrow{\, \sim \, } \CC^r \times U$
 making the following diagram commute
  \begin{center}
  \begin{tikzpicture}
  \node at (0,0) (U)  {$U$};
  \node at (1.5,1) (C) {$\CC^r \times U$};
    \node at (-1.5,1) (pi) {$\pi^{-1}(U)$};
    \node at (0,1.15) (sim) {\begin{scriptsize} $\sim$ \end{scriptsize}};
    
    \draw [ ->> ] (C) -- (U);
    \draw [ ->> ] (pi) -- (U);
    \draw [->] (pi) -- (C);
  \end{tikzpicture}
  \end{center}
  where the map $\CC^r \times U \twoheadrightarrow U$ is the canonical projection.
  A vector bundle $\LLL$ of rank $1$ is called a {\em line bundle}.

  Vector space operations on fibers can be used to construct new bundles from old. 
  The {\em direct sum} (or {\em Whitney sum}) of two vector bundles $\EEE$
  and $\FFF$ over $X$ is defined by 
  \begin{equation}
  (\EEE \oplus \FFF)_x := \EEE_x \oplus \FFF_x \quad \quad x \in X
  \end{equation}
  so that $\mathrm{rank}(\EEE \oplus \FFF) = \mathrm{rank}(\EEE) + \mathrm{rank}(\FFF)$.
  The {\em dual} of a vector bundle $\EEE$ over $X$ has fibers
  \begin{equation}
  \EEE^*_x := (\EEE_x)^* =  \mathrm{Hom}(\EEE_x, \CC) \quad \quad x \in X
  \end{equation}
  so that $\mathrm{rank}(\EEE^*) = \mathrm{rank}(\EEE)$.

A vector bundle $\EEE$ over $X$ can give geometric information about $X$.
For $1 \leq i \leq r$,  the {\em Chern class}
\begin{equation}
c_i(\EEE) \in H^{2i}(X)
\end{equation}
lives in the cohomology ring of $X$.
The {\em total Chern class} is the generating function of these individual Chern classes
\begin{equation}
c(\EEE) := 1 + c_1(\EEE) + c_2(\EEE) + \cdots + c_d(\EEE) \in H^{\bullet}(X)
\end{equation}
which is a typically inhomogeneous element of cohomology. In the special
case where $\EEE = \LLL$ is a line bundle, the total Chern class has the 
simple form $c(\LLL) = 1 + c_1(\LLL)$.

For $r \geq 0$, the {\em trivial vector bundle} of rank $r$ over $X$ is $\EEE = \CC^r \times X$ 
equipped with the canonical projection $\CC^r \times X \twoheadrightarrow X$.
By standard abuse, we denote this trivial bundle by $\CC^r$. 
The definition of this bundle has nothing to do with $X$, and thus carries no cohomological data:
we have $c(\CC^r) = 1$.
When $X$ is a moduli space of complex vector spaces, natural choices of vector bundles $\EEE$ 
can give more information.

 \begin{example}
 \label{projective-space-example}
 Let $\PP^{k-1}$ be the variety of lines through the origin in $\CC^k$.  The {\em tautological line bundle}
 $\LLL$ over $\PP^{k-1}$ has fiber $\ell^* = \mathrm{Hom}(\ell, \CC)$
 over a point $\ell \in \PP^{k-1}$.  If $x = c_1(\LLL)$, we have the cohomology presentation
 \begin{equation*}
 H^{\bullet}(\PP^{k-1}) = \CC[x]/\langle x^k \rangle
 \end{equation*}
as a truncated polynomial ring.
 \end{example}
 
 Recall that the {\em K\"unneth Theorem} gives an isomorphism
 \begin{equation}
 \label{kunneth}
 H^{\bullet}(X \times Y) \cong H^{\bullet}(X) \otimes H^{\bullet}(Y)
 \end{equation}
 for any spaces $X$ any $Y$.  Therefore, the $n$-fold self-product $(\PP^{k-1})^n$ has cohomology presentation
 \begin{equation}
 H^{\bullet}(\PP^{k-1})^n = \CC[x_1, \dots, x_n]/\langle x_1^k, \dots, x_n^k \rangle
 \end{equation}
 where $x_i = c_1(\LLL_i)$, the Chern class of the tautological line bundle over the $i^{th}$ copy of 
 $\PP^{k-1}$.
 
 We shall bootstrap Example~\ref{projective-space-example} to present the cohomology of $X_{n,k}$.
 In order to do this, we will need a piece of Chern calculus.  
 Given a variety $X$, a {\em short exact sequence}
 \begin{equation}
 0 \rightarrow \EEE' \rightarrow \EEE \rightarrow \EEE'' \rightarrow 0
 \end{equation}
of vector bundles over  $X$ consists of
 morphisms  $\EEE' \rightarrow \EEE \rightarrow \EEE''$ of varieties commuting with 
the maps to $X$ such that
  for each point $x \in X$, the induced maps  
 $\EEE'_x \rightarrow \EEE_x \rightarrow \EEE''_x $ form a short exact sequence of vector spaces.
The direct sum gives a short exact sequence
 $ 0 \rightarrow \EEE' \rightarrow \EEE' \oplus \EEE'' \rightarrow \EEE'' \rightarrow 0$
  as expected, but more complicated examples are possible.
 
 Given a short exact sequence as above, the
 Chern classes of $\EEE, \EEE',$ and $\EEE''$ are related in the ring $H^{\bullet}(X)$.
 This relationship is most cleanly stated
 in terms of total Chern classes:
 \begin{equation}
 \label{add-across-exact}
 c(\EEE) = c(\EEE') \cdot c(\EEE'')
 \end{equation}
 When $\EEE = \EEE' \oplus \EEE''$, the corresponding relation
\begin{equation}
\label{whitney-sum}
c(\EEE \oplus \FFF) = c(\EEE) \cdot c(\FFF)
\end{equation}
inside $H^{\bullet}(X)$ is called the {\em Whitney sum formula}.
The Whitney sum formula is a key ingredient for presenting the cohomology of $X_{n,k}$.

\begin{theorem}
\label{x-cohomology-presentation}
For $1 \leq i \leq n$, let $\LLL_i$ be the vector bundle over $X_{n,k}$ whose fiber over a point 
$(\ell_1, \dots, \ell_n)$ is the dual space $\ell_i^*$.  
The cohomology ring $H^{\bullet}(X_{n,k})$ has presentation
\begin{equation*}
H^{\bullet}(X_{n,k}) = R_{n,k} =
 \CC[x_1, \dots, x_n]/\langle x_1^k, \dots, x_n^k, e_n, e_{n-1}, \dots, e_{n-k+1} \rangle
\end{equation*}
where $x_i \leftrightarrow c_1(\LLL_i)$.
\end{theorem}

\begin{proof}
Let us make the abbreviation $y_i := c_1(\LLL_i) \in H^{\bullet}(X_{n,k})$.
By the definition of $X_{n,k}$, for any point $(\ell_1, \dots, \ell_n) \in X_{n,k}$ we have a linear surjection
\begin{equation}
\ell_1 \oplus \cdots \oplus \ell_n \twoheadrightarrow \CC^k
\end{equation} 
given by 
$(v_1, \dots, v_n) \mapsto v_1 + \cdots + v_n$.  Dualizing this surjection yields an injection of vector bundles
$(\CC^k)^* \hookrightarrow \LLL_1 \oplus \cdots \oplus \LLL_n$ over $X_{n,k}$ which may be completed
to a short exact sequence
\begin{equation}
0 \rightarrow (\CC^k)^* \rightarrow \LLL_1 \oplus \cdots \oplus \LLL_n \rightarrow \EEE \rightarrow 0
\end{equation}
where $(\CC^k)^*$ is a trivial rank $k$ bundle and $\EEE$ has rank $n-k$.  Applying 
Equations~\eqref{add-across-exact} and \eqref{whitney-sum}, we have
\begin{equation}
(1 + y_1) \cdots (1 + y_n)  = c((\CC^k)^*) \cdot c(\EEE) = 1 \cdot c(\EEE) = c(\EEE)
\end{equation}
where $c((\CC^k)^*) = 1$ since $(\CC^k)^*$ is a trivial bundle. As a member of the graded
ring $H^{\bullet}(X)$,
the element $c(\EEE)$ has degree $\leq 2(n-k)$ so that 
\begin{equation}
e_d(y_1, \dots, y_n) = 0  \text{ whenever $d > n-k$}
\end{equation}
since each $y_i = c_1(\LLL_i)$ has degree 2.
Furthermore, since $\LLL_i$ is the pullback to $X_{n,k}$ of the tautological line bundle over the $i^{th}$
factor in $(\PP^{k-1})^n$, we have 
\begin{equation}
y_i^k = 0 \text{ for $i = 1, 2, \dots, n$. }
\end{equation}
Consequently, we have a well-defined homomorphism 
\begin{equation}
\varphi: R_{n,k} \rightarrow H^{\bullet}(X_{n,k})
\end{equation}
given by $\varphi: x_i \mapsto y_i$.

We claim that $\varphi$ is surjective. 
If $\iota: X_{n,k} \hookrightarrow (\PP^{k-1})^n$ is the inclusion, Theorem~\ref{affine-paving-theorem}
implies that the induced map 
\begin{equation}
\iota^*: H^{\bullet}(\PP^{k-1})^n \twoheadrightarrow H^{\bullet}(X_{n,k})
\end{equation}
on cohomology is a surjection. 
Example~\ref{projective-space-example} and the subsequent discussion show that the 
cohomology classes of the tautological line bundles over the $n$ factors 
in $(\PP^{k-1})^n$ generate the source ring 
$H^{\bullet}(\PP^{k-1})^n$.
Consequently, the Chern classes $y_1, \dots, y_n$
generate the target ring $H^{\bullet}(X_{n,k})$ and $\varphi$ is a surjection.

Finally, we show that $\varphi$ is an isomorphism.  Theorem~\ref{affine-paving-theorem} implies that 
$\{ \mathring{X}_w \,:\, w \in \WWW_{n,k} \}$ induces an affine paving of $X_{n,k}$, so that 
$H^{\bullet}(X_{n,k})$ has vector space 
dimension  $|\WWW_{n,k}|$, the number of length $n$ Fubini words with maximum 
letter $k$.
Theorem~\ref{orbit-harmonics-identification} shows that the domain $R_{n,k}$ of $\varphi$
has the same dimension, forcing the epimorphism $\varphi$ to be an isomorphism.
%By Lemma~\ref{free-basis} $R_{n,k}^{\ZZ}$ is also a free $\ZZ$-module of rank $k! \cdot \Stir(n,k)$. 
%Since $\varphi$ is a surjective homomorphism between two free $\ZZ$-modules of the same rank,
%it must be an isomorphism.
\end{proof}

We can consider
Theorem~\ref{x-cohomology-presentation} for varying $n$ and $k$.
The cohomology rings $H^{\bullet}(\Fl(n))$ of the flag varieties $\Fl(n)$ exhibit favorable stability properties as $n \rightarrow \infty$
(see \cite{Fulton}). Analogously,
we can consider $\mathrm{ind}$-varieties formed by the glueing $X_{n,k}$ along either of the sequences $(n,k) \leadsto (n+1, k)$
or $(n,k) \leadsto (n+1,k+1)$ which preserve the condition $n \geq k$ necessary for $X_{n,k}$ to be nonempty.
We refer the reader to \cite{PR, PRR} for details.

\begin{remark}
\label{coefficient-remark}
{\em (From $\CC$ to $\ZZ$.)}
Theorem~\ref{x-cohomology-presentation}, as well as the other geometric results presented in this chapter, hold over $\ZZ$ as well as $\CC$.
We outline the basic argument used in this setting
for deducing an integral cohomology presentation from a complex cohomology presentation.

Let $I_{n,k}^{\ZZ} \subseteq \ZZ[x_1, \dots, x_n]$ be the ideal in $\ZZ[x_1, \dots, x_n]$ with the same generators as $I_{n,k}$:
\begin{equation}
I_{n,k}^{\ZZ} := \langle e_n, e_{n-1}, \dots, e_{n-k+1}, x_1^k, x_2^k, \dots , x_n^k \rangle \subseteq \ZZ[x_1, \dots, x_n]
\end{equation}
and denote the corresponding quotient ring by 
\begin{equation}
R_{n,k}^{\ZZ} := \ZZ[x_1, \dots, x_n]/I_{n,k}^{\ZZ}.
\end{equation}
We claim that the family $\AAA_{n,k}$ of substaircase monomials in $x_1, \dots, x_n$ descends to a  $\ZZ$-basis of $R_{n,k}^{\ZZ}$.
Linear independence over $\ZZ$ follows from linear independence over $\CC$.  To show spanning, one can use the Gr\"obner-theoretic
part of Theorem~\ref{orbit-harmonics-identification} together with explicit witness relations \cite[Lem. 3.4]{HRS}
which realize the relevant Demazure characters $\kappa_{\rev(\gamma(S))} \in I_{n,k}^{\ZZ}$, or an explicit straightening algorithm
appearing in Tanisaki \cite{Tanisaki} and
Garisa-Procesi \cite{GP} which was adapted by Griffin
\cite{Griffin} to this setting.
In particular, the ring $R_{n,k}^{\ZZ}$ is a free $\ZZ$-module of rank $k! \cdot \Stir(n,k)$.

The affine paving result Theorem~\ref{affine-paving-theorem} implies that the integral cohomology ring $H^{\bullet}(X_{n,k};\ZZ)$ is also
 a free $\ZZ$-module of rank $k! \cdot \Stir(n,k)$, with basis given by the classes $[X_w]$ of the closures $X_w$
 of cells $\mathring{X}_w$ indexed by Fubini words $w \in \WWW_{n,k}$.  Theorem~\ref{affine-paving-theorem} also guarantees that the 
  map $\iota^*: H^{\bullet}(X_{n,k};\ZZ) \twoheadrightarrow H^{\bullet}((\PP^{k-1})^n;\ZZ)$ on integral cohomology induced by 
  $\iota: X_{n,k} \hookrightarrow (\PP^{k-1})^n$ is a surjection.  The same Whitney sum reasoning as in the proof of Theorem~\ref{x-cohomology-presentation}
  yields a surjective map 
  \begin{equation}
  \varphi: R^{\ZZ}_{n,k} \twoheadrightarrow H^{\bullet}(X_{n,k};\ZZ)
  \end{equation}
  of $\ZZ$-algebras.  To see that $\varphi$ is an isomorphism, we use the result that any epimorphism between two $\ZZ$-modules of the same finite rank
  is automatically an isomorphism (a fact easily proven using Smith Normal Form).
\end{remark}

\subsection{Fubini word Schubert polynomials}
Given a Fubini word $w$, let $X_w \subseteq X_{n,k}$ be the closure of $\mathring{X}_w$.
Theorem~\ref{affine-paving-theorem} guarantees that the classes $[X_w]$ of these cells forms a 
basis of $H^{\bullet}(X_{n,k})$ as $w$ varies over $\WWW_{n,k}$.  
In this subsection we describe explicit polynomial representatives
$\symm_w$ for the $[X_w]$ under the 
cohomology presentation $H^{\bullet}(X_{n,k}) = R_{n,k}$.
The $\symm_w$ will specialize to the usual Schubert polynomials when $k = n$ and $w \in S_n$ is a permutation.

In order to define our generalized Schubert polynomials, we will need some combinatorial definitions 
related to words.
A word $v = [v(1), \dots, v(n)]$ is {\em convex} if it contains no subword of the form $\cdots i \cdots j \cdots i \cdots$
where $i \neq j$.
The {\em convexification}  of an arbitrary word $w = [w(1), \dots,  w(n)]$ is the unique
convex word $\conv(w) = [v(1), \dots, v(n)]$ 
obtained by rearranging the letters of $w$ in which the initial letters appear
in the same order.
The {\em sorting permutation} $\sort(w) \in S_n$ of $w$ is the unique Bruhat minimal $u \in S_n$ for which
$[w(u(1)), \dots , w(u(n))] = \conv(w)$.
The {\em standardization} of a convex Fubini word $v = [v(1), \dots, v(n)] \in \WWW_{n,k}$ 
is the permutation 
$\st(v) \in S_n$ whose one-line notation is obtained by replacing the non-initial letters in $v$,
from left to right, with $k+1, k+2, \dots, n$.

\begin{defn}
Let $w \in \WWW_{n,k}$ be a Fubini word. Define a polynomial $\symm_w \in \ZZ[x_1, \dots, x_n]$ by 
\begin{equation*}
\symm_w := \sort(w)^{-1} \cdot \symm_{\st(\conv(w))}
\end{equation*}
where $\symm_{\st(\conv(w)}$ is the usual Schubert polynomial attached to the permutation
$\st(\conv(w)) \in S_n$.
\end{defn}

As an example, consider the Fubini word $w = [2,1,2,1,3,3,2,3] \in \WWW_{8,3}$.
The convexification of $w$ is $\conv(w)= [2,2,2,1,1,3,3,3]$  and  
is $\sort(w) = [1, 3 , 7 , 2  , 4 , 5 , 6 , 8 ] \in S_8$ is the Bruhat-minimal permutation which sorts $\conv(w)$
into $w$.
The standardization of $\conv(w)$ is  the permutation
$\st(\conv(w)) = [2,4 , 5, 1, 6 , 3, 7 ,  8] \in S_8$
whose Schubert polynomial is
\begin{equation*}
\symm_{\st(\conv(w))} = x_1^2 x_2^2 x_3^2 + x_1^2 x_2^2 x_3 x_4 + x_1^2 x_2 x_3^2 x_4 + 
x_1 x_2^2 x_3^2 x_4 + x_1^2 x_2^2 x_3 x_5 + x_1^2 x_2 x_3^2 x_5 + x_1 x_2^2 x_3^2 x_5.
\end{equation*}
In order to obtain $\symm_w$ itself, we apply the permutation 
$\sort(w)^{-1} = [ 4 , 1 , 6 , 2 , 3 , 5 , 7 , 8 ]$ 
to the subscripts in the above polynomial yielding
\begin{equation*}
\symm_w = 
x_4^2 x_1^2 x_6^2 + x_4^2 x_1^2 x_6 x_2 + 
x_4^2 x_1 x_6^2 x_2 +
x_4 x_1^2 x_6^2 x_2 +
x_4^2 x_1^2 x_6 x_3 +
x_4^2 x_1 x_6^2 x_3 +
x_4 x_1^2 x_6^2 x_3.
\end{equation*}

The usual permutation-indexed Schubert polynomials represent closures of Schubert cells
in the flag variety $\Fl_n$.
This extends to the setting of $X_{n,k}$; the following is a result of Pawlowski and Rhoades \cite{PR}.

\begin{theorem}
\label{schubert-represent}
For any $w \in \WWW_{n,k}$, the cohomology class $[X_w]$ is represented by $\symm_w$
under the cohomology presentation $H^{\bullet}(X_{n,k}) = R_{n,k}$
in Theorem~\ref{x-cohomology-presentation}. In particular, the set 
$\{ \symm_w \,:\, w \in \WWW_{n,k} \}$ descends to a basis of $R_{n,k}$.
\end{theorem}

Theorem~\ref{schubert-represent} is proven using Fulton's theory of degeneracy loci and matrix 
Schubert varieties.
When $w$ is convex, one observes that the closure of $\mathring{X}_w$ inside $(\PP^{k-1})^n$
is an appropriate rank variety $\Omega_{\rho}$
as in \eqref{omega-rho-definition}.
Degeneracy locus theory applies to show that the representative of $[X_w]$ is as claimed.
If $w$ is not convex, consider the Bruhat-minimal permutation $\sort(w) \in S_n$ which carries $w$ to 
its convexification $\conv(w)$. 
The permutation $\sort(w)$ has a reduced factorization $\sort(w) = s_{i_1} \cdots s_{i_r}$ where 
each $s_{i_j} = (i_j, i_j+1)$ interchanges at most one initial letter.  It follows from the definitions that 
$\sort(w) \cdot \mathring{X}_{\conv(w)} = \mathring{X}_w$ as sets, and Theorem~\ref{x-cohomology-presentation}
allows us to reduce to the convex setting.

Unlike in the case of the flag variety $\Fl(n)$, the Schubert basis of Theorem~\ref{schubert-represent}
does not have positive structure constants in general.
For example, if $n = 4$ and $k = 3$ we have
\begin{equation*}
\symm_{[1,1,2,3]} \cdot \symm_{[1,2,3,2]} = - \symm_{[1,1,3,2]} + 2 \symm_{[2,2,1,3]}
\end{equation*}
inside $R_{4,3}$.

\subsection{Variations}
The ideas used to present the cohomology of $X_{n,k}$ apply to more general moduli 
spaces of spanning configurations. 
In each case, the cohomology of the space in question is generated by the Chern classes of a natural
family of
vector bundles, subject only to relations coming from the Whitney Sum Formula.
We give three examples of this paradigm.

Our first example, due to Pawlowski and Rhoades \cite{PR},
 gives a geometric interpretation of the rings $R_{n,k,s}$ of Definition~\ref{rnks-definition}.
Fix three integers $k \leq s \leq n$ and let $\pi: \CC^s \twoheadrightarrow \CC^k$
be the projection which forgets the last $s-k$ coordinates.  Let $X_{n,k,s}$ be the open
subvariety of $(\PP^{s-1})^n$ given by
\begin{equation}
X_{n,k,s} := \left\{ (\ell_1, \dots, \ell_n) \,:\, 
\begin{array}{c} \text{each $\ell_i$ is a line in $\CC^s$ and} \\ \pi(\ell_1 + \cdots + \ell_n) = \CC^k  \end{array} 
  \right\}.
\end{equation}
This space reduces to $X_{n,k}$ when $k = s$.  Let $\LLL_i$ be the line bundle over $X_{n,k,s}$
whose fiber over $(\ell_1, \dots, \ell_n)$ is the dual space $\ell_i^*$.

Just as was the case for $X_{n,k}$,
we have an inclusion $\iota: X_{n,k,s} \hookrightarrow (\PP^{s-1})^n$ and a nonstandard affine paving of 
$(\PP^{s-1})^n$ in which $(\PP^{s-1})^n - X_{n,k,s}$ forms an initial part.
Thus, the cohomology map $\iota^{*}: H^{\bullet}((\PP^{s-1})^n) \rightarrow H^{\bullet}(X_{n,k,s})$ 
 is a surjection and the ring $H^{\bullet}(X_{n,k,s})$ is generated by the
 Chern classes $x_i := c_1(\LLL_i) \in H^2(X_{n,k,s})$
 for $i = 1, 2, \dots, n$.
 The number of cells contained in $X_{n,k,s}$ is the number of $s$-block ordered 
set partitions $(B_1 \mid \cdots \mid B_s)$ of $[n]$ in which the last
$s-k$ blocks $B_{k+1}, B_{k+2}, \dots, B_s$ are allowed
to be empty; these objects index a  basis of $H^{\bullet}(X_{n,k,s})$.

The Chern class relations $x_i^s = 0$ already hold over the projective space product $(\PP^{s-1})^n$.  
For any point $(\ell_1, \dots, \ell_n) \in X_{n,k,s}$ we have a surjection
\begin{equation}
\ell_1 \oplus \cdots \oplus \ell_n \twoheadrightarrow \CC^k
\end{equation}
given by $(v_1, \dots, v_n) \mapsto \pi(v_1 + \cdots + v_n)$. Dualizing gives a short exact sequence
\begin{equation}
0 \rightarrow (\CC^k)^* \rightarrow \LLL_1 \oplus \cdots \oplus \LLL_n \rightarrow \EEE \rightarrow 0
\end{equation}
where $(\CC^k)^*$ is trivial of rank $k$ and $\EEE$ has rank $n-k$. The Whitney Sum Formula implies
that the terms of degree $> n-k$ in the product $(1 + x_1) \cdots (1 + x_n)$ vanish.
Consequently, we have a surjection
\begin{equation}
\label{nks-surjection}
R_{n,k,s} \twoheadrightarrow
H^{\bullet}(X_{n,k,s}).
\end{equation}
As before, we see that the domain $R_{n,k,s}$ and codomain $H^{\bullet}(X_{n,k,s})$ are vector spaces
of the same dimension, so we have an isomorphism
$H^{\bullet}(X_{n,k,s}) \cong R_{n,k,s}$.

For our second example, due to Rhoades and Wilson \cite{RWLine},
we consider spanning configurations of lines in which 
an initial set
is required to be linearly independent.
For parameters $r \leq k \leq n$, we consider the open subvariety
\begin{equation}
X_{n,k}^{(r)} := \left\{
(\ell_1, \dots, \ell_n) \,:\, 
\begin{array}{c}
\text{$\ell_i$ a line in $\CC^k$, } 
\text{$\ell_1 + \cdots + \ell_n = \CC^k$, and} \\
\text{the sum $\ell_1 \oplus \cdots \oplus \ell_r$ is direct}
\end{array}
\right\}
\end{equation}
of $(\PP^{k-1})^n$.
When $r = 1$ this space reduces to $X_{n,k}$.

We define the line bundles $\LLL_i = \ell_i^*$  over $X_{n,k}^{(r)}$ and their Chern classes 
$x_i = c_1(\LLL_i) \in H^{\bullet}(X_{n,k}^{(r)})$ as before.
The same paving used to study the embedding $X_{n,k} \subseteq (\PP^{k-1})^n$
shows that we have a cohomology surjection
$\iota^*: H^{\bullet}((\PP^{k-1})^n) \twoheadrightarrow H^{\bullet}(X_{n,k}^{(r)})$ 
induced by  
$\iota: X_{n,k}^{(r)} \hookrightarrow (\PP^{k-1})^n$ so that the $x_i$ generate 
$H^{\bullet}(X_{n,k}^{(r)})$.

We examine relations among the Chern classes $x_i \in H^{\bullet}(X_{n,k}^{(r)})$.
The identities $x_i^k = 0$ already hold over $(\PP^{k-1})^n$.
The same exact sequence of vector bundles used to study $H^{\bullet}(X_{n,k})$ applies 
to show $e_d(x_1, \dots, x_n) = 0$ for $d > n-k$.  
Since the first 
$r$ members of a point $(\ell_1, \dots, \ell_n) \in X_{n,k}^{(r)}$ have full $r$-dimensional span,
vector addition gives an injection
\begin{equation}
\ell_1 \oplus \cdots \oplus \ell_r \hookrightarrow \CC^k
\end{equation}
which leads to a short exact sequence
\begin{equation}
0 \rightarrow \FFF \rightarrow (\CC^k)^* \rightarrow \LLL_1 \oplus \cdots \oplus \LLL_r \rightarrow 0
\end{equation}
of vector bundles over $X_{n,k}^{(r)}$ where the middle bundle is trivial and $\FFF$ has rank $k-r$.
The Whitney Sum Formula implies
\begin{equation}
c(\FFF) \cdot (1 + x_1) \cdots (1 + x_r) = c ((\CC^k)^*) = 1,
\end{equation}
so that terms of degree $> k-r$ in the expression
\begin{equation}
\frac{1}{(1+x_1) \cdots (1+x_r)} = \sum_{d \geq 0} (-1)^d \cdot h_d(x_1, \dots, x_r)
\end{equation}
vanish in $H^{\bullet}(X_{n,k}^{(r)})$.
In summary, if we introduce the ideal 
\begin{equation}
I_{n,k}^{(r)} := 
\langle x_1^k, x_2^k, \dots, x_n^k \rangle + 
\langle e_d(x_1, x_2, \dots, x_n) \,:\, d > n-k \rangle + 
\langle h_d(x_1, x_2, \dots, x_r) \,:\, d > k-r \rangle
\end{equation}
in $\CC[x_1, \dots, x_n]$ generated by these relations, we have a surjection
\begin{equation}
\label{nkr-surjection}
\varphi: \CC[x_1, \dots, x_n]/I_{n,k}^{(r)} \twoheadrightarrow H^{\bullet}(X_{n,k}^{(r)}).
\end{equation}
Orbit harmonics and paving theory show that the domain and codomain of 
\eqref{nkr-surjection} are vector spaces of the same dimension,
so  \eqref{nkr-surjection} is an isomorphism.

For our final example (see \cite{RhoadesSpan})
we consider configurations of subspaces of arbitrary dimension $d \geq 1$.
Recall that $\Gr(d,k)$ denotes the Grassmannian of $d$-dimensional subspaces 
$W \subseteq \CC^k$.
Extending this notation, if $\alpha = (\alpha_1, \dots, \alpha_n)$ is any list of $n$
 integers $1 \leq \alpha_i \leq k$, we write
$\Gr(\alpha,k)$ for the product
\begin{equation}
\Gr(\alpha,k) := \Gr(\alpha_1,k) \times \cdots \times \Gr(\alpha_n,k) = 
\{ (W_1, \dots, W_n) \,:\, W_i \subseteq \CC^k, \, \, \dim W_i = \alpha_i \}.
\end{equation}
The spanning subspace configurations
\begin{equation}
X_{\alpha,k} := \{ (W_1, \dots, W_n) \,:\, W_1 + \cdots + W_n = \CC^k \}
\end{equation}
in $\Gr(\alpha,k)$ form an open subvariety which is nonempty if and only if 
$\alpha_1 + \cdots + \alpha_n \geq k$.

The varieties $X_{\alpha,k}$ generalize classical objects in type A Schubert calculus.
\begin{itemize}
\item
If $\alpha_i = k$ for any $i$, the space $X_{\alpha,k} = \Gr(\alpha,k)$ is a 
product of Grassmannians.  
\item
If $\alpha_1 + \cdots + \alpha_n = k$, 
the assignment $(W_1, W_2, \dots, W_n) \mapsto (W_1, W_1 + W_2, \dots, W_1 + \cdots + W_n)$
gives a homotopy equivalence $X_{\alpha,k} \simeq \Fl(\alpha)$
between $X_{\alpha,k}$ and the partial flag variety $\Fl(\alpha)$ indexed by $\alpha$.
At the level of algebraic groups, this equivalence is the canonical surjection 
$GL_k(\CC)/L_{\alpha} \rightarrow GL_k(\CC)/P_{\alpha}$ where $L_{\alpha}$ and $P_{\alpha}$
are the Levi and parabolic subgroups of $GL_k(\CC)$ indexed by $\alpha$, respectively.
\end{itemize}
When $\alpha = (1, \dots, 1)$ consists of $n$ copies of $1$, the space $X_{\alpha,k}$ specializes
to $X_{n,k}$.

For $\alpha = (\alpha_1, \dots, \alpha_n)$ arbitrary, we let $\EEE_i$ be the vector bundle over $X_{\alpha,k}$
whose fiber over $(W_1, \dots, W_n)$ is the dual space $W_i^*$.  The bundle $\EEE_i$ has rank $\alpha_i$.
The total Chern class $c(\EEE_i)$ factors 
\begin{equation}
c(\EEE_i) = 1 + c_1(\EEE_i) + c_2(\EEE_i) + \cdots + c_{\alpha_i}(\EEE_i) = 
(1 + x_{i,1})(1 + x_{i,2}) \cdots (1 + x_{i,\alpha_i})
\end{equation}
inside  $H^{\bullet}(X'_{\alpha,k})$ where $X'_{\alpha,k}$ is a flag extension of $X_{\alpha,k}$ and  
 $x_{i,1}, x_{i,2}, \dots, x_{i,\alpha_i} \in H^2(X'_{\alpha,k})$ are the {\em Chern roots} of $\EEE_i$.
Writing $N := \alpha_1 + \cdots + \alpha_n$ for the sum of the $\alpha_i$, we refer to 
the total list of these Chern roots as 
\begin{equation}
(x_1, \dots, x_N) := (x_{1,1}, \dots, x_{1,\alpha_1}, x_{2,1}, \dots, x_{2,\alpha_2}, \dots, 
x_{n,1}, \dots, x_{n,\alpha_n}).
\end{equation}
Any polynomial in the variables $(x_1, \dots, x_N)$ which is invariant under 
$S_{\alpha} = S_{\alpha_1} \times \cdots \times S_{\alpha_n}$ is an element of the subring
$H^{\bullet}(X_{\alpha,k}) \subseteq H^{\bullet}(X'_{\alpha,k})$.

Write $\iota: X_{\alpha,k} \hookrightarrow \Gr(\alpha,k)$ for the  embedding of $X_{\alpha,k}$
into the Grassmannian product.
The Schubert cell decomposition of the Grassmannian induces a product paving of $\Gr(\alpha,k)$,
but this is not adapted to the map $\iota$. However,
the echelon form of Theorem~\ref{row-reduction-theorem} generalizes to yield
 a nonstandard affine paving
of $\varnothing = Z_0 \subset Z_1 \subset \cdots \subset Z_m = \Gr(\alpha,k)$ 
where $\Gr(\alpha,k) - X_{\alpha,k} = Z_j$ for some $j$. Consequently, we have a surjection
\begin{equation}
\iota^*: H^{\bullet}(\Gr(\alpha,k)) \twoheadrightarrow H^{\bullet}(X_{\alpha,k}).
\end{equation}
The vector bunde $\EEE_i$ is the pullback to $X_{n,k}$ of the tautological bundle over the $i^{th}$ factor
$\Gr(\alpha_i,k)$ of $\Gr(\alpha,k)$.  Since $\iota^*$ is a surjection, we conclude that 
$H^{\bullet}(X_{\alpha,k})$ is generated by the Chern classes the bundles
$\EEE_1, \dots, \EEE_n$.
Our next task is to write down relations satisfied by the Chern roots $x_1, \dots, x_N$ of these bundles.

Given any configuration $(W_1, \dots, W_n) \in X_{\alpha,k}$,
 for each $i$ we have a short exact sequence of vector spaces
\begin{equation}
0 \rightarrow W_i \rightarrow \CC^k \rightarrow \CC^k/W_i \rightarrow 0 
\end{equation}
which dualizes to a short exact sequence
\begin{equation}
0 \rightarrow \FFF_i \rightarrow (\CC^k)^* \rightarrow \EEE_i \rightarrow 0
\end{equation} 
of vector bundles where $\FFF_i$ has rank $k - \mathrm{rank}(\EEE_i) = k - \alpha_i$.
Since $c(\FFF_i) c(\EEE_i) = c ((\CC^k)^*) = 1$, the rational expression
\begin{equation*}
c(\EEE_i)^{-1} = \frac{1}{(1 + x_{i,1})(1 + x_{i,2}) \cdots (1 + x_{i,\alpha_i})}
\end{equation*}
is a polynomial of degree $\leq k - \alpha_i$ so that $h_d(x_{i,1}, x_{i,2}, \dots, x_{i,\alpha_i}) = 0$ inside  
$H^{\bullet}(X_{\alpha,k})$ whenever $d > k - \alpha_i$. These relations also hold over the Grassmannian
$\Gr(\alpha,k)$.  Over $X_{\alpha,k}$ the spanning condition yields an additional short exact sequence
\begin{equation}
0 \rightarrow (\CC^k)^* \rightarrow \EEE_1 \oplus \cdots \oplus \EEE_n \rightarrow \FFF \rightarrow 0
\end{equation}
where $\FFF$ has rank $N - k$.  Reasoning as in the case of $X_{n,k}$, we have 
$e_d(x_1, \dots, x_N) = 0$ inside $H^{\bullet}(X_{\alpha,k})$ whenever $d > N-k$.
These relations suffice to present the cohomology of $X_{\alpha,k}$.
An orbit harmonics argument leads to a presentation
\begin{equation}
\label{subspace-cohomology-presentation}
H^{\bullet}(X_{\alpha,k}) = (\CC[x_1, \dots, x_N]/I_{\alpha,k})^{S_{\alpha}}
\end{equation}
as a $S_{\alpha}$-invariant subring
where $I_{\alpha,k} \subseteq \CC[x_1, \dots, x_N]$ is the ideal 
\begin{equation}
I_{\alpha,k} :=   
\langle e_d(x_1, x_2, \dots, x_N) \,:\, d > N-k \rangle +
\sum_{i = 1}^n \langle h_{d_i}(x_{i,1}, x_{i,2}, \dots, x_{i,\alpha_i}) \,:\, d_i > k - \alpha_i \rangle.
\end{equation}
When $\alpha = (1, \dots, 1)$ is a sequence of $n$ copies of $1$, the presentation
\eqref{subspace-cohomology-presentation} is the assertion
$H^{\bullet}(X_{n,k}) = R_{n,k}$ of Theorem~\ref{x-cohomology-presentation}.

\section{Generalized Springer fibers}
\label{Generalized}

The Garsia-Procesi-Tanisaki rings are quotients $R_{\lambda}$
of $\CC[x_1, \dots, x_n]$ which depend on a partition $\lambda \vdash n$.
Like the rings $R_{n,k}$, they are instances of orbit harmonics quotients.
Geometrically, the $R_{\lambda}$ present \cite{HS} the cohomology of the Springer fibers 
$\BBB_{\lambda}$.

In his Ph.D. thesis, Griffin \cite{Griffin} introduced a common generalization $R_{n,\lambda,s}$
of the two families of rings $R_{n,k}$ and $R_{\lambda}$.
Many algebraic properties of $R_{n,k}$ have been generalized to $R_{n,\lambda,s}$.
Griffin, Levinson, and Woo \cite{GLW} defined a  {\em $\Delta$-Springer fiber}
$Y_{n,\lambda,s}$ whose cohomology is presented by
$R_{n,\lambda,s}$.  Unlike the variety $X_{n,k}$ of spanning line configurations (which is 
smooth but not compact), the varieties $Y_{n,\lambda,s}$ are compact but not smooth.
As such, they give a different geometric model for the symmetric function $\Delta'_{e_{k-1}} e_n$
at $t = 0$.
We recall the classical story of Springer fibers, and then discuss their recent generalization.

\subsection{Springer fibers}
Given a partition $\lambda \vdash n$, let $\BBB_{\lambda}$ be the 
associated {\em Springer fiber}  \cite{Springer} defined as follows.
There is a natural action of $GL_n(\CC)$ on the flag variety $\Fl(n)$ given by
\begin{equation}
X \cdot (V_0 \subset V_1 \subset \cdots \subset V_{n-1} \subset V_n) =
(X V_0 \subset X V_1 \subset \cdots \subset X V_{n-1} \subset X V_n).
\end{equation}
For any invertible matrix $X$, the set $\{ V_{\bullet} \in \Fl(n) \,:\, X \cdot V_{\bullet} = V_{\bullet} \}$
of flags fixed by $X$ forms a closed subvariety of $\Fl(n)$.  If $X = U_{\lambda}$ is a unipotent 
$n \times n$ matrix of Jordan type $\lambda$, the Springer fiber
\begin{equation}
\BBB_{\lambda} := \{ V_{\bullet} \in \Fl(n) \,:\, U_\lambda \cdot V_{\bullet} = V_{\bullet} \}
\end{equation}
is the associated fixed space.  For example, when $\lambda = (1^n)$ we have $U_{\lambda} = I$
so that $\BBB_{(1^n)} = \Fl(n)$.  The only fixed point of $U_{(n)}$ 
is the standard flag
 $0 \subset \langle e_1 \rangle \subset \langle e_1, e_2 \rangle 
 \subset \cdots \subset \langle e_1, \dots, e_n \rangle$
so that $\BBB_{(n)} = \{ * \}$ is a single point.
The space $\BBB_{\lambda}$ is compact, but is typically singular.

The cohomology ring of $\BBB_{\lambda}$ has a nice presentation. 
Given a subset $S \subseteq [n]$, let $e_d(S)$ be the elementary symmetric polynomial of degree $d$
in the restricted variable set $\{ x_i \,:\, i \in S \}$.  In particular, we have $e_d(S) = 0$ whenever $d > |S|$.
We write 
$\lambda$ and 
\begin{equation}
m_{i,n}(\lambda) := \lambda_n' + \lambda_{n-1}' + \cdots  +\lambda'_{n-i+1}
\end{equation}
for the number of boxes in the Young diagram of $\lambda$ which do not lie in the first $n-i$ columns.
The {\em Tanisaki ideal}  is
\begin{equation}
I_{\lambda} := \langle e_d(S) \,:\,  S \subseteq [n], \, \, d > n - m_{|S|,n}(\lambda) \rangle \subseteq 
\CC[x_1, \dots, x_n].
\end{equation}
We write $R_{\lambda} := \CC[x_1, \dots, x_n]/I_{\lambda}$ for the corresponding quotient ring.
Hotta and Springer \cite{HS} presented the cohomology of $\BBB_{\lambda}$ as
\begin{equation}
\label{springer-cohomology-presentation}
H^{\bullet}(\BBB_{\lambda}) = R_{\lambda}.
\end{equation}

As with the flag manifold $\Fl(n)$, the symmetric group does not act on the variety 
$\BBB_{\lambda}$ in any natural way.
Despite this, the {\em Springer construction} gives a symmetric group action on the cohomology
of $\BBB_{\lambda}$.  
With respect to the presentation 
\eqref{springer-cohomology-presentation},
 this $S_n$-action is subscript permutation.
 Hotta and Springer \cite{HS} 
 proved that the top-degree piece $H^{\mathrm{top}}(\BBB_{\lambda})$ carries the irreducible representation
$S^{\lambda}$ of $S_n$, giving a geometric model for irreducible $S_n$-modules.

Garsia-Procesi and Tanisaki \cite{GP, Tanisaki} studied
 the $S_n$-action on $R_{\lambda}$ from a combinatorial point of view.
For distinct parameters $\alpha_1, \alpha_2, \dots  \in \CC$ and a partition $\lambda \vdash n$,
let $Z_{\lambda} \subset \CC^n$ be the locus of points $(x_1, \dots, x_n)$ in which the number $\alpha_i$
appears with multiplicity $\lambda_i$ among the coordinates.  For example, we have
$Z_{(2,1)} = \{ (\alpha_1, \alpha_1, \alpha_2), (\alpha_1, \alpha_2, \alpha_1), (\alpha_2, \alpha_1, \alpha_1) \}$.
It is evident that $Z_{\lambda}$ is closed under the coordinate permuting action of $S_n$ and carries a copy
of the parabolic coset 
representation $\CC[S_n/S_{\lambda}]$.
The graded ideal $\gr \, \II(Z_{\lambda}) \subseteq \CC[x_1, \dots, x_n]$ 
arising from orbit harmonics was computed \cite{GP, Tanisaki} to be
\begin{equation}
\gr \, \II(Z_{\lambda}) = I_{\lambda}
\end{equation}
so that  as an ungraded $S_n$-module, the ring $R_{\lambda}$ is an orbit harmonics quotient:
\begin{equation}
R_{\lambda} = 
\CC[x_1, \dots, x_n]/\gr \, \II(Z_{\lambda}) \cong \CC[S_n/S_{\lambda}].
\end{equation}
At the level of Frobenius images, we have
\begin{equation}
\label{ungraded-frobenius-image}
\Frob(R_{\lambda}) = h_{\lambda}.
\end{equation}
Garsia and Procei \cite{GP} refined \eqref{ungraded-frobenius-image} by showing
\begin{equation}
\grFrob(R_{\lambda};q) = \widetilde{H}_{\lambda}(\xx;q)
\end{equation}
so that the graded character of $R_{\lambda}$ is a Hall-Littlewood polynomial.

\subsection{The rings $R_{n,\lambda,s}$}
We are ready to give Griffin's common generalization of the rings $R_{n,k}$
and $R_{\lambda}$.
These rings depend on a partition $\lambda$ 
with at most $n$ boxes and an integer $s \geq \ell(\lambda)$.

\begin{defn}
\label{griffin-rings}
{\em (Griffin \cite{Griffin})}
Let $n$ and $s$ be positive integers and let $\lambda = (\lambda_1 \geq \cdots \geq \lambda_s \geq 0)$
be a partition with $\leq s$ positive parts
with $| \lambda | = \lambda_1 + \cdots + \lambda_s \leq n$.  Let 
$I_{n,\lambda,s} \subseteq \CC[x_1, \dots, x_n]$ be the ideal
\begin{equation}
I_{n,\lambda,s} = \langle e_d(S) \,:\, S \subseteq [n], \, \, d > n - m_{|S|,n}(\lambda) \rangle +
\langle x_1^s, x_2^s, \dots, x_n^s \rangle
\end{equation}
where $m_{i,n}(\lambda) = \lambda_n' + \lambda_{n-1}' + \cdots + \lambda_{n-i+1}'$.
Write 
\begin{equation}
R_{n,\lambda,s} := \CC[x_1, \dots, x_n]/I_{n,\lambda,s}
\end{equation}
for the corresponding quotient ring.
\end{defn}

The rings $R_{n,\lambda,s}$ have the following interesting specializations.
\begin{itemize}
\item When $\lambda = (1^n)$ is a single column of length $n$, for any $s \geq n$ the ring 
$R_{n,(1^n),s}$ is the classical coinvariant ring $R_n$.
\item When $\lambda = (1^k)$ is a single column of length $k \leq n$ and $s = k$, we have
$R_{n,(1^k),k} = R_{n,k}$.  More generally, for $s \geq k$, we have $R_{n,(1^k),s} = R_{n,k,s}$.
\item When $|\lambda| = n$, we recover the Tanisaki quotient
$R_{n,\lambda,s} = R_{\lambda}$ for any $s \geq \ell(\lambda)$.
\end{itemize}

The algebra of $R_{n,\lambda,s}$ is described by the combinatorics of ordered set partition-like
objects.  More precisely, given $(n,\lambda,s)$ as in Definition~\ref{griffin-rings}, let 
$\OP_{n,\lambda,s}$ be the family of length $s$ sequences 
$(B_1 \mid \cdots \mid B_s)$ of pairwise disjoint subsets of $[n]$ such that 
$B_1 \sqcup \cdots \sqcup B_s = [n]$ and the set $B_i$ has at least $\lambda_i$ elements.
Thus
$\OP_{n,(1^k),k} = \OP_{n,k}$ are the ordered set partitions we met before.
Observe that a block $B_i$ of $(B_1 \mid \cdots \mid B_s) \in \OP_{n,\lambda,s}$ is allowed
to be empty whenever $\lambda_i = 0$.

The symmetric group $S_n$ acts on the set $\OP_{n,\lambda,s}$ of generalized ordered set partitions.
Griffin extended the orbit harmonics techniques of \cite{GP} and \cite{HRS} to prove the 
following result.

\begin{theorem}
\label{triple-ungraded-structure}
We have the isomorphism $R_{n,\lambda,s} \cong \CC[\OP_{n,\lambda,s}]$ of ungraded 
$S_n$-modules.
\end{theorem}

Griffin refined \cite{GriffinDelta} 
Theorem~\ref{triple-ungraded-structure} by giving
a
 positive expansion of $\grFrob(R_{n,\lambda,s};q)$
into the basis $\widetilde{H}_{\mu}(\xx;q)$ of Hall-Littlewood polynomials, giving another positive instance of
Question~\ref{orbit-positivity-conjecture}.
Finding the graded $S_n$-structure of $R_{n,\lambda,s}$ uses substantially different
(and more involved) methods than those used for $R_{n,k}$.

\subsection{The $\Delta$-Springer fibers}
Like the quotient rings $R_{\lambda}$, the rings $R_{n,\lambda,s}$ also present cohomology rings of varieties
$Y_{n,\lambda,s}$.  When $\lambda \vdash n$ and $s = \ell(\lambda)$, these are precisely the Springer fibers
$\BBB_{\lambda}$.
In general, the variety $Y_{n,\lambda,s}$ consists of certain `initial' partial flags 
$(V_0 \subset V_1 \subset \cdots \subset V_n)$ where $\dim V_i = i$ sitting inside a 
high-dimensional vector space 
$\CC^K$.

Let $\lambda \vdash k$ be a partition with at most $s$ parts.  We  form a larger partition
\begin{equation*}
\tilde{\lambda} := (n-k+\lambda_1, n-k+ \lambda_2, \dots, n-k+\lambda_s)
\end{equation*}
by adding $n-k$ to every part of 
$\lambda = (\lambda_1, \dots, \lambda_s)$, including any trailing zeros.
The partition $\tilde{\lambda}$ is a partition of $K := s(n-k) + k$.
We let $X_{\tilde{\lambda}} \in M_K(\CC)$ be a nilpotent endomorphism of $\CC^K$
of Jordan type $\tilde{\lambda}$.
Griffin, Levinson, and Woo used $X_{\tilde{\lambda}}$ to 
define the following delta-extension of the Springer fiber.

\begin{defn}
{\em (Griffin-Levinson-Woo \cite{GLW})}
For a positive integer $n$ and  a partition $\lambda$ of $k \leq n$ having $\leq s$ parts,
the {\em $\Delta$-Springer fiber} is the variety
$Y_{n,\lambda,s}$ given by
\begin{equation*}
Y_{n,\lambda,s} := 
\{ V_{\bullet} = (V_0 \subset V_1 \subset \cdots \subset V_n) \,:\, V_i \subseteq \CC^K, \,\, 
\dim V_i = i,  \, \,  X_{\tilde{\lambda}} V_i \subseteq V_i, \, \, \text{{\em and}} \, \,  \, X^{n-k}_{\tilde{\lambda}} \CC^K \subseteq V_n \}.
\end{equation*}
\end{defn}

The $\Delta$-Springer fiber
 $Y_{n,\lambda,s}$ is a closed subvariety of the  partial flag variety 
$\Fl(1^n,K-n)$.
As such, the variety $Y_{n,\lambda,s}$ is compact, but rarely smooth.
Intuitively,
the condition $X^{n-k}_{\tilde{\lambda}} \CC^K \subseteq V_n$ on the largest subspace
in flags $V_{\bullet} \in \Fl(1^n,K-n)$ means that the space $Y_{n,\lambda,s}$
is constructed out of Springer fibers.
This is made precise in \cite{GLW} where it is shown that $Y_{n,\lambda,s}$ can be filtered with Springer fibers 
crossed with affine spaces.

\begin{theorem}
\label{y-cohomology-presentation}
{\em (Griffin-Levinson-Woo \cite{GLW})}
We have the presentation $H^{\bullet}(Y_{n,\lambda,s}) = R_{n,\lambda,s}$.
In this presentation,
the variable $x_i$ represents the Chern class $c_1(\LLL_i)$ of the line 
bundle $\LLL_i$ whose fiber over $V_{\bullet} = (V_0 \subset \cdots \subset V_n)$ 
is the dual space $(V_i/V_{i-1})^*$.
\end{theorem}

The proof of Theorem~\ref{y-cohomology-presentation} uses 
some of the same ideas as the spanning lines case of Theorem~\ref{x-cohomology-presentation},
but is significantly more difficult.
Giving a full treatment is beyond our scope, but we sketch the main ideas.

Coordinate permutation yields
an embedding $\iota: Y_{n,\lambda,s} \hookrightarrow Y_{n,\varnothing,s}$
of $Y_{n,\lambda,s}$ into the space $Y_{n,\varnothing,s}$ labelled by the 
`empty partition' $\varnothing = (0, \dots, 0)$ consisting of $s$ zeros.
A  combinatorial affine paving of $Z_0 \subset Z_1 \subset \cdots \subset Z_m$ 
of $Y_{n,\varnothing,s}$ is constructed for which $\iota(Y_{n,\lambda,s}) = Z_i$ for some $i$.
By virtue of this paving, the induced map
 $\iota^*: H^{\bullet}(Y_{n,\varnothing,s}) \twoheadrightarrow H^{\bullet}(Y_{n,\lambda,s})$ 
 is a surjection.

 An argument using iterated projective bundles shows that
 $Y_{n,\varnothing,s}$ has the same cohomology as $(\PP^{s-1})^n$, i.e.
 \begin{equation}
 H^{\bullet}(Y_{n,\varnothing,s}) = \CC[x_1, \dots, x_n]/\langle x_1^s, \dots, x_n^s \rangle
 \end{equation}
 where $x_i = c_1(\LLL_i)$; by naturality $x_i^s$ also vanishes in 
 $H^{\bullet}(Y_{n,\lambda,s})$.
 Checking that the relevant elementary symmetric polynomials $e_d(S)$ vanish in 
 $H^{\bullet}(Y_{n,\lambda,s})$ is  more involved and uses work of 
 Brundan and Ostrik \cite{BO} on Spaltenstein varieties.
 This given, the generators of $I_{n,\lambda,s}$ vanish in $H^{\bullet}(Y_{n,\lambda,s})$ and
 we have a map of rings
 $\varphi: R_{n,\lambda,s} \rightarrow H^{\bullet}(Y_{n,\lambda,s})$.
 The surjectivity of $\iota^*: H^{\bullet}(Y_{n,\varnothing,s}) \twoheadrightarrow H^{\bullet}(Y_{n,\lambda,s})$
 implies that $\varphi$ is an epimorphism.
 Since the domain and target of $\varphi$ are vector spaces of the same dimension,
 $\varphi$ is an isomorphism.

\section{Future Directions}
\label{Future}

In this final section, we indicate directions for future research on the algebraic combinatorics of delta operators.
In particular, we use anticommuting variables to give models for the full symmetric function $\Delta'_{e_{k-1}} e_n$ rather than just its 
$t = 0$ specialization.
The material in this section is more algebraic  than geometric, but the rich connections between $DR_n$, Hilbert schemes, and affine Springer fibers 
\cite{CO, HaimanVanish}
suggest that geometry will ultimately play a large role in this story.

\subsection{Superspace, the Fields Conjecture, and the Zabrocki Conjecture}  
The diagonal coinvariant ring $DR_n$ is obtained from the classical coinvariant ring $R_n$
by adding in a new set of commuting variables. 
In the past few years, various authors have considered coinvariant quotients involving anticommuting 
variables.

For $n \geq 0$, the {\em superspace} ring of rank $n$ is the unital $\CC$-algebra $\Omega_n$
generated by the $2n$ symbols $x_1, \dots, x_n, \theta_1, \dots, \theta_n$ subject to the relations
\begin{equation}
x_i x_j = x_j x_i \quad \quad x_i \theta_j = \theta_j x_i \quad \quad \theta_i \theta_j = - \theta_j \theta_i
\end{equation}
for $1 \leq i, j \leq n$. 
This algebra admits a decomposition
\begin{equation}
\Omega_n = \CC[x_1, \dots, x_n] \otimes \wedge \{ \theta_1, \dots ,\theta_n \}
\end{equation}
into a symmetric algebra tensor an exterior algebra. 

The term ``superspace" comes from supersymmetry in physics. In this context the $x_i$  
correspond to bosons (so that $x_i^2$ represents two indistinguishable bosons in state $i$)
and the $\theta_i$ correspond to fermions (so that $\theta_i^2 = 0$ reflects the Pauli Exclusion Principle:
two fermions cannot occupy state $i$ at the same time).
We may also view $\Omega_n$ as the ring of holomorphic differential forms on affine $n$-space.

The symmetric group $S_n$ acts diagonally on $\Omega_n$
\begin{equation}
w \cdot x_i := x_{w(i)} \quad \quad  w \cdot \theta_i := \theta_{w(i)} \quad \quad (w \in S_n, \, \, 1 \leq i \leq n).
\end{equation}
and members of the invariant ring $\Omega_n^{S_n}$ are {\em superspace symmetric functions}.
Bases of $\Omega_n^{S_n}$ correspond to pairs $(\lambda, \mu)$ of partitions where $\mu$ has distinct parts
and $\ell(\lambda) + \ell(\mu) \leq n$.
Blondeau-Fournier, Desrosiers, Lapointe, and Mathieu \cite{BDLM} extended Macdonald polynomials
$\widetilde{H}_{\mu}(\xx;q,t)$ to the setting 
 $\Omega_n^{S_n}$ of superspace symmetric functions.

Let $(\Omega_n^{S_n})_+$ be the superspace symmetric functions with vanishing constant term
and let $I \subseteq \Omega_n$ be the (two-sided) ideal in $\Omega_n$ generated by 
$(\Omega_n^{S_n})_+$.  The {\em superspace coinvariant algebra} is the quotient
\begin{equation}
SR_n := \Omega_n/I.
\end{equation}
Like the diagonal coinvariants $DR_n$, the ring $SR_n$ is a bigraded $S_n$-module.
The Fields Institute Combinatorics Group (see \cite{ZabrockiDelta}) posed a conjecture relating
 $SR_n$ to delta operators.

\begin{conjecture}
\label{fields-conjecture}
{\em (Fields Conjecture)}
The bigraded Frobenius image of $SR_n$ is 
\begin{equation}
\grFrob(SR_n;q,z) = \sum_{k = 1}^n z^{n-k} \cdot \Delta'_{e_{k-1}} e_n \mid_{t = 0} 
\end{equation}
where $q$ tracks commuting degree and $z$ tracks anticommuting degree.
\end{conjecture}

Equivalently, the Fields Conjecture predicts that $SR_n$ is related to the quotient rings $R_{n,k}$
and the spaces $X_{n,k}$ of spanning line configurations 
via 
\begin{multline}
\grFrob(SR_n;q,z) \\ = \sum_{k = 1}^n z^{n-k} \cdot (\rev_q \circ \omega) \grFrob(R_{n,k};q)  = 
\sum_{k = 1}^n z^{n-k} \cdot (\rev_q \circ \omega) \grFrob(H^{\bullet}(X_{n,k});q).
\end{multline}
In terms of bigraded Hilbert series, this specializes to
\begin{equation}
\label{fields-bigraded-hilbert}
\Hilb(SR_n;q,z) = \sum_{k = 1}^n z^{n-k} \cdot [k]!_q \cdot \Stir_q(n,k)
\end{equation}
so that $\dim SR_n$ is the total number of ordered set partitions of $[n]$.

Despite the prominence of the ring $\Omega_n$ throughout mathematics and the natural definition
of the quotient $SR_n$, the Fields Conjecture has proven quite challenging. 
Rhoades and Wilson proved \cite{RWQuotient} the predicted formula \eqref{fields-bigraded-hilbert} for the bigraded Hilbert 
series of $SR_n$.
%The Fields Group has 
%shown (private communication) that 
%\begin{equation}
%\Hilb(SR_n;q,z) \geq \sum_{k = 1}^n z^{n-k} \cdot [k]!_q \cdot \Stir_q(n,k),
%\end{equation}
%where $f(q,z) \geq g(q,z)$ means that $f(q,z) - g(z,q)$ is a polynomial in $q,z$ with nonnegative coefficients.
Swanson and Wallach proved \cite{SW} that the alternating subspace $SR_n^{\sign} \subset SR_n$
has bigraded Hilbert series as predicted by Conjecture~\ref{fields-conjecture}, i.e.
\begin{equation}
\Hilb(SR_n^{\sign};q,z) = \sum_{k = 1}^n z^{n-k} \cdot q^{{k \choose 2}}  {n-1 \brack k-1}_q.
\end{equation}
Ideally, the quotient $SR_n$ will have some direct connection to the geometry of $X_{n,k}$, but such a link
is yet to be discovered.

The ring $R_{n,k}$, the variety $X_{n,k}$ and (conjecturally) the superspace coinvariants $SR_n$
give algebraic and geometric models for the $t = 0$ specialization of $\Delta'_{e_{k-1}} e_n$.
Zabrocki defined \cite{ZabrockiDelta} an extension of $SR_n$ which (conjecturally) gives an algebraic model
for $\Delta'_{e_{k-1}} e_n$ itself.
Let $\Omega_n[y_1, \dots, y_n]$ be the superspace ring with $n$ commuting variables $y_1, \dots, y_n$ 
adjoined. Formally, we have 
\begin{equation}
\Omega_n[y_1, \dots, y_n] = \CC[x_1, \dots, x_n] \otimes \CC[y_1, \dots, y_n] \otimes \wedge 
\{ \theta_1, \dots, \theta_n\}.
\end{equation}
The group $S_n$ acts `triply diagonally' on this space. Modding out by the $S_n$-invariants with 
vanishing constant term gives the {\em superspace diagonal coinvariants}
\begin{equation}
SDR_n := \Omega_n[y_1, \dots, y_n] / \langle \Omega_n[y_1, \dots, y_n]^{S_n}_+ \rangle
\end{equation}
which carry a triply graded action of $S_n$.

\begin{conjecture} \cite{ZabrockiDelta}
\label{zabrocki-conjecture}
{\em (Zabrocki Conjecture)}
The triply graded Frobenius image of $SDR_n$ is given by
\begin{equation}
\grFrob(SDR_n; q, t, z) = \sum_{k = 1}^n z^{n-k} \cdot \Delta'_{e_{k-1}} e_n
\end{equation}
where $q$ tracks $x$-degree, $t$ tracks $y$-degree, and $z$ tracks $\theta$-degree.
\end{conjecture}

Setting the $y$-variables to zero in the Zabrocki Conjecture recovers the Fields Conjecture~\ref{fields-conjecture}.
Setting the $\theta$-variables to zero recovers Haiman's expression
$\grFrob(DR_n; q,t) = \nabla e_n$ for the bigraded Frobenius image of the diagonal coinvariant ring 
\cite{HaimanVanish}.
As the only known proof of $\grFrob(DR_n;q,t) = \nabla e_n$ involves isospectral Hilbert schemes,
a proof of the Zabrocki Conjecture is likely to involve substantial geometric innovation.

\subsection{Superspace Vandermondes and Harmonics} While the coinvariant quotients $SR_n$ and $SDR_n$
have so far resisted analysis, there is a related family of modules defined using superspace which 
is better understood. To motivate these modules, we  recall an alternative perspective
on quotients of $\CC[x_1, \dots, x_n]$.

Given $f = f(x_1, \dots, x_n) \in \CC[x_1, \dots, x_n]$, let $\partial f$ be the differential operator obtained by replacing 
each variable $x_i$ with the corresponding partial derivative $\partial/\partial x_i$:
\begin{equation}
\partial f := f(\partial/\partial x_1, \dots, \partial/\partial x_n).
\end{equation}
The operators $\partial f$ give rise to a bilinear pairing $\langle - , - \rangle$ on $\CC[x_1, \dots, x_n]$
defined on monomials by
\begin{equation}
\langle m, m' \rangle := \text{constant term of $(\partial m) m'$}.
\end{equation}
If $I \subset \CC[x_1, \dots, x_n]$ is a homogeneous ideal, we have a direct sum decomposition
$\CC[x_1, \dots, x_n] = I \oplus I^{\perp}$ where
\begin{equation}
I^{\perp} = \{ f \in \CC[x_1, \dots, x_n] \,:\, \langle f, g \rangle = 0 \text{ for all $g \in I$} \}
\end{equation}
and the composite map
\begin{equation}
I^{\perp} \hookrightarrow \CC[x_1, \dots, x_n] \twoheadrightarrow \CC[x_1, \dots, x_n]/I
\end{equation}
is an isomorphism of graded vector spaces. The space $I^{\perp}$ is the {\em harmonic space} (or {\em Macaulay inverse system})
of the quotient $\CC[x_1, \dots, x_n]/I$. Although $I^{\perp}$ is almost never closed under multiplication, it allows
for the study of $\CC[x_1, \dots, x_n]/I$ without the use of cosets.

The harmonic space of the classical coinvariant ring $R_n$ has a nice description.
The {\em Vandermonde determinant} $\delta_n \in \CC[x_1, \dots, x_n]$ is given by
\begin{equation}
\delta_n := \varepsilon_n \cdot (x_1^{n-1} x_2^{n-2} \cdots x_{n-1}^1 x_n^0)
\end{equation}
where $\varepsilon_n := \sum_{w \in S_n} \sign(w) \cdot w \in \CC[S_n]$ is the antisymmetrizing
element of the symmetric group algebra.
The harmonic space $V_n := I_n^{\perp}$ of the coinvariant quotient
$R_n = \CC[x_1, \dots, x_n]/I_n$ is the smallest linear subspace of $\CC[x_1, \dots, x_n]$
containing $\delta_n$ which is closed under the operators $\partial/\partial x_i$  for all $1 \leq i \leq n$.
The harmonic spaces of the generalized coinvariant rings $R_{n,k}, R_{\lambda},$ and $R_{n,\lambda,s}$
also admit descriptions in terms of Vandermondes \cite{RYZ}.

Harmonic theory extends to the superspace ring $\Omega_n$.  
For $1 \leq i \leq n$, let $\partial/\partial \theta_i: \Omega_n \rightarrow \Omega_n$ by the 
$\CC[x_1, \dots, x_n]$-linear operator which acts on $\theta$-monomials by
\begin{equation}
\partial/\partial \theta_i: \theta_{j_1} \cdots \theta_{j_r} = \begin{cases}
(-1)^{s-1} \theta_{j_1} \cdots \widehat{\theta_{j_s}} \cdots \theta_{j_r} & \text{if $i = j_s$} \\
0 & \text{if $i \neq j_1, \dots, j_r$}
\end{cases}
\end{equation}
for all distinct indices $1 \leq j_1, \dots, j_r \leq n$ where $\widehat{\theta_j}$ denotes omission.
 The map $\partial/\partial \theta_i$ 
is  called a {\em contraction} operator.
The $\partial/\partial x_i$ and $\partial/\partial \theta_i$ satisfy the same relations 
as the superspace generators: we have
\begin{equation}
\partial/\partial x_i \partial/\partial x_j = \partial/\partial x_j \partial / \partial x_i \quad
\partial/\partial x_i \partial/\partial \theta_j = \partial/\partial \theta_j \partial/\partial x_i \quad
\partial/\partial \theta_i \partial/\partial \theta_j = - \partial/\partial \theta_j \partial/\partial \theta_i
\end{equation}
for $1 \leq i , j \leq n$.  Consequently, for any
$f = f(x_1, \dots, x_n, \theta_1, \dots, \theta_n) \in \Omega_n$ we have an endomorphism
$\partial f: \Omega_n \rightarrow \Omega_n$ given by
\begin{equation}
\partial f = f(\partial/\partial x_1, \dots, \partial/\partial x_n, \partial/\partial \theta_1, \dots, \partial/\partial \theta_n).
\end{equation}

The operators $\partial f$ give rise to an inner product on $\Omega_n$ characterized by 
\begin{equation}
\langle m , m' \rangle = \text{constant term of $(\partial m)m'$}
\end{equation}
whenever $m, m' \in \Omega_n$ are superspace monomials
(expressions of the form $x_1^{a_1} \cdots x_n^{a_n} \theta_{j_1} \cdots \theta_{j_r}$).
If $I \subset \Omega_n$ is a bihomogeneous ideal, the composite
\begin{equation}
I^{\perp} \hookrightarrow \Omega_n \twoheadrightarrow \Omega_n/I
\end{equation}
is an isomorphism where 
$I^{\perp} = \{ f \in \Omega_n \,:\, \langle f, g \rangle = 0 \text{ for all $g \in I$} \}.$
As in the commutative setting, the space $I^{\perp}$ is called the {\em harmonic space}
of $\Omega_n/I$.

The harmonic space of the superspace coinvariant ring $SR_n$ may be described as follows.
For $j \geq 1$, let $d_j: \Omega_n \rightarrow \Omega_n$ be the linear operator
\begin{equation}
d_j(f) := \sum_{i = 1}^n \theta_i \cdot \frac{\partial^j f}{\partial x_i^j}.
\end{equation}
When $j = 1$, the operator $d_1$ is the usual total derivative on differential forms.
These `higher' total derivatives conjecturally generate the superspace harmonics.
The following result was conjectured by the Fields Group (personal communication) and proven by Rhoades and Wilson \cite{RWQuotient}.

\begin{theorem}
\label{superspace-harmonic-generation}
 (Rhoades-Wilson \cite{RWQuotient})
 {\em (Operator Theorem)}
The harmonic space to the superspace coinvariant ring $SR_n$ is the smallest linear subspace of $\Omega_n$ which is closed
under the operators $(\partial/\partial x_1), \dots , (\partial/\partial x_n)$ and contains the $2^{n-1}$
superspace elements
\begin{equation*}
d_1^{\epsilon_1} d_2^{\epsilon_2} \cdots d_{n-1}^{\epsilon_{n-1}} (\delta_n) \in \Omega_n
\end{equation*}
where $\delta_n \in \CC[x_1, \dots, x_n]$ is the Vandermonde determinant and
$\epsilon_1, \epsilon_2, \dots, \epsilon_{n-1} \in \{0,1\}$.
\end{theorem}

Haiman proved \cite{HaimanVanish} an analogous result to Theorem~\ref{superspace-harmonic-generation} which describes the harmonic space
of the diagonal coinvariant ring $DR_n$ in terms of polarization operators.
Theorem~\ref{superspace-harmonic-generation} gives a model for $SR_n$ involving honest superspace elements
rather than cosets, thus potentially giving a method for finding the bigraded $S_n$-structure of $SR_n$ and proving the Fields Conjecture.
Josh Swanson has made substantial progress in uncovering this module structure
of this harmonic space by relating the superspace elements $d_1^{\epsilon_1} d_2^{\epsilon_2} \cdots d_{n-1}^{\epsilon_{n-1}} (\delta_n) $ to
 Tanisaki quotients \cite{Swanson}, but much  remains to be done.

Rhoades and Wilson \cite{RWSuper} exploited the anticommuting $\theta$-variables to 
generalize the Vandermonde $\delta_n$ to $\Omega_n$ and give an
alternative to the harmonic space of Conjecture~\ref{superspace-harmonic-generation}.
For $k \leq n$, the {\em superspace Vandermonde} $\delta_{n,k} \in \Omega_n$ is given by
\begin{equation}
\delta_{n,k} := \varepsilon_n \cdot \left( x_1^{k-1} \cdots x_{n-k}^{k-1} x_{n-k+1}^{k-1} 
x_{n-k+2}^{k-2} \cdots x_{n-1}^1 x_n^0 \cdot \theta_1 \cdots \theta_{n-k} \right)
\end{equation}
so that, for example, we have
\begin{equation*}
\delta_{3,2} = \varepsilon_3 \cdot (x_1 x_2 \theta_1) = 
x_1 x_2 \theta_1 - x_1 x_2 \theta_2 - x_1 x_3 \theta_1 + x_2 x_3 \theta_2 + x_1 x_3 \theta_3 - x_2 x_3 \theta_3.
\end{equation*}
Thanks to the $\theta$-variables, the operator $\varepsilon_n$ does not annihilate the generating monomial
of $\delta_{n,k}$ and we obtain nonzero elements of $\Omega_n$.

The monomials in $\delta_{n,k}$ appear among those in 
the elements $d_1 d_2 \cdots d_{n-k} (\delta_n)$ of Conjecture~\ref{superspace-harmonic-generation},
but $\delta_{n,k} \neq d_1 d_2 \cdots d_{n-k}(\delta_n)$ as elements of $\Omega_n$.
In keeping with the harmonic paradigm, we use the $\delta_{n,k}$ as seeds to grow modules.

\begin{defn}
Let $V_{n,k}$ be the smallest linear subspace of $\Omega_n$ which 
contains $\delta_{n,k}$ and is closed under the operators $\partial/\partial x_i$ for $1 \leq i \leq n$.
\end{defn}

The spaces $V_{n,k}$ are concentrated in $\theta$-degree $n-k$. Considering $x$-degree alone
gives $V_{n,k}$ the structure of a singly graded $S_n$-module. Rhoades and Wilson proved \cite{RWSuper}
that
\begin{equation}
\grFrob(V_{n,k};q) = \Delta'_{e_{k-1}} e_n \mid_{t = 0}
\end{equation}
and conjectured that the composite
\begin{equation}
\label{bridge-composite-single}
\bigoplus_{k = 1}^n V_{n,k} \hookrightarrow \Omega_n \twoheadrightarrow SR_n
\end{equation}
is an isomorphism. If \eqref{bridge-composite-single} is bijective, the Fields Conjecture would follow.
Superspace Vandermondes can also be used to define modules inside $\Omega_n[y_1, \dots, y_n]$.

\begin{defn}
\label{double-v-definition}
Let $\VV_{n,k}$ be the smallest linear subspace of $\Omega_n[y_1, \dots, y_n]$ which
\begin{itemize}
\item contains the superspace Vandermonde $\delta_{n,k}$ (in the $x$-variables and $\theta$-variables alone),
\item is closed under the differentiation operators $\partial/\partial x_i$ and $\partial/\partial y_i$ for $1 \leq i \leq n$,
and
\item if closed under the {\em polarization operator} 
$y_1 \partial^j/\partial x_1^j + \cdots + y_n \partial^j/\partial x_n^j$ for each $j \geq 1$.
\end{itemize}
\end{defn}

The polarization operators  in Definition~\ref{double-v-definition} lower $x$-degree by $j$
and raise $y$-degree by 1.  
The space $\VV_{n,k}$ is closed under the action of $S_n$ and concentrated in 
$\theta$-degree $n-k$.  By considering $x$-degree and $y$-degree, the space $\VV_{n,k}$
is a doubly graded $S_n$-module.  Its bigraded $S_n$-structure is conjecturally governed by delta operators.

\begin{conjecture} (Rhoades-Wilson \cite{RWSuper})
\label{double-v-conjecture}
We have $\grFrob(\VV_{n,k};q,t) = \Delta'_{e_{k-1}} e_n$.  Furthermore, the composite map
\begin{equation*}
\bigoplus_{k = 1}^n \VV_{n,k} \hookrightarrow \Omega_n[y_1, \dots, y_n] \twoheadrightarrow SDR_n
\end{equation*}
is bijective.
\end{conjecture}

Conjecture~\ref{double-v-conjecture} implies Zabrocki's Conjecture~\ref{zabrocki-conjecture}.
Since the first assertion of Conjecture~\ref{double-v-conjecture} is known at $t = 0$ but the corresponding
coinvariant statement (the Fields Conjecture~\ref{fields-conjecture})
 remains open, it appears that Vandermondes could give an 
easier road to delta operator modules than coinvariants.

\subsection{The Theta and Boson-Fermion coinvariant conjectures}
The modules in the Fields and Zabrocki Conjectures suggest a natural generalization.
Given integers $n, r, \ell \geq 0$, consider an $r \times n$ matrix $X_{r \times n} = (x_{i,j})$ of 
commuting variables and an $\ell \times n$ matrix $\Theta_{\ell \times n} =  (\theta_{i,j})$ of anticommuting variables
and let
\begin{equation}
\SSS(n;r, \ell) := \CC[X_{r \times n}] \otimes \wedge \{ \Theta_{\ell \times n} \}
\end{equation}
be the $\CC$-algebra generated by these variables. This is a symmetric algebra of rank $r \times n$ 
tensor an exterior algebra of rank $\ell \times n$.
The symmetric group $S_n$ acts on the columns of the matrices $X_{r \times n}$ and $\Theta_{\ell \times n}$.
We let $I \subset \SSS(n;r,\ell)$ be the ideal generated by $S_n$-invariants with vanishing constant terms and
write
\begin{equation}
\RRR(n;r,\ell) := \SSS(n;r,\ell)/I
\end{equation}
for the corresponding quotient ring. The ring $\RRR(n;r,\ell)$ is multigraded, with $r$ flavors of commuting
graded and $\ell$ flavors of anticommuting grading.
We have met a number of special cases of $\RRR(n;r,\ell)$.
\begin{itemize}
\item When $r = 1$ and $\ell = 0$ we have the classical coinvariant ring $\RRR(n;1,0) = R_n$.
\item When $r = 2$ and $\ell = 0$ we have the diagonal coinvariants $\RRR(n;2,0) = DR_n$.
\item When $r = \ell = 1$ we have the superspace coinvariants $\RRR(n;1,1) = SR_n$.
\item When $r = 2$ and $\ell = 1$ we have the module 
$\RRR(n;2,1) = SDR_n$.
\end{itemize}

The ring $\RRR(n;r,\ell)$ has the structure of a multigraded $S_n$-module. 
D'Adderio, Iraci, and Vanden Wyngaerd \cite{DIW} introduced the {\em theta operators} on the ring of symmetric 
functions and used them to conjecturally describe (part of) the modules $\RRR(n;2,\ell)$ in the case
of $r = 2$ sets of commuting variables.
We define their operators and state their conjecture.

Let $\Pi: \Lambda \rightarrow \Lambda$ be the Macdonald eigenoperator defined by
\begin{equation}
\Pi: \widetilde{H}_{\mu}(\xx;q,t) \mapsto \prod_{(i,j)} (1 - q^{i-1} t^{j-1}) \cdot \widetilde{H}_{\mu}(\xx;q,t)
\end{equation}
where the product ranges over the coordinates $(i,j) \neq (1,1)$ of cells in the Young diagram of $\mu$.
For example, if $\mu = (3,2)$ we fill the Young diagram as in
\begin{equation*}
\begin{Young}
 \cdot & q & q^2 \cr t & qt
\end{Young}
\end{equation*}
so that 
\begin{equation*}
\Pi:  \widetilde{H}_{(3,2)}(\xx;q,t) \mapsto (1-q)(1-q^2)(1-t)(1-qt) \cdot \widetilde{H}_{(3,2)}(\xx;q,t).
\end{equation*}
Omitting the cell $(1,1)$ in the eigenvalue of $\widetilde{H}_{\mu}(\xx;q,t)$ ensures that the operator
 $\Pi: \Lambda \rightarrow \Lambda$ is invertible.
 
 Given a symmetric function $F \in \Lambda$, the {\em theta operator}\footnote{not to be confused with
 the anticommuting $\theta$-variables!} 
 $\Theta_F: \Lambda \rightarrow \Lambda$
 is defined by
 \begin{equation}
 \Theta_F: G \mapsto \Pi \cdot  F^* \cdot \Pi^{-1} \cdot G
 \end{equation}
 for any symmetric function $G$. Here $F^*$ denotes multiplication by the plethystically transformed
 version $F \left[ \frac{\xx}{(1-q)(1-t)} \right]$ of the symmetric function $F$ labeling $\Theta_F$.

 The theta operators were introduced in \cite{DIW} to give a compositional refinement of the
 symmetric function side $\Delta'_{e_{k-1}} e_n$ of the Delta Conjecture involving the $\nabla$ operator.
 This refinement was the basis of the D'Adderio-Mellit proof \cite{DM} of the Rise Version of the Delta Conjecture.
 Theta operators are also expected to have ties to coinvariant theory.

 \begin{conjecture}
 \label{theta-conjecture}
 (D'Adderio-Iraci-Vanden Wyngaerd \cite{DIW})  Let $\jj = (j_1, j_2)$ be a pair of nonnegative integers
 with sum $j_1 + j_2 < n$. The piece $R(n;2,2)_{\jj}$ of $\theta$-bidegree $\jj$
 is, by considering commuting degrees, a bigraded $S_n$-module.  
 We have
 \begin{equation*}
 \grFrob( \RRR(n;2,\ell)_{\jj}; q, t) = \Theta_{e_{j_1}}  \Theta_{e_{j_{2}}} \nabla e_{n- (j_1 + j_2)}.
 \end{equation*} 
 \end{conjecture}

We have (see \cite{IRR, KR}) that $\RRR(n;2,2)_{\jj} = 0$
 whenever $j_1 + j_2  \geq n$
so that Conjecture~\ref{theta-conjecture} gives a complete description of the multigraded 
$S_n$-structure of $\RRR(n;2,2)$.  
However, in the case of $\ell > 2$ fermionic sets of variables there are nonzero pieces 
$\RRR(n;2,\ell)_{\jj}$ for which $j_1 + \cdots + j_{\ell} \geq n$.
D'Adderio et. al. show that Conjecture~\ref{theta-conjecture} implies the Zabrocki Conjecture 
when $\ell = 1$.

The $\ell = 2$ case of Conjecture~\ref{theta-conjecture} is especially intriguing due to the symmetry of
the ring $\RRR(n;2,2)$.
The quotient of $\RRR(n;2,2)$ obtained by setting the anticommuting variables to zero is the diagonal coinvariant ring $DR_n$;
the corresponding specialization of
Conjecture~\ref{theta-conjecture}  is equivalent to Haiman's \cite{HaimanVanish} formula 
 $\grFrob(DR_n;q,t) = \nabla e_n$.
 The `fermionic diagonal coinvariant ring' $FDR_n$ obtained from $\RRR(n;2,2)$
 obtained by setting the {\em commuting} variables to zero was
introduced and studied by Kim and Rhoades \cite{KR}.  In terms of combinatorics,
 the ring $FDR_n$ encodes a kind of crossing resolution in set partitions
of $[n]$; see \cite{KR2}.
Iraci, Rhoades, and Romero \cite{IRR}  proved the case of Conjecture~\ref{theta-conjecture} involving $FDR_n$.

Returning to the general setting of $\RRR(n;r,\ell)$, F. Bergeron has a remarkable conjecture relating these 
modules when $r \rightarrow \infty$ or $\ell \rightarrow \infty$.
In order to state his conjecture, we need to consider  additional structure on $\RRR(n;r,\ell)$.

The action of the general linear groups $GL_r(\CC)$ and $GL_{\ell}(\CC)$ on the rows of 
$X_{r \times n}$ and $\Theta_{\ell \times n}$ gives $\RRR(n;r,\ell)$ the  structure
of a $GL_r(\CC) \times GL_{\ell}(\CC)$-module.  Since the actions of $S_n$ and
$GL_r(\CC) \times GL_{\ell}(\CC)$ commute, 
we have an action of the product group
\begin{equation}
\GGG(n;r, \ell) := S_n \times GL_r(\CC) \times GL_{\ell}(\CC)
\end{equation}
on $\RRR(n;r,\ell)$.

We encode $\GGG(n;r,\ell)$-structure of $\RRR(n;r,\ell)$ as a formal power series.
More precisely, we define the {\em character} of $\RRR(n;r,\ell)$ by
\begin{equation}
\ch \, \RRR(n;r,\ell) := \sum_{\mathbf{i}, \mathbf{j}}  \Frob( \RRR(n;r,\ell)_{\mathbf{i}, \mathbf{j}} \cdot
q_1^{i_1} \cdots q_r^{i_r} \cdot z_1^{j_1} \cdots z_{\ell}^{j_{\ell}}
\end{equation}
where the sum is over all $r$-tuples $\mathbf{i} = (i_1, \dots, i_{\ell})$ and $\ell$-tuples
$\jj = (j_1, \dots ,j_{\ell})$ of nonnegative integers
and $\RRR(n;r,\ell)_{\ii,\jj}$ is the piece of $\RRR(n;r,\ell)$ with commuting multidegree $\ii$
and anticommuting multidegree $\jj$. By the structure of polynomial representations of  
general linear groups, we have an expansion
\begin{equation}
\ch \, \RRR(n;r,\ell) = \sum_{\lambda, \mu, \nu} d_{\lambda, \mu, \nu} \cdot
s_{\lambda}(\xx) \cdot s_{\mu}(q_1, \dots, q_r) \cdot
s_{\nu}(z_1, \dots, z_{\ell})
\end{equation}
where $\lambda$ ranges over partitions of $n$, $\mu$ and $\nu$ range over arbitrary partitions, 
and the $d_{\lambda,\mu,\nu}$ are nonnegative integers.

Temporarily consider the rings $\RRR(n;r,0)$ involving commuting variables alone.
F. Bergeron proved \cite{BergeronMulti} that the limit 
\begin{equation}
\EEE_n(\xx, \mathbf{q}) := \lim_{r \rightarrow \infty} \ch \, \RRR(n;r,0)
\end{equation}
converges in the ring of formal power series where $\mathbf{q} = (q_1, q_2, \dots )$ is an infinite
alphabet tracking commuting multidegrees. We may write
\begin{equation}
\label{b-coefficient}
\EEE_n(\xx, \mathbf{q}) = \sum_{\lambda, \mu} b_{\lambda, \mu} \cdot s_{\lambda}(\xx) \cdot s_{\mu}(\mathbf{q})
\end{equation}
for nonnegative integers $b_{\lambda, \mu}$.

Although $\EEE_n$ does not explicitly
 involve  anticommuting information, F. Bergeron conjectured \cite{Bergeron} that all of the characters
  $\ch \, \RRR(n;r,\ell)$ are determined 
by $\EEE_n$ alone.
 Let $\zz = (z_1, z_2, \dots )$ be an infinite alphabet of anticommutative tracking variables.
 
 \begin{conjecture}
 \label{supersymmetry-conjecture}
 {\em (Combinatorial Supersymmetry \cite{Bergeron})} For any $r, \ell \geq 0$, the character 
 $\ch \, \RRR(n;r,\ell)$ may be obtained from the `universal' symmetric function
 \begin{equation}
 \label{universal-formula}
\sum_{\lambda, \mu, \nu, \rho} b_{\lambda, \mu} \cdot c_{\nu, \rho}^{\mu} \cdot s_{\lambda}(\xx) \cdot
 s_{\nu}(\mathbf{q}) \cdot s_{\rho'}(\zz)
 \end{equation}
 by evaluating $q_i = 0$ for $i > r$ and $z_j = 0$ for $j > \ell$.
 Here the 
 $b_{\lambda, \mu}$ are as in the expression \eqref{b-coefficient}
for  $\EEE_n(\xx, \mathbf{q})$ and $c_{\nu,\rho}^{\mu}$ is a Littlewood-Richardson coefficient.
 \end{conjecture}
 
 Notice that the partition $\rho$ is unconjugated in the Littlewood-Richardson coefficient
 $c_{\nu,\rho}^{\mu}$ in \eqref{universal-formula} but is conjugated in the Schur 
 function $s_{\rho'}(\zz)$.  The polynomial \eqref{universal-formula} may be expressed more succinctly
 in plethystic notation as 
 $\EEE_n(\xx, \mathbf{q} - \varepsilon \zz) 
 = \sum_{\lambda, \mu} b_{\lambda, \mu} \cdot s_{\lambda}(\xx) \cdot s_{\mu}[\mathbf{q} - \varepsilon \cdot \zz]$.

 \begin{remark}
 \label{fermionic-supersymmetry-remark}
 As stated, Conjecture~\ref{supersymmetry-conjecture} implies that the characters 
 $\ch \, \RRR(n;r,\ell)$ of the general modules $\RRR(n;r,\ell)$ are determined by those of 
 the `purely commuting' modules $\RRR(n;r,0)$.  However, there is an equivalent formulation
 involving the `purely anticommuting' modules $\RRR(n;0,\ell)$. 
More precisely, let 
 \begin{equation}
 \FFF_n(\xx, \zz) := \lim_{\ell \rightarrow \infty} \ch \, \RRR(n;0,\ell)
 \end{equation}
 be the limit of the `purely anticommuting characters'.  As in the commuting case,
this limit exists and there are nonnegative integers $d_{\lambda, \mu}$ such that
\begin{equation}
\FFF_n(\xx, \zz) = \sum_{\lambda, \mu} d_{\lambda, \mu} \cdot s_{\lambda}(\xx) \cdot s_{\mu}(\zz).
\end{equation}
Motivated by the relationship between \eqref{universal-formula} and $\EEE_n(\xx, \mathbf{q})$, we 
consider the symmetric
function
\begin{equation}
\label{fermionic-universal-formula}
\sum_{\lambda, \mu, \nu, \rho} d_{\lambda, \mu} \cdot c_{\nu, \rho}^{\mu} \cdot
s_{\lambda}(\xx) \cdot s_{\nu'}(\mathbf{q}) \cdot s_{\rho}(\zz)
\end{equation}
where the partition $\nu$ is unconjugated in $c_{\nu,\rho}^{\mu}$ but conjugated in $s_{\nu'}(\mathbf{q})$.
Conjecture~\ref{supersymmetry-conjecture} is equivalent to the assertion that, for any $r, \ell \geq 0$,
the symmetric function $\ch \, \RRR(n;r,\ell)$ is obtained from  
\eqref{fermionic-universal-formula} by evaluating $q_i = 0$ for $i > r$ and $z_j = 0$ for $j > \ell$.
\end{remark}

At a high level,
Remark~\ref{fermionic-supersymmetry-remark} says that the anticommuting modules $\RRR(n;0,\ell)$ should
`contain the same information as' the commuting modules $\RRR(n;r,0)$ (and, in particular, the diagonal
coinvariant ring $DR_n = \RRR(n;2,0)$).
Thus, the modules $\RRR(n;0,\ell)$ will likely prove very 
challenging to study as $\ell \rightarrow \infty$.

 Another remarkable conjecture of F. Bergeron's states that the symmetric function $\EEE_n(\xx,\mathbf{q})$ 
 ``contains the data of" symmetric functions appearing in the Delta Conjecture.
 For any symmetric function $F$, let $F^{\perp}: \Lambda \rightarrow \Lambda$ be the adjoint operation
 to multiplication by $F$ under the (nondegenerate) Hall inner product. The map $F^{\perp}$ is characterized 
 by 
 \begin{equation}
 \langle F^{\perp} G, H \rangle = \langle G, FH \rangle \quad \quad \text{for all $G, H \in \Lambda$.}
 \end{equation}

 \begin{conjecture}
 \label{skewing-conjecture}
 {\em (Skewing Conjecture \cite{BergeronMultiSkew})} 
 For any $n = a + b + 1$, we have
 \begin{equation}
 \Delta'_{e_a} e_n = \left[ e_b(\mathbf{q})^{\perp} \EEE_n(\xx,\mathbf{q})  \right]_{\mathbf{q} \rightarrow (q,t,0,0,0, \dots )}
 \end{equation}
 where the skewing operator $e_b(\mathbf{q})^{\perp}$ acts on the grading $\mathbf{q}$-variables alone.
 \end{conjecture}

 At a combinatorial level, the action of $e_b(\mathbf{q})^{\perp}$ on a Schur function $s_\lambda(\mathbf{q})$ in the 
 $\mathbf{q}$-variables is a sum over Schur functions $s_\mu(\mathbf{q})$
 such that $\lambda/\mu$ is a skew diagram of size $b$ in which every row has at most one box.
 Algebraically, the action of 
 $e_b(\mathbf{q})^{\perp}$ on a $GL_\infty(\CC)$-module corresponds to restriction to the `parabolic subgroup'
 $GL_b(\CC) \times GL_\infty(\CC)$, and then considering the $(\{ \mathrm{Id.} \} \times GL_\infty(\CC))$-action on the 
 $(1^b)$-weight space of $GL_b(\CC)$.
 Even though the formula in 
 Conjecture~\ref{skewing-conjecture} relates symmetric functions in the two parameters $q, t$ alone, 
 more parameters $q_1, q_2, \dots $ are necessary before specialization for the 
 formula to hold.

\end{document}